\documentclass[a4paper,11pt]{article}

\usepackage{amssymb}
\usepackage{graphicx}
\usepackage{tikz}
\usetikzlibrary{intersections}
\usetikzlibrary{svg.path}
\usetikzlibrary{decorations.pathmorphing}

\usepackage{pgfplots}
\pgfplotsset{compat=1.11}
\usepackage{mathrsfs}
\usepackage{amsmath}
\usepackage{amsthm}
\usepackage{enumerate}
\usepackage{color}
\usepackage{verbatim}
\usepackage{todonotes}

\usepackage[margin=1in, includefoot, footskip=30pt]{geometry}

\usepackage{array}
\newcolumntype{e}{>{\displaystyle}r @{\,} >{\displaystyle}c @{\,} >{\displaystyle}l}

\usepackage{amssymb,amsopn,amscd}

\usepackage[colorlinks]{hyperref}
\usepackage{xcolor}
\hypersetup{
    linkcolor={blue!40!black},
    citecolor={blue!40!black},
    urlcolor={blue!80!black}
}
\usepackage{nameref}

\theoremstyle{plain}
\newtheorem{Theorem}{Theorem}[section]
\newtheorem{Corollary}[Theorem]{Corollary}
\newtheorem{Lemma}[Theorem]{Lemma}
\newtheorem{Proposition}[Theorem]{Proposition}

\newtheorem{Question}[Theorem]{Question}
\newtheorem{Claim}[Theorem]{Claim}

\theoremstyle{definition}

\newtheorem{Definition}[Theorem]{Definition}
\newtheorem{Remark}[Theorem]{Remark}

\newcommand{\T}{\mathbb{T}}

\newcommand{\Z}{\mathbb{Z}}
\newcommand{\s}{\mathbf{s}}
\newcommand{\w}{\mathbf{w}}

\newcommand{\B}{\mathbf{B}}
\newcommand{\ow}{\overline{\mathbf{w}}}

\newcommand{\gpath}{\mathbf{\Xi}}

\newcommand{\cyl}{\mathscr{C}}

\newcommand{\Poi}{\operatorname{Poi}}
\newcommand{\ent}{\mathbf{b}}
\renewcommand{\o}{\mathbf{a}}
\newcommand{\Fo}{\operatorname{FPP}_1}
\newcommand{\Tol}{\overline{\operatorname{T}}_{1}}
\newcommand{\Aol}{\overline{\operatorname{A}}}
\newcommand{\Fl}{\operatorname{FPP}_\lambda}
\newcommand{\cin}{\mathbf{c}_{\operatorname{in}}}
\newcommand{\cout}{\mathbf{c}_{\operatorname{out}}}
\renewcommand{\P}{\mathbb{P}}

\renewcommand{\r}{\mathbf{r}}
\renewcommand{\u}{\mathbf{u}}

\renewcommand{\j}{\mathbf{j}}
\renewcommand{\i}{\mathbf{i}}

\numberwithin{equation}{section}

  \newcounter{constant}
  \newcommand{\nc}[1]{\refstepcounter{constant}\label{#1}}
  \newcommand{\uc}[1]{c_{\textnormal{\tiny \ref{#1}}}}
\newcounter{radius}
  \newcommand{\nr}[1]{\refstepcounter{radius}\label{#1}}
  \newcommand{\ur}[1]{\mathbf{r}_{\textnormal{\tiny \ref{#1}}}}

\begin{document}

\title{Coexistence of competing first passage percolation on hyperbolic graphs}

\author{Elisabetta Candellero\footnote{ecandellero@mat.uniroma3.it;  Universit\`a Roma Tre, Dipartimento di Matematica e Fisica, Largo S.\ Murialdo 1, 00146, Rome, Italy.} \and Alexandre Stauffer\footnote{a.stauffer@bath.ac.uk;  University of Bath, Dept of Mathematical Sciences, BA2 7AY Bath, UK.}}

\maketitle

\begin{abstract}
   We study a natural growth process with competition, which was recently introduced to analyze MDLA, a challenging model for the growth of an aggregate by diffusing particles. 
   The growth process consists of two first-passage percolation processes $\Fo$ and $\Fl$, spreading with rates $1$ and $\lambda>0$ respectively, on a graph $G$.
   $\Fo$ starts from a single vertex at the origin $o$, while the initial configuration of $\Fl$ consists of infinitely many \emph{seeds} distributed according to a product of 
   Bernoulli measures of parameter $\mu>0$ on $V(G)\setminus \{o\}$.
   $\Fo$ starts spreading from time 0, while each seed of $\Fl$ only starts spreading after it has been reached by either $\Fo$ or $\Fl$.
   A fundamental question in this model, and in growth processes with competition in general, is whether the two processes coexist (i.e., both produce infinite clusters) with positive probability.
   We show that this is the case when $G$ is vertex transitive, non-amenable and hyperbolic, in particular, for any $\lambda>0$ there is a $\mu_0=\mu_0(G,\lambda)>0$ such that for all $\mu\in(0,\mu_0)$ 
   the two processes coexist with positive probability. 
   This is the first non-trivial instance where coexistence is established for this model. 
   We also show that $\Fl$ produces an infinite cluster almost surely for any positive $\lambda,\mu$, establishing fundamental differences with the behavior of such processes on $\Z^d$.
\end{abstract}


\section{Introduction}

%

%
%
 

%

We consider a randomly growing process with competition, which was introduced in~\cite{Stauffer-Sidoravicius-MDLA} under the name of \emph{first passage percolation in a hostile environment} (FPPHE). 
FPPHE consists of two first passage percolation processes, denoted $\Fo$ and $\Fl$, which spread inside an infinite graph $G$.
At time $0$, $\Fo$ occupies only a single vertex of $G$, called \emph{the origin} $o$, 
whereas $\Fl$ starts from countably many ``sources'', which we call \emph{seeds} and are distributed according to a product of Bernoulli measures of parameter $\mu\in (0,1)$ on $V(G)\setminus\{o\}$.

The process evolves from time $0$ as follows. 
$\Fo$ starts spreading through the edges of $G$ at rate $1$. 
On the other hand, $\Fl$ does not start spreading from time $0$, but waits.  
Whenever a process (either $\Fo$ or $\Fl$) attempts to occupy a vertex that hosts a seed of $\Fl$, the attempt fails, the 
seed is \emph{activated}, and $\Fl$ starts spreading from that seed 
through the edges of $G$ at rate $\lambda>0$. Seeds that have not been activated remain dormant until they are activated.
A vertex that is occupied by either of the processes will remain so forever, and will never be occupied by the other process; hence the two processes compete for space as they grow.

We say that $\Fo$ (resp.\ $\Fl$) \emph{survives} if in the limit as time goes to infinity we obtain that $\Fo $ (resp.\ $\Fl$) occupies an infinite \emph{connected} region of the graph; otherwise we say that it 
\emph{dies out}. 
We stress that the region must be connected for the definition of survival, since almost surely $\Fl$ already starts from an infinite set of seeds. In particular, when we say that $\Fl$ ``dies out'', 
we mean that it ends up consisting of an infinite collection of 
connected regions, each of which almost surely of finite size. 

FPPHE was introduced in~\cite{Stauffer-Sidoravicius-MDLA} in order to study MDLA, a classical aggregation process showing dendritic growth. 
Nonetheless, FPPHE is a very interesting process in its own right. In particular, simulations show that FPPHE has a very rich behavior, producing a delicate geometry similar to dendritic growth and 
undergoing phase transitions.
In Figure~\ref{fig:fpphe} one can see how, by increasing slightly the density of seeds, the set of vertices occupied by $\Fo$ becomes extremely thin.
This process is a modified version of the classical FPP model and the Richardson model for competing species with different growth rates.
Yet, a fundamental difference is that the Richardson model is a monotone process in the sense that one can easily devise a coupling to show that if one increases the set of sites occupied by one of the types, this type can only occupy a larger portion of the space. 
On the other hand, FPPHE is not monotone at all: for instance, in some graphs, increasing the initial density $\mu$ can increase the probability that $\Fo$ survives (see e.g.\ \cite{Candellero-Stauffer-NotMonotone}). 
Related phenomena for the Richardson model have been investigated in \cite{Deijfen-Haggstrom}.

\begin{figure}[tbp]
   \begin{center}
      \hspace{\stretch{1}}
      \includegraphics[width=.3\textwidth]{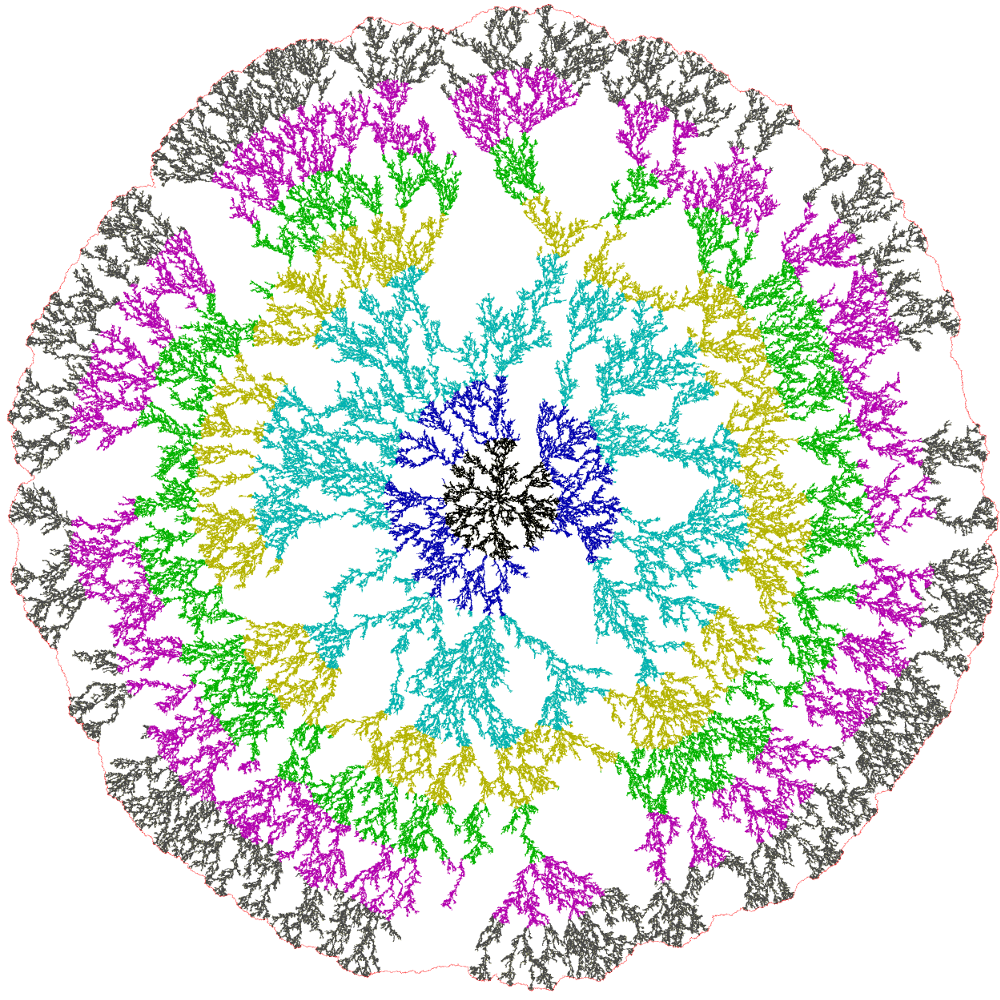}
      \hspace{\stretch{1}}
      \includegraphics[width=.3\textwidth]{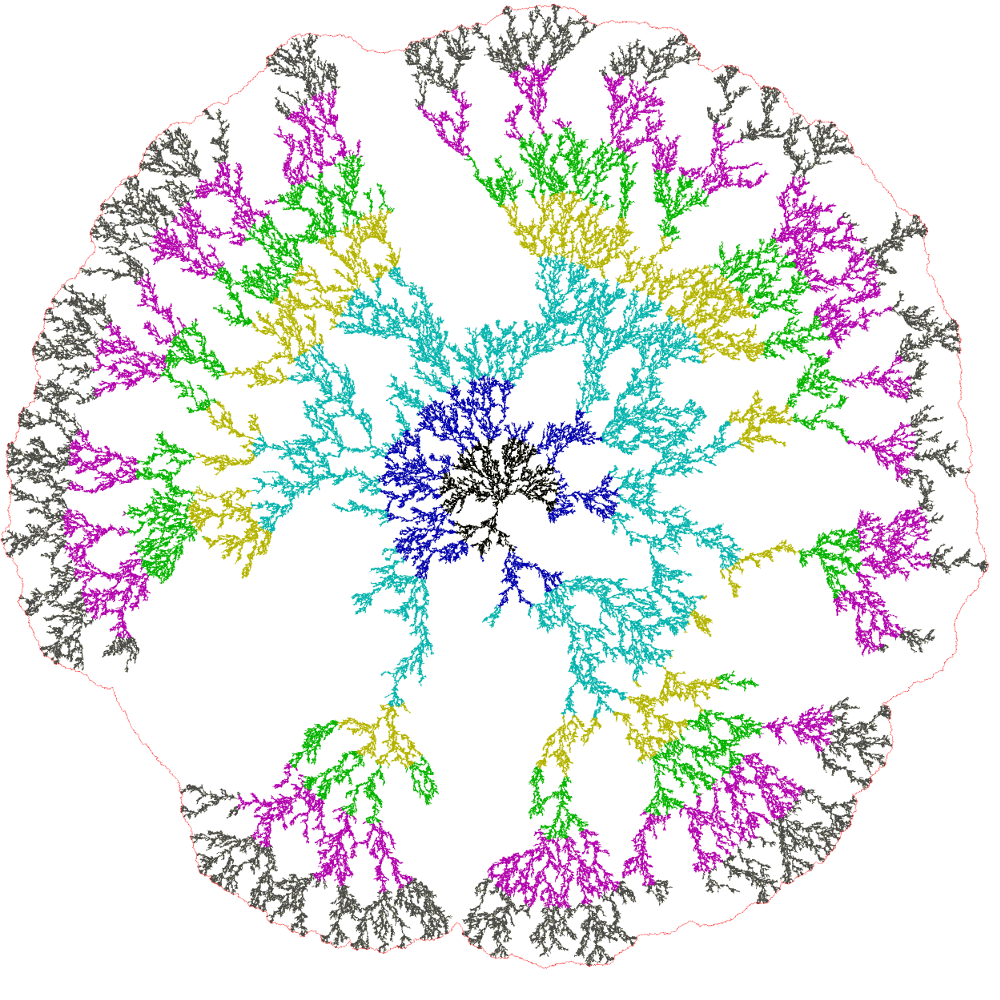}
      \hspace{\stretch{1}}
      \includegraphics[width=.3\textwidth]{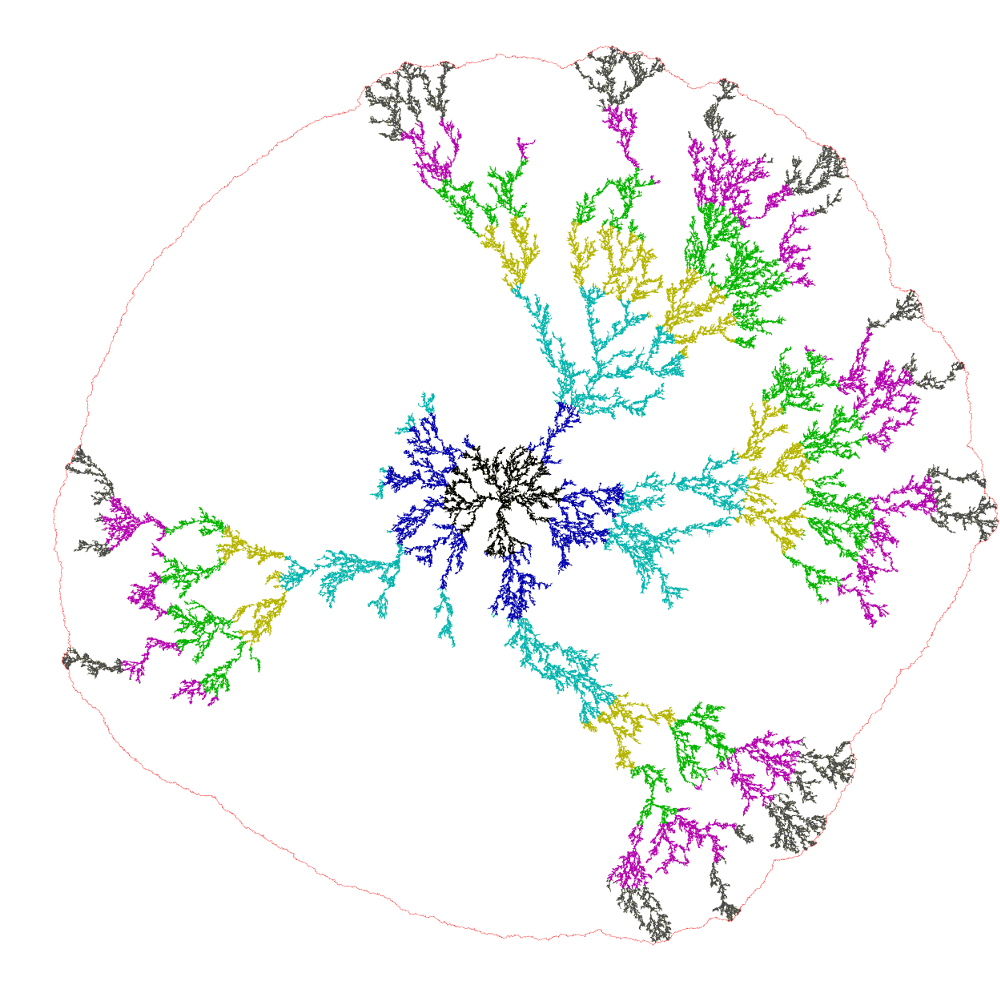}
      \hspace{\stretch{1}}
   \end{center}\vspace{-.5cm}
   \caption{A representation of FPPHE on the lattice $\Z^2$ with $\lambda=0.7$ and $\mu=0.027, 0.029$ and $0.030$, respectively. Colors represent different epochs of 
   the growth of $\Fo$, while the thin curve at the boundary represents the boundary between $\Fl$ and vertices that are either unoccupied or host an inactive seed of $\Fl$. 
   The whole white region within this boundary 
   is occupied by activated $\Fl$.}
   \label{fig:fpphe}
\end{figure}

There are three possible behaviors for FPPHE:
\begin{itemize}
   \item \emph{Extinction}, where $\Fo$ dies out \emph{almost surely}\footnote{For any $\mu>0$ and any locally finite $G$, with positive probability $\Fo$ dies out; see Remark~\ref{rem:fodies}.},

   \item \emph{Strong survival}, where  with positive probability 
$\Fo$ survives but $\Fl$ dies out, 

   \item \emph{Coexistence}, where with positive probability $\Fo$ and $\Fl$ survive simultaneously.  
\end{itemize}
The above definitions do not exclude that, for some values of $\mu$ and $\lambda$, FPPHE is both in the strong survival regime and in the coexistence regime, though it is believed that the three regimes are 
mutually exclusive. 

The extinction regime occurs, for example, when $1-\mu < p_\mathrm{c}(G)$, where  
$p_\mathrm{c}(G)$ is the critical probability for site percolation on $G$ (see~\cite{Grimmett} for a classical reference on percolation and \cite{CandelleroTeixeira, Yadin} for recent results on the non-triviality of 
$p_\mathrm{c}$ on rather general amenable graphs). In fact, when $1-\mu < p_\mathrm{c}(G)$, the set of vertices of $V(G)$ not occupied by seeds of $\Fl$ consists only of finite clusters; thus $\Fo$ is confined to the finite cluster containing the origin.
The other two regimes are much more interesting and less understood.
For example, in Figure~\ref{fig:fpphe}, the leftmost picture seems to be in a regime of strong survival, nonetheless $\Fl$ manages to conquer quite large regions, evidencing the 
long-range dependencies that appear in this process.
%
%
%
%
%

It is natural to expect that decreasing $\mu$ or $\lambda$ should favor the survival of $\Fo$. 
Unfortunately, 
there is no proof of such monotonicity. In fact, one can even engineer instances of the evolution of the process (that is, instances of passage times and locations of seeds of $\Fl$) 
such that removing a seed of $\Fl$ or slowing down the spread of $\Fl$ through a 
single edge could actually \emph{harm} the survival of $\Fo$. 
This lack of monotonicity makes this process particularly challenging to analyze. 

%
In~\cite{Stauffer-Sidoravicius-MDLA} FPPHE was analyzed with $G$ being the lattice $\mathbb{Z}^d$. 
They developed a very involved proof, based on a multi-scale analysis, to show that there exists strong survival for $d\geq 2$. 
More precisely, in~\cite[Theorem 1.2]{Stauffer-Sidoravicius-MDLA}, it is shown that for \emph{any} $\lambda<1$, there exists 
$\mu_0>0$ such that if $\mu \in (0,\mu_0)$ then there is strong survival. 
Unfortunately,~\cite{Stauffer-Sidoravicius-MDLA} gives no information regarding whether FPPHE has a coexistence regime in $\mathbb{Z}^d$, $d\geq 2$, which remains a fascinating open problem. 
In general, establishing coexistence is even more challenging than establishing strong survival. 
The main reason is that, in the strong survival regime, each seed of $\Fl$ that gets activated will produce a finite cluster.
Thus, after some finite time, $\Fo$ will go around that seed and encapsulate its cluster inside a finite region.
At that moment, the presence of that seed ceases to interfere with the spread of $\Fo$. 
On the other hand, in a coexistence regime, dependencies are even stronger (in the sense that their range is unbounded) since there will be seeds of $\Fl$ that will have an everlasting effect in the evolution of $\Fo$. 
Our main contribution is to establish for the first time (with the exception of the trivial cases below)
a regime of coexistence for FPPHE. 

When $d=1$, the strong survival regime and the coexistence regime do not exist. 
More generally, when $G$ is a tree, or even a free product (see Section~\ref{sect:questions} for a definition and a more thorough discussion), 
the regime of strong survival cannot exist since $\Fo$ is not able to ``go around and block the spread'' of an activated seed of 
$\Fl$. 
When $G$ is a tree, survival of $\Fo$ translates to the origin being in an infinite cluster of the graph obtained from $G$ by removing the vertices occupied by seeds of $\Fl$. 
Consequently, on trees, coexistence occurs 
whenever $1-\mu > p_\mathrm{c}(G)$; an analogous result can be derived for free products. 

\subsection{Our results}
We will consider the case of $G$ being \emph{hyperbolic} and \emph{non-amenable} (see Section \ref{sect:why-hyperbolic} for a rigorous definition). 
The main goals of this paper are to establish a regime of coexistence for FPPHE, and to 
show that FPPHE on such graphs has a substantially different behavior than on $\mathbb{Z}^d$.

Roughly  speaking, hyperbolicity means that the sides of any geodesic triangle are ``close'' to one another, and non-amenability means that each finite subset has a large boundary.
Hyperbolic and non-amenable graphs are classical and important classes of graphs, for example, hyperbolic graphs were introduced by Gromov \cite{Gromov} and can be seen as a discrete analogue of manifolds with negative curvature.
We will further consider natural assumptions on $G$, such as vertex transitivity and boundedness of degrees.

In the first theorem, we show that $\Fo$ survives with positive probability for any value of $\lambda$, provided that $\mu$ is small enough.
This theorem does not require $G$ to be transitive.
\begin{Theorem}[Survival of $\Fo$]\label{thm:survival_FPP_1}
   Let $G$ be a hyperbolic, non-amenable graph of bounded degree.
   For any growth rate $\lambda> 0$ of $\Fl$, there is a value $\mu_0=\mu_0(G,\lambda)>0$ such that whenever $\mu\in(0,\mu_0)$, $\Fo$ survives with positive probability.
\end{Theorem}
A trivial and immediate (therefore not very interesting) example of a graph that satisfies the assumptions of Theorem \ref{thm:survival_FPP_1} is an infinite tree of bounded degree.
Less trivial examples are given by the graphs produced by regular tessellations of the hyperbolic plane.
\begin{Remark}\label{rem:fodies}
   $\Fo$ surviving \emph{with positive probability} is the best one can hope for since, for any $\mu>0$ and any $\lambda>0$, with positive probability $\Fo$ dies out: 
   for example, because all neighbors of $o$ are occupied by seeds of $\Fl$. 
   Furthermore, non-amenability is essential in Theorem~\ref{thm:survival_FPP_1} since on $\mathbb{Z}$ (which is hyperbolic but amenable) we obtain that $\Fo$ dies out for all $\mu,\lambda>0$.
\end{Remark}
Theorem~\ref{thm:survival_FPP_1} resembles the result of~\cite[Theorem 1.2]{Stauffer-Sidoravicius-MDLA} on $\mathbb{Z}^d$. The main difference is that, on $\mathbb{Z}^d$, 
$\Fo$ can only survive if $\lambda<1$. 
On hyperbolic, non-amenable graphs, $\Fo$ can survive even if $\lambda$ is arbitrarily large; that is, if $\Fl$ spreads at a much faster rate than $\Fo$. 

Our next result establishes that, regardless of the values of $\mu$ and $\lambda$, $\Fl$ survives almost surely; this does not even require non-amenability.
This is a striking difference with the behavior on $\mathbb{Z}^d$: for example, it shows that there is no strong survival on hyperbolic graphs.
\begin{Theorem}[Survival of $\Fl$]\label{thm:survival_FPP_lambda}
   Let $G$ be a vertex transitive, hyperbolic graph.
   For any $\mu\in (0,1)$ and $\lambda>0$, $\Fl$  survives almost surely.
\end{Theorem}

As an immediate corollary, we establish coexistence for FPPHE.
\begin{Corollary}[Coexistence]\label{corollary:coexistence}
   Let $G$ be a vertex transitive, hyperbolic, non-amenable graph.
   For all $\lambda>0$ there is a value $\mu_0=\mu_0(G,\lambda)>0$ such that for all $\mu\in(0,\mu_0)$, $\Fo$ and $\Fl$ coexist with positive probability.
\end{Corollary}
\begin{proof}
The proof follows directly by putting together the statements of Theorem \ref{thm:survival_FPP_1} and Theorem  \ref{thm:survival_FPP_lambda}.
\end{proof}

\subsection{Related processes}\label{sec:ifpp}
Even though FPPHE was invented to analyze MDLA, growth processes with competition is a classical area of research. 
One example is the two-type Richardson model, which consists of $\Fo$ starting from $o$ and $\Fl$ starting from one active seed 
(for example, located at a neighbor of $o$); thus both $\Fo$ and $\Fl$ start spreading from time $0$. 
It is conjectured that 
coexistence occurs if and only if $\lambda=1$. 
The fact that coexistence occurs when $\lambda=1$ was established in~\cite{HP} on the two-dimensional lattice, and extended to other dimensions by~\cite{Hoffman} and~\cite{GM2005}
(see also~\cite{Hoffman2008,ADH}). 
The converse has not been fully resolved, but 
\cite{HP2} established that the set of values of $\lambda$ for which there is coexistence is at most a countable subset of $(0,\infty)$. 
Simulation suggests that, unlike the two-type Richardson model, FPPHE on $\mathbb{Z}^d$ has a regime of coexistence even when $\lambda\neq 1$; this seems to be a consequence of the richer set of behaviors 
coming from the interplay between $\lambda$ and $\mu$ in FPPHE.


For a simplified, deterministic version of the two-type Richardson model, 
\cite{Benjamini} shows that, on a hyperbolic graph, coexistence is possible for \emph{all} values of $\lambda>0$. Our proof of Theorem~\ref{thm:survival_FPP_lambda} can also 
be used to establish this result for the actual two-type Richardson model. This is the content of the next corollary, which does not need non-amenability.
\begin{Corollary}
   Consider the two-type Richardson model defined above on a vertex transitive, hyperbolic graph.
   Then, for any $\lambda>0$, with positive probability both $\Fo$ and $\Fl$ survive.
\end{Corollary}

FPPHE can also be regarded as a natural model for the spread of two conflicting rumors throughout a network. 
In this setting, $\Fo$ represents a false rumor that starts spreading from time $0$ from the origin.
Vertices hosting seeds of $\Fl$ represent the nodes of the network that can verify that the rumor is actually false. Thus when they receive the rumor (thereby becoming an activated seed), they start spreading 
the correct information.     
In this setting, $G$ being hyperbolic and non-amenable is of particular relevance, since 
several authors observed that some important real-world networks can be modeled by graphs equipped with a hyperbolic structure
\cite{Montgolfier,WSK-bis,Mahoney,Guo}.
The last one, in particular, investigates a randomized algorithm to stop the spread of a false rumor within a large network, 
by using its hyperbolicity properties. This line of work substantially deviates from ours. Nonetheless, our theorems show that, in a hyperbolic graph, in order to stop 
the spread of the false rumor via FPPHE, it is not enough to have $\lambda\geq 1$ (as simulations suggest for $\mathbb{Z}^d$), but one needs to have a sufficiently large density of seeds of $\Fl$.

\subsection{Ideas of the proofs}

\paragraph{Proof idea of Theorem~\ref{thm:survival_FPP_1}.}
In order to show survival of $\Fo$, 
we cannot employ the same strategy as in~\cite{Stauffer-Sidoravicius-MDLA} for $\mathbb{Z}^d$. The reason is that 
\cite{Stauffer-Sidoravicius-MDLA} showed \emph{strong survival}, and their proof technique resorts to showing that every activated seed of $\Fl$ gets eventually 
encapsulated by $\Fo$. In our case, we know from Theorem~\ref{thm:survival_FPP_lambda} that strong survival does not occur for any $\lambda,\mu$. So we had to develop a new strategy, 
which uses the hyperbolicity and non-amenability. 

First non-amenability gives that one can embed a binary tree inside $G$ (cf. Theorem \ref{thm:embedding} below).
Using this, we will construct a tree $\mathbb{T}$ where 
the ``edges'' of $\T$ will represent \emph{almost} geodesic paths of $G$ of some fixed length.
Exploring the exponential growth of $\T$, one could suspect that it is possible to find an infinite path $v_1,v_2,\ldots$ in $\T$ such that the distance between $v_i$ and the set of seeds increases with $i$. 
This would be convenient, as it would imply 
existence of an infinite path in $G$ that gets further and further away from seeds, 
which would give a positive probability for this path to get entirely occupied by $\Fo$.
However, one can show that an infinite path satisfying the above requirement simply does not exist. 

Our approach is to resort to a multi-scale analysis: we will look for  
an infinite path in $\T$ such that the passage times near this path are ``good'' (in the sense that they are close to their expected value) at all scales.  
To do this, we define certain \emph{cylinders} of different sizes, and each size is identified with a \emph{scale}.
The first scale consists of cylinders around each edge of $\T$, 
and each subsequent scale $\j\geq 2$ is formed by cylinders whose axis is a path of length $\j$ in $\T$, and whose width will grow linearly with $\j$.

We say that a cylinder $\cyl$ is good if FPP (of a single type, without competition) started anywhere inside $\cyl$ advances linearly in time (during a time interval of order $\j$). 
For scale $1$, we will also require that $\cyl$ is completely free of seeds to be good.
As a consequence, if $\cyl$ is good then it will follow that the paths traversed by FPPHE from one end of $\cyl$ to the other end will remain close to the axis of $\cyl$.
This is a consequence of the fact that 
hyperbolic graphs have long detours away from the geodesics (cf.\ Proposition \ref{prop:gromov-result} below), and such detours will have typical (or linear) passage times since $\cyl$ is a good cylinder.

We will then show that there is an infinite path $\gpath$ in the tree $\T$ that only intersects good cylinders. The existence of such a path implies that $\Fo$ advances quickly enough near a $d_G$-geodesic ray close enough to $\gpath$.
This does not allow $\Fl$ to get anywhere near $\gpath$ because for that to happen, a seed located outside of scale-$1$ cylinders has to be activated by $\Fo$ coming from $\gpath$ and then propagate $\Fl$ back to $\gpath$.
However, such a detour out of scale-$1$ cylinders is too long, and since $\gpath$ is covered by good cylinders at all scales, the passage time around such a detour will not succeed in bringing $\Fl$ back to $\gpath$ before $\Fo$. 
We emphasize here that this proof can be carried out even when $G$ is not transitive.


\paragraph{Proof idea of Theorem \ref{thm:survival_FPP_lambda}.}

To establish survival of $\Fl$ we proceed as follows. 
Let $\Fo$ run from the origin until it hits a large ball (of fixed radius) completely occupied by seeds.
This will eventually occur.
While $\Fo$ tries to go around this ball (for which, with high probability, it will require an amount of time that is exponential in the radius of the ball), 
$\Fl$ is free to proceed, activate all seeds in the ball and then continue to occupy further regions. In particular, $\Fl$ will have enough time to occupy a larger ball, which will then force $\Fo$ to do an even longer detour.
We will show that, with positive probability, $\Fl$ succeeds in occupying an infinite sequence of balls of increasing radii, before they can be reached by $\Fo$. When this happens, we say that this iteration succeeded.

An iteration succeeding implies that $\Fl$ survives. 
However the probability that an iteration succeeds is only positive, and there are difficulties in showing that $\Fl$ actually survives almost surely. 
For example, if an iteration is unsuccessful, 
one needs to take special care of the measurability of the events that have been observed. 
In particular, we need to define each iteration in such a way that we only observe events that are measurable locally (that is, 
with respect to a finite subset of edges). So, once an iteration fails, we can carry out another iteration inductively.
Another technical difficulty is given by the fact that whenever there is a failure, both $\Fo$ and $\Fl$ have already occupied large regions of $G$.
Thus, we need to carry out each iteration such that its success probability does not depend on the size of the previously explored areas of $G$. 
This is crucial since the set occupied by $\Fo$ grows over time, and each iteration starts from a ball of seeds of $\Fl$ of a large but fixed radius.
Using hyperbolicity, we show that this can be done in such a way that the probability of success of each iteration remains 
bounded away from zero as time goes, implying that eventually an iteration will succeed.
Finally, 
we also need to ensure that at each iteration
$\Fl$ has room to continue expanding ``towards 
infinity'', without getting trapped inside a dead-end of $G$, or curving back towards $o$.
It is in this part that we use transitivity of $G$. 

\subsection*{Outline of the paper}
In Section \ref{sect:preliminaries} we present fundamental properties of first passage percolation and hyperbolic graphs, and give a construction of FPPHE in terms of passage times. 
In Section~\ref{sect:tree} we construct the tree $\mathbb{T}$ described in the proof overview above. 
In Sections \ref{sect:scales} and \ref{sec:tpaths} we set up an inductive argument which is the core of the proof of Theorem \ref{thm:survival_FPP_1}, which is described in Section \ref{sect:survival_FPP_1}.
The proof of Theorem \ref{thm:survival_FPP_lambda} can be found in Section \ref{sect:survival_FPP_lamda}. 
We present conclusive remarks and open questions in Section \ref{sect:questions}, 
and prove a technical results regarding detours away from geodesics on hyperbolic graphs in Appendix \ref{sect:proof_prop_detour}.

\section{Preliminaries}\label{sect:preliminaries}
All graphs considered in this paper will be infinite and of bounded degree. For a graph $G$, we denote by $V(G)$ its set of vertices and by $E(G)$ its set of edges.

\subsection{Hyperbolic and non-amenable graphs}\label{sect:why-hyperbolic}
For any subset of vertices $S\subset V(G)$ the \emph{internal boundary} of $S$ is
\begin{equation}\label{eq:internal-bdary}
\partial S:= \{v\in S \ : \ \exists \, x\in V(G)\setminus S \text{ such that }\{x,v\}\in E(G)\}.
\end{equation}
Moreover, we recall that the \emph{Cheeger constant} of $G$ is defined as 
\begin{equation}\label{eq:cheeger}
h(G):= \inf_{S\subset V(G)}\frac{|\partial S|}{|S|},
\end{equation}
where $|S|$ denotes the cardinality of $S$, and the infimum is taken over all \emph{finite} sets $S \subset V(G)$.
A graph is non-amenable if and only if it has positive Cheeger constant.

Roughly speaking, $G $ is non-amenable if each finite subset has a large boundary, otherwise $G$ is called \emph{amenable}.
%
Non-amenability is often responsible for the appearance of intriguing phenomena in random processes. 
Notable examples are \emph{percolation}, in which non-amenability allows the appearance of infinitely many infinite components~\cite[Chapter 7]{LyonsPeres-book}, 
and \emph{branching random walks}, in which non-amenability brought about a fascinating regime where the process survives indefinitely but is nonetheless transient~\cite{BenjaminiPeres_BRW}.

Let $d_G$ denote the \emph{metric} induced by the shortest-path distance in $G$.
The graph $G$ is called $\delta$-\emph{hyperbolic}, where $\delta \in [0, \infty)$, if for any three vertices $x,y,z\in V(G)$ and any three connecting geodesic segments $\gamma_{x,y}, \gamma_{y,z}, \gamma_{z,x}$ 
between these vertices we have 
\begin{equation}\label{eq:def-hyperbolic}
\forall u\in \gamma_{x,y},  \text{ there is }v\in \gamma_{y,z}\cup \gamma_{z,x} \text{ such that }u\in B_G(v, \delta),
\end{equation}
where $B_G(v, \delta)$ is the ball (with respect to the $d_G$ metric) centered at $v$ with radius $\delta$, and we regard a path in $G$ (such as $\gamma_{x,y}$) as a sequence of vertices.
It is known (cf.\ e.g.\ \cite{hamann} and references therein) that $\delta=0$ if and only if $G$ is a tree, since our results are trivial for trees, we will always assume that $\delta>0$.
In words,~\eqref{eq:def-hyperbolic} means that the entire geodesic $\gamma_{x,y}$ is contained in the union of balls of radius $\delta$ centered at vertices of $\gamma_{y,z}\cup \gamma_{z,x} $.
In this case we shall say that the triangle with vertices $x,y,z$ is $\delta$-thin.
Thus a graph is $\delta$-hyperbolic if and only if each of its triangles is $\delta$-thin. 
If $G$ is $\delta$-hyperbolic for some $0\leq \delta < \infty$ we will simply say that $G$ is \emph{hyperbolic}.
From now on, whenever we consider triangles on $G$ we always mean \emph{geodesic} triangles, that is, their sides lie on geodesic segments.


A class of hyperbolic graphs of large interest is that of Cayley graphs of \emph{hyperbolic groups}.
These were introduced in \cite{RandomGroups-Gromov} in order to study properties of a randomly chosen group.
For example, it is shown in \cite{RandomGroups-Gromov} that if we choose a group with a finite (symmetric) set of generators and a random set of finite representations, 
then the resulting group is hyperbolic with high probability; see also \cite{RandomGroups-Silberman,RandomGroups-Ollivier}.
Recently there have been several works about 
the study of random processes on hyperbolic groups, usually establishing very different behavior with respect to such processes on lattices.
For instance,~\cite{GouezelLalley, Gouezel, Mathieu, Ledrappier} studied the asymptotic behavior of random walks on hyperbolic groups, 
\cite{BenjaminiTessera} investigated \emph{first-passage percolation} on hyperbolic groups and showed the presence of doubly-infinite geodesics, 
and~\cite{Hutchcroft} showed that Bernoulli bond percolation on a non-amenable, hyperbolic, quasi-transitive graph has an intermediate phase in which there are infinitely many infinite clusters. 

\paragraph{Embedding of a tree into $G$.}

A \emph{bilipschitz embedding} of a metric space $(X,d_X)$ into another metric space $(Y,d_Y)$ is a map $f:X\to Y$ such that for some constant $\widehat{\alpha} \geq 1$ and all vertices $u,v\in X$ we have
\[
\widehat{\alpha}^{-1}d_X(u,v)\leq d_Y \bigl ( f(u),f(v)\bigr )\leq \widehat{\alpha} d_X(u,v).
\]
The following result 
is going to be crucial in our proofs. It gives that one can embed a binary tree in any non-amenable graph (without requiring hyperbolicity).
\begin{Theorem}[{\cite[Theorem 1.5]{Benjamini-Schramm-GAFA}}]\label{thm:embedding}
   Let $H$ be a bounded degree, non-amenable graph, considered together with the graph metric. 
   Then there is a tree $T$ contained in $H$, with positive Cheeger constant $h(T) > 0$, such that the inclusion map $T \to H$ is a bilipschitz embedding, and there is a bilipschitz embedding of the binary tree into $H$.
\end{Theorem}

Let $\T_3$ denote the infinite regular tree of degree $3$. 
All graphs $G$ considered in this paper will have bounded degrees. Whenever $G$ is also non-amenable, we will denote by $\widehat{\alpha}$ the constant obtained from the above
bilipschitz embedding; that is, we will take a map $ f:\T_3 \to G$ and denote by $\widehat{\alpha} \geq 1$ the best constant such that, for all $ x,y \in V(\T_3)$,
\begin{equation}\label{eq:bilipschitz-tree}
\widehat{\alpha}^{-1}d_{\T_3}(x,y)\leq d_G(f(x),f(y))\leq \widehat{\alpha} d_{\T_3}(x,y).
\end{equation}
To simplify the notation, after fixing the map $f$ and the value of $\widehat{\alpha}$ so that \eqref{eq:bilipschitz-tree} is satisfied, 
we identify the vertices of $\T_3$ with their images in $G$ through $f$, 
by saying simply ``take $x\in V(\T_3)$'' instead of ``take $x\in V(G)$ such that $f^{-1}(x) \in V(\T_3)$''.

One of the characteristics of hyperbolic graphs is that whenever there is a bilipschitz embedding such as in \eqref{eq:bilipschitz-tree}, then
for any pair of vertices $x,y \in V(\T_3)$ the image of any geodesic segment (on $\T_3$) connecting $x$ with $y$ via $f$ is ``close'' to some geodesic segment (in $G$) connecting $f(x)$ with $f(y)$.
This is the content of the next result, whose proof can be found in \cite{Ohshika, Papadopoulos}.
The statement that we write here is adapted to our context from \cite[Theorem 2.31]{Ohshika} and {\cite[Corollaire 2.6]{Papadopoulos}.

\begin{Proposition}
\label{prop:ohshika-papad}
Let $X_1$ and $X_2$ be two $\delta$-hyperbolic, bounded-degree graphs, and let $f:X_1\to X_2$ be a bilipschitz embedding with constant $\widehat{\alpha} \geq 1$.
Then there is a constant $\kappa=\kappa(\widehat{\alpha},\delta)$ so that the following holds.
For any pair of vertices $x,y\in V(X_1)$, let $\gamma_{f(x),f(y)} $ be any geodesic in $X_2$ from $f(x)$ to $f(y)$, and $\gamma_{x,y}$ be any geodesic in $X_1$ from $x$ to $y$. 
Let $\tilde\gamma_{x,y}$ be a quasi-geodesic in $X_2$ obtained from the concatenation of the geodesic segments $\gamma_{f(z_1),f(z_2)}$ of $X_2$ for each edge $\{z_1,z_2\}$ in $\gamma_{x,y}$.
Then,  
$$
   \forall z\in \tilde \gamma_{x,y}, \; d_{X_2}(z, \gamma_{f(x),f(y)}) \leq \kappa 
   \quad
   \text{and} 
   \quad
   \forall z\in \gamma_{f(x),f(y)}, \; d_{X_2}(z, \tilde \gamma_{x,y})\leq \kappa.
$$
\end{Proposition}

We point out that since we are dealing with discrete graphs, the constant $\delta$ will always be a non-negative integer.

Any graph $G$ with $h(G)>0$ has exponential growth, and its growth rate is bounded from below by $1+h(G)>1$.
Furthermore, let $1<\Delta<\infty$ denote the maximum degree of $G$.
Then, for all $n>0$ and all vertices $x\in V(G)$, we have
\begin{equation}\label{eq:exp-gwth}
\bigl ( 1+h(G)\bigr )^n \leq \left | B_G(x,n) \right |\leq \Delta^n.
\end{equation}
A crucial property of hyperbolic graphs that we are going to use is that detours from geodesic segments are very long.
\begin{Proposition}\label{prop:gromov-result}
Let $G$ be a $\delta$-hyperbolic graph.
Given a vertex $x_0\in V(G)$ and a value $r>0$, consider the ball $B_G(x_0,r)$ with center $x_0$ and radius $r$.
Now take a geodesic segment $\gamma$ that goes through $x_0$ and has its endpoints $y,z$ outside of $B_G(x_0,r)$.
Then, any path started at $y$ and ended at $z$ which does not intersect $B_G(x_0,r)$ has length bounded from below by $\delta 2^{r/\delta}$.
\end{Proposition}
The proof of this fact can be found in \cite[Sections 6 and 7]{Gromov}, in the more general context of $\delta$-thin, geodesic metric spaces.

\subsection{First-passage percolation and coexistence}\label{sect:fpp-definitions}
We start by defining the \emph{first passage time}, following the notation of \cite[Section 1]{Auffinger}.
For any edge $\{x_0,x_1\}\in E(G)$ take a non-negative random variable $t_{\{x_0,x_1\}}$, which we will refer to as \emph{the passage time} of the edge $\{x_0,x_1\}$.
We assume that the collection $\{t_e\}_{e\in E(G)}$ is i.i.d.\ with a common distribution fixed beforehand.
In this work we will always assume that the common distribution is exponential of mean $1$.
We will later show how FPPHE can be constructed from $\{t_e\}_{e\in E(G)}$.
The random variable $t_{\{x_0,x_1\}}$ can be interpreted as the time needed to cross the edge $\{x_0,x_1\}$.
%
%
For any finite path $\gamma=(x_0, x_1, \ldots , x_n)$ we define the FPP-time of $\gamma$ to be
\begin{equation}\label{eq:def_T(P)}
T(\gamma):=\sum_{i=0}^{n-1} t_{\{x_i,x_{i+1}\}}.
\end{equation}
Moreover, for any pair of vertices $x,y\in V(G)$ we set
\[
T(x\to y):=\inf_{\gamma:x\to y}T(\gamma),
\]
where the infimum is taken over all paths from $x$ to $y$.
Roughly speaking, the quantity $T(x\to y)$ can be seen as the ``shortest time'' needed to go from vertex $x$ to vertex $y$.

A standard interpretation of FPP is that of a \emph{random metric} on $G$, in fact one can define the new distance $d_{\operatorname{FPP}}$ as 
$d_{\operatorname{FPP}}(x,y):=T(x\to y)$.
In this way one can think of FPP as a process ``spreading'' in time, in the sense that at time $T=0$ the set of vertices ``occupied'' by the process consists only of the starting vertex.
Inductively, this set will grow and for all times $T>0$ the set of vertices occupied by the process consists of all vertices at $d_{\operatorname{FPP}}$-distance at most $T$ from the starting vertex.
We point out, however, that by the definition of the model, the (random) metric structure induced by FPP is lost whenever $\lambda\neq 1$, which reminds the behavior of other FPP-based competition models, such as the Richardson model with unequal rates and chase-escape.

\subsubsection*{First-passage percolation spreads linearly in time}
We start with a lemma showing that on a graph with bounded degree, FPP is likely to move linearly in time.
To fix the notation, for $x\in V(G)$ and $T\geq 0$ we let 
$
A^x_T:= \bigl \{y\in V(G) \ : \ T(x\to y)\leq T\bigr \};
$
and for simplicity we set
$
A_T:=A^o_T.
$
In words, for all $T\geq 0$ and each $x\in V(G)$ we have
\[
A^x_T=\{\text{vertices reached by FPP started at }x\text{ and run for time }T\}.
\]
Recall that $B_G(x,L)$ is the ball centered at $x$ of radius $L>0$ with respect to the metric $d_G$.

\nc{c:typical-FPP-cin}
\nc{c:typical-FPP}
\begin{Lemma}\label{lemma:typical-FPP}
Suppose that $G$ has maximum degree $\Delta$, then for any constant $\uc{c:typical-FPP}>0$ there exists a constant $\cout:=\cout(\Delta,\uc{c:typical-FPP})>1$, such that for every $x\in V(G)$ and all $T>0$,
\begin{equation}\label{eq:Eq1}
\P\left [ A^x_T\subseteq B_G(x, \cout T) \right ]\geq 1-e^{-\uc{c:typical-FPP} T}.
\end{equation}
Moreover, for any constant $0<\uc{c:typical-FPP-cin}<1$ there exists a positive constant $\cin:=\cin(\Delta,\uc{c:typical-FPP-cin})<1$, such that for every $x\in V(G)$ and all $T>0$
\begin{equation}\label{eq:Eq2}
\P\left [ B_G(x, \cin T)\subseteq A^x_T \right ]\geq 1-e^{-\uc{c:typical-FPP-cin} T}.
\end{equation}
\end{Lemma}

The lemma above follows from the estimate below, which we collect in another lemma since we will need to refer to it later.
\begin{Lemma}\label{lemma:typical-FPP2}
   For any positive integer $\ell$, any $S\leq \ell/2$ and any path $P=(x_0,x_1,\ldots ,x_\ell)$ of length $\ell$ in $G$, we have 
   $
      \P\left [ T(P) \leq  S \right ]\leq 2 \frac{e^{-S}S^\ell}{\ell!}.
   $
   Similarly, for any $T\geq1$, any integer $\ell \leq T$ and any path $P=(x_0,x_1,\ldots ,x_\ell)$ of length $\ell$ in $G$, we have 
   $
      \P\left [ T(P) \geq T \right ]\leq \ell \left(\frac{T e}{\ell}\right)^\ell e^{-T}.
   $
\end{Lemma}
\begin{proof}
   Note that
   $
      \P \left [T(P) \leq S \right ]
      = \P\left [ \Poi ( S)\geq \ell \right ],
   $
   where $\Poi ( S)$ denotes a Poisson random variable with parameter $S$.
   This is due to the following. The first-passage time $T(P)$ of a path $P$ is a sum of $\ell$ i.i.d.\ exponential random variables with mean $1$.
   Thus, to demand that a sum of $\ell$ i.i.d.\ random variables distributed as Exp$(1)$ is at most $S$, 
   is equivalent to demand that within a time-interval of length $S$, one has witnessed at least $\ell$ observations of a Poisson process of rate $1$.
   This is equivalent to ask that a Poisson random variable with parameter $S$ is at least $\ell$.
   Hence,
   \[
      \P \left [T(P) \leq S \right ]
      \leq \sum_{k\geq \ell} \frac{e^{- S} \, S^k}{k!}
      \leq 2 \frac{e^{- S} \, S^\ell}{\ell!},
   \]
   where in the last step we use the fact that $\ell \geq 2S$.
   The second part of the lemma is analogous:
   \[
      \P \left [ T(P)\geq  T \right ]
      \leq \P\left [ \Poi ( T)\leq \ell \right ] 
      \leq \sum_{k\leq \ell} \frac{e^{- T}\, T^k}{k!}.
   \]
   For all $k \leq \ell \leq T$ the quantity $ \frac{ T^k}{k!}$ is increasing in $k$, then for all values $\ell \leq T$ we have
   $$
      \P \left [ T(P)\geq  T \right ]
      \leq 
      e^{-T}\sum_{k\leq \ell} \frac{ T^k}{k!}
      \leq e^{-T} \ell\frac{ T^\ell}{\ell!} 
      \leq e^{-T} \ell \left(\frac{Te}{\ell}\right)^{\ell},
   $$
   thus finishing the proof.
\end{proof}

\begin{proof}[Proof of Lemma~\ref{lemma:typical-FPP}]
   We start with the first part of the statement.
   Observe that 
   \[
   \P\left [  A^x_T\subseteq B_G(x, \cout T) \right ]  = \P \left [\sup_{w\in A^x_T}d_G(x,w)\leq \cout T \right ] 
   \geq 1- \P \left [\sup_{w\in A^x_T}d_G(x,w)> \cout T \right ].
   \]
   Recalling that, for any path $P_\ell=(x_0,x_1,\ldots ,x_\ell)$ of length $\ell\geq 1$, 
   $T(P_\ell) $ denotes the FPP time of $P_\ell$ as defined in \eqref{eq:def_T(P)}. 
   Using that 
   $\P \left [\sup_{w\in A^x_T}d_G(x,w)> \cout T \right ]$ is the probability that 
   there exists a path started at $x$ of length larger than $\cout T$  with FPP-time at most $T$,
   and applying Lemma~\ref{lemma:typical-FPP2}, we get that for all $\uc{c:typical-FPP}$ large enough we can choose $\cout$ such that  
   \[
   \P \left [\sup_{w\in A^x_T}d_G(x,w)> \cout T \right ]
   \leq \sum_{\ell> \cout T}\, \sum_{P_\ell \text{ started at }x} \P \left [ T(P_\ell)\leq  T \right ]
   \leq \sum_{\ell> \cout T}\, \sum_{P_\ell \text{ started at }x} e^{-2 \uc{c:typical-FPP}\ell}.
   \]
   Then we observe that the number of paths of length $\ell\geq 1$ started at a fixed vertex $x$ is at most $\Delta^\ell$, by the bounded degree assumption.
   Thus,
   we take $\uc{c:typical-FPP}$ large enough so that, for all $\ell \geq \cout T$, we have
   $
   \P \left [\sup_{w\in A^x_T}d_G(x,w)> \cout T \right ]\leq 
   \sum_{\ell> \cout T}\Delta^\ell e^{-2\uc{c:typical-FPP} \ell}
   \leq  e^{-\uc{c:typical-FPP} \ell}.
   $
   
   To show the second part of the statement we proceed analogously:
   \[
   \P\left [ B_G(x, \cin T)\subseteq A^x_T \right ] = \P \left [ \inf_{w\notin A^x_T}d_G(x,w)> \cin T\right ]
    \geq 1- \P \left [ \inf_{w \notin A^x_T}d_G(x,w) \leq \cin T\right ].
   \]
   Then, $\P \left [ \inf_{w\notin A^x_T}d_G(x,w) \leq \cin T\right ] \leq \P\left [\exists \text{ a path of length }\leq \cin T \text{ with FPP-time}\geq  T\right ]$, and 
   \[
   \P \left [ \inf_{w\notin A^x_T}d_G(x,w) \leq \cin T\right ] 
    \leq \sum_{\ell \leq \cin T}\, \sum_{P_\ell \text{ started at }x} \P \left [ T(P_\ell)\geq  T \right ]
    \leq \sum_{\ell \leq \cin T}\, \sum_{P_\ell \text{ started at }x} e^{-T} \ell \left(\frac{Te}{\ell}\right)^\ell.
   \]
   Therefore, we obtain
   \[
   \P \left [ \inf_{w\notin A^x_T}d_G(x,w) \leq \cin T\right ] 
   \leq \sum_{\ell \leq \cin T} e^{-T}\ell \left(\frac{ \Delta T e}{\ell}\right)^\ell
   \leq (\cin T)^2e^{-T} \left(\frac{ \Delta T e}{\cin T}\right)^{\cin T},
   \]
   where in the last step we used that 
   $\ell \leq T$.
   The above can be expressed as
   \[
   (\cin T)^2e^{-T} \frac{ (\Delta T e)^{\cin T}}{(\cin T)^{\cin T}} = \exp \left \{ 2\ln (\cin T)+\cin T \ln \left (\frac{\Delta}{\cin}\right ) +(-1+\cin) T \right \},
   \]
   which implies that by taking $\cin=\cin(\Delta, \uc{c:typical-FPP-cin})<1$ small enough, we have
   $
   2\ln (\cin T)+\cin T \ln \left (\frac{\Delta}{\cin}\right ) +(-1+\cin) T < -\uc{c:typical-FPP-cin}T,
   $
   concluding the proof of the lemma.
\end{proof}
\begin{Remark}\label{rem:FPPs}
   In the proof of Lemma \ref{lemma:typical-FPP} we used that the distribution of each $t_e$ is exponential with mean $1$.
   However, the very same proof can be repeated if the $\{t_e\}_{e\in E(G)}$ are i.i.d.\ exponential random variables with rate $\lambda$.
   In fact, in this case, note that $\{\lambda t_e\}_{e\in E(G)}$ are exponential random variables of rate $1$, and hence,
    \eqref{eq:Eq1} and \eqref{eq:Eq2} could be simply replaced by
   \[
   \P\left [ A^x_{T/\lambda}\subseteq B_G\left (x, \cout T\right ) \right ]\geq 1-e^{-\uc{c:typical-FPP} T}
   \quad\text{and}\quad
   \P\left [ B_G\left (x, \cin T\right )\subseteq A^x_{T/\lambda} \right ]\geq 1-e^{-\uc{c:typical-FPP-cin} T},
   \]
   respectively. Another way would be to write
   \[
   \P\left [ A^x_{T}\subseteq B_G\left (x, \cout \lambda T\right ) \right ]\geq 1-e^{-\uc{c:typical-FPP} \lambda T}
   \quad\text{and}\quad
   \P\left [ B_G\left (x, \cin \lambda T\right )\subseteq A^x_{T} \right ]\geq 1-e^{-\uc{c:typical-FPP-cin} \lambda T}.
   \]
   %
   %
   %
\end{Remark}

\begin{Remark}
   Note that~\eqref{eq:Eq1} is stronger than~\eqref{eq:Eq2} 
   in the sense that $\uc{c:typical-FPP}$ can be arbitrarily large, while $\uc{c:typical-FPP-cin}$ is restricted to be smaller than $1$. 
   The reason for this is that it is much easier to violate the event $\{B_G(x,\cin T) \subseteq A_T^x\}$ 
   in~\eqref{eq:Eq2}, 
   since for this it is enough that all edges adjacent to $x$ have a large passage time. 
   On the other hand, to violate the event  $\{A_T^x \subset B_G(x, \cout T)\}$ 
   from~\eqref{eq:Eq1} 
   one needs to have several edges with a small passage time. 
\end{Remark}

\subsection{Construction of FPPHE from $\{t_e\}_{e\in E(G)}$.}\label{sec:constrfpphe}
Let $\{$seeds$\}$ denote the set of seeds of $\Fl$; so $\{$seeds$\}$ is the subset of $V(G)\setminus\{o\}$ obtained by adding each vertex independently with probability $\mu$.
We start with a collection of independent random variables $\{t_e\}_{e\in E(G)}$, with common distribution  Exp($1$), and a collection of \emph{seeds}.
Then $\Fo$ uses the passage times given by $\{t_e\}_{e\in E(G)}$ to spread from $o$, activating seeds whenever it tries to occupy vertices already occupied by a seed.
Activated seeds of $\Fl$ spread using the passage times $\{ t_e/\lambda\}_{e \in E(G)}$.
In other words, if $x\in V(G)$ gets occupied by $\Fl$ at time $t$, and a neighbor $y$ of $x$ is not occupied, then $y$ gets occupied by $\Fl$ at time $t+t_{\{x,y\}}/\lambda$ unless it gets occupied through another edge first.

\section{Construction of the embedded tree $\mathbb{T}$}\label{sect:tree}
Recall that $\T_3$ denotes the embedded tree found in Theorem \ref{thm:embedding}.
When $G$ is a vertex-transitive graph we can safely assume that $o\in V(\T_3)$, but for general graphs this might not be the case.
In order to solve this problem, we consider an ``augmented'' version of $\T_3$ in the following sense.
We take any $d_G$-geodesic path between the origin $o$ and the set $V(\T_3)$, and look at it as if it was a (unique) finite extra branch of $\T_3$.
More precisely, let $o'\in V(\T_3)$ be such that the following relation is satisfied
\[
d_G(o,o')=\min_{v\in V(\T_3) } \{d_G(o,v)\}.
\]
If the above holds for more than one such $o'$ we just fix one arbitrarily.

Now we proceed to the construction of another embedded tree $\T$ starting from $\T_3$.
Recall that $\widehat{\alpha}\geq 1$ is the bilipschitz constant appearing in \eqref{eq:bilipschitz-tree}, and recall Theorem \ref{thm:embedding}.
Let $\gamma_{o,o'} $ denote a $d_G$-geodesic between $o$ and $o'$. 
We now define a tree $\T_3\cup \gamma_{o,o'}$, whose vertex set is $V(\T_3)\cup \gamma_{o,o'}$, and 
for every $u\in \gamma_{o,o'} $ and $x\in V(\T_3)\cup \gamma_{o,o'}$ we have 
\begin{equation}\label{eq:d_T3}
d_{\T_3\cup\gamma_{o,o'}}(u,x):=\left \{ 
\begin{array}{ll}
d_G(u,o')+d_{\T_3}(o',x ) & \text{ if }x\in V(\T_3)\\
d_G(u,x) & \text{ if } x\in \gamma_{o,o'}.
\end{array}
\right .
\end{equation}
It is straightforward to verify that this notation is well defined and there is a bilipschitz embedding of $\T_3\cup \gamma_{o,o'}$ into $G$ with the same constant $\widehat{\alpha}$ that we had before.

We now start the construction of the embedded tree $\T\subset\T_3\cup \gamma_{o,o'} $.
First set the root of $\T$ to be $o$.
\nr{r:o-o'}
Now let  
\[
\ur{r:o-o'}:=\ur{r:o-o'}(G, \widehat{\alpha}):=\lceil 2\widehat{\alpha} d_G(o,o')\rceil,
\]
and fix a large integer $\r>0$ (which will be specified later on) such that $\r\geq \ur{r:o-o'}$.
Note that, from this choice, we have that $d_G(o,o')\leq \r/2$.
\begin{Remark}
In the following we will deal with a sequence of values $\ur{r:o-o'}, \ldots,\ur{r:another-1}$ all of them are ``large enough'' so that several constraints are satisfied.
For our proof to work it suffices to choose a value 
$\r \geq \max \{ \ur{r:o-o'}, \ldots,\ur{r:another-1}\},$
and leave it fixed throughout.
\end{Remark}
Note that if $G$ is transitive then we should go through the same construction with the only simplification that $o'\equiv o$.

Consider two vertices $w,z\in V(\T_3)$ such that the following occur:
\[
d_{\T_3}(o',w)=d_{\T_3}(o',z)=\r, \quad \text{and}\quad d_{\T_3}(w,z)=2\r.
\]
(Since $o'$ is now the \emph{root} of $\T_3$, we are picking $w$ and $z$ to be two arbitrary vertices in generation $\r$ of $\T_3$.)
Note that their images via $f$ will be vertices at $d_G$-distance $ r\in [\widehat{\alpha}^{-1}\r, \widehat{\alpha} \r]$ from the $o'$.
We add $w$ and $z$ to $\T$ as children of $o$ so that 
the two paths joining $o$ with $w$ and $o$ with $z$ on $\T_3\cup \gamma_{o,o'}$ correspond to \emph{edges} of $\T$. We say that $w$ and $z$ together form the ``first generation'' of $\T$.
Inductively, suppose that we have defined $\T$ up to generation $k$, for some $k\geq 1$ and denote the vertices at generation $k$ by $u_1, \ldots u_{2^k}$.
Then, for each such $u_\ell$, $\ell\in \{1, \ldots 2^k\}$, 
consider two vertices $w_\ell,z_\ell$ of $\T_3$ such that all the following occur:
\begin{itemize}
\item[(i)] $d_{\T_3}(o',w_\ell)=d_{\T_3}(o',z_\ell )=(k+1)\r$, 
\item[(ii)] $d_{\T_3}(w_\ell,z_\ell)=2\r$, 
\item[(iii)] $d_{\T_3}(u_\ell,w_\ell)=d_{\T_3}(u_\ell,z_\ell )=\r$.
\end{itemize}
The set of vertices $\{w_1,z_1, \ldots , w_{2^k},z_{2^k}\}$ form the $(k+1)$-th generation of $\T$.
By proceeding inductively in this way we obtain a (rooted) binary tree $\T$ embedded in $G$ whose edges are \emph{almost} geodesic segments, in the sense that their length corresponds to the distance of the two endpoints in $G$, up to a multiplicative constant.
More precisely, for any two vertices $u,v\in V(\T)$ we say that $u,v$   are $\T$-neighbors whenever $u$ and $v$ are neighboring vertices in $\T$, that is, $d_{\T}(u,v)=1$.
From now on, by ``\emph{the embedded tree}'' we will always mean $\T$, unless otherwise stated.
\begin{Lemma}\label{lemma:T-bilip-emb}
For $\T$ defined as above there is a bilipschitz embedding of $\T$ into $G$.
\end{Lemma}
\begin{proof}
From the definition of $\r$ it follows that for any pair of distinct vertices $x,y\in V(\T)$ we have the following two possibilities:
\begin{itemize}
\item[(i)] $x= \{o\}$, and $y \in V(\T_3)\setminus \{o'\}$;
\item[(ii)] $x, y \in V(\T_3)$.
\end{itemize}
We start by showing that there is a constant $\alpha:=\alpha(G)\geq  \widehat{\alpha}\geq 1$ such that 
$
d_G(x,y)\leq \alpha\r  d_{\T}(x,y),
$
for all pairs $x,y\in V(\T)$.
In case (i) we have
\[
d_G(x,y) = d_G(o,y) \leq \frac{\r}{2}+\widehat{\alpha} d_{\T_3}(o',y )= \frac{\r}{2}+\widehat{\alpha}\r d_{\T}(o,y ) \leq \left(\widehat{\alpha}+\frac{1}{2}\right)\r d_{\T}(o,y).
\]
Analogously, in case (ii) we have
\[
d_G(x,y)\leq \widehat{\alpha}  d_{\T_3}(x,y )\leq \r \widehat{\alpha} d_{\T}(x,y ).
\]
Now we proceed to show the reversed inequality, that is, there is $\alpha>\widehat{\alpha}$ such that 
$
d_G(x,y)\geq \alpha^{-1}\r  d_{\T}(x,y),
$
for all pairs $x,y \in V(\T)$.
We start with case (i) as above.
\[
d_G(x,y)  = d_G(o,y) \geq \left | d_G(o,o')-d_G(o',y)\right |.
\]
From the definition of $\r$ it follows that
\[
\left | d_G(o,o')-d_G(o',y)\right |=d_G(o',y)-d_G(o,o')=d_G(o',y)\left ( 1-\frac{d_G(o,o')}{d_G(o',y)}\right ).
\]
Using the definitions of $\T$ and $\r$ we deduce that
\[
d_G(o,o')\frac{1}{d_G(o',y)}\leq \left (\frac{\r}{2\widehat{\alpha} }\right )\frac{1}{\widehat{\alpha}^{-1} d_{\T_3}(o',y)}\leq \frac{\r}{2 }\frac{1}{\r d_{\T}(o',y)}\leq \frac{1}{2 }.
\]
Therefore we have
\[
d_G(o,y) 
 \geq d_G(o',y)-d_G(o,o')\geq d_G(o',y)\left ( 1-\frac{1}{2}\right )
 \geq \widehat{\alpha}^{-1}d_{\T_3}(o',y)\frac{1}{2}\geq \frac{\r}{2\widehat{\alpha}}d_{\T}(o,y).
\]
Case (ii) is a straightforward consequence of the fact that $\T_3$ is a bilipschitz embedding, in fact
$
d_G(x,y)\geq \widehat{\alpha}^{-1}d_{\T_3}(x,y)\geq \widehat{\alpha}^{-1}\r d_{\T}(x,y).
$
\end{proof}

\section{Construction of multiscale and good cylinders}\label{sect:scales}
In this section we set up the definitions and methods for an inductive argument that will be carried on to prove Theorem \ref{thm:survival_FPP_1}.
In order to proceed, we need to introduce some more notation, including the concept of \emph{cylinder}, which will be used to prove Theorem \ref{thm:survival_FPP_1}.

As before, for any two vertices $x,y \in V(G)$, let $\Gamma_{x,y}$ denote the \emph{set} of geodesics (with respect to the graph distance $d_G$) that connect $x$ to $y$ on $G$.
Note that this set might have more than one element because we are not assuming that geodesics are unique.
\begin{Definition}[Cylinder]\label{def:cylinder}
For any two vertices $x,y \in V(G)$ and constant $L\geq 1$ we define the \emph{cylinder} $\cyl^{(L)}_{x,y}$ as the union of all balls (with respect to the metric $d_G$) of radius $L$ centered at vertices on any geodesic connecting $x$ with $y$.
More precisely, we set
\[
\cyl^{(L)}_{x,y}:= \bigcup_{\gamma\in \Gamma_{x,y} } \bigcup_{w\in \gamma } B_G(w, L).
\]
\end{Definition}
A straightforward application of $\delta$-hyperbolicity implies that for any $x,y\in V(G)$, for any two geodesics $\gamma_1$ and $\gamma_2$ in $\Gamma_{x,y}$ we have
\begin{equation}\label{eq:2geod}
\sup_{u\in \gamma_1} \inf_{v\in \gamma_2} d_G(u,v)\leq \delta.
\end{equation}

\begin{Proposition}\label{prop:detour_cylinders}
   Consider two vertices $x,y\in V(G)$ such that $d_G(x,y)\geq  50 \delta$,
   and take any geodesic $\gamma_{x,y}$ between $x$ and $y$.
   Set $d:=d_G(x,y)$, and fix an integer value $L$ such that $9\delta\leq L\leq d/5$. 
   Moreover, fix two distinct vertices $u,v$ on $\gamma_{x,y}$ such that both $d_G(x,u)$ and $d_G(y,v)$ are at least $L+1$ and $ d_G(u,v)\geq 2.5 L$.
   Then any path started at $x$ and ended at $y$ that avoids the cylinder $\cyl^{(L)}_{u,v}$ has length bounded from below by
   $
   \frac{1}{3L}\delta d_G(u,v)\cdot 2^{L/(2\delta)} .
   $
\end{Proposition}
The above result is just an extension of the fact that detours away from geodesics are exponentially large (cf.\ Proposition~\ref{prop:gromov-result}); we defer the proof to Appendix~\ref{sect:proof_prop_detour}.
\begin{Remark}\label{rem:new_c''}
We emphasize that if the path under consideration starts from somewhere in $B_G(x, L)$ and ends somewhere in $B_G(y,L)$ (instead of starting at vertex $x$ and ending at vertex $y$), then it suffices to modify the result by a multiplicative constant.
In fact, by excluding the two balls $B_G(x, L) $ and $B_G(y, L)$, we would obtain a new lower bound 
\begin{equation}\label{eq:lwer_bound_prop}
\frac{1}{3L}\delta d_G(u,v)\cdot 2^{L/(2\delta)} -2L\geq \frac{1}{3L}\delta [d_G(u,v)-2L]\cdot 2^{L/(2\delta)} \geq \frac{1}{12L}\delta d_G(u,v)\cdot 2^{L/(2\delta)},
\end{equation}
where the last inequality follows because $d_G(u,v)\geq 2.5 L$ and  $L\leq d/5$.
\end{Remark}

\subsection{First Scale (Scale $1$)}\label{sect:scale_1}
Recall that we are dealing with two different FPP processes, $\Fo$ and $\Fl$ described in the Introduction,
and that by ``\emph{seeds}'' we mean the starting points of $\Fl$.
Recall also that $T(P)$ is the passage time of a path $P$ with respect to the passage times $\{t_e\}_{e\in E(G)}$, which are of rate $1$ and are used in the construction of FPPHE, and that 
$A_t^w$ is the ball of radius $t$ centered at $w\in V(G)$ according to (the random metric induced by) $t_e$, $e\in E(G)$. We then define 
\begin{equation}\label{eq:Awmax}
   A_t^{w,\max} := \text{ ball of radius $t$ centered at $w$ according to $\max\{1,\lambda\}t_e$, $e\in E(G)$}
\end{equation}
and
\begin{equation}\label{eq:Awmin}
   A_t^{w,\min} := \text{ ball of radius $t$ centered at $w$ according to $\min\{1,\lambda\}t_e$, $e\in E(G)$}.
\end{equation}

Choose a constant $\varepsilon>0$ arbitrarily small which will be kept fixed throughout, and recall that $\r\geq 1$ is a fixed integer used in the definition of $\T$. 
We will take $\r$ to be large enough with respect to $\varepsilon$. 
Now, for $w\in V(G)$ let
\[
\mathscr{P}_w:=\{\text{all self-avoiding and finite }G\text{-paths starting from }w\},
\]
and for $P\in \mathscr{P}_w$ let $|P|$ denote its $d_G$-length.

Consider any pair $x,y \in V(\T)$ and denote by $\j :=d_{\T}(x,y)$.
By the bilipschitz embedding, this implies (cf.\ Lemma \ref{lemma:T-bilip-emb}) that there is a value $\alpha\geq \widehat{\alpha}\geq 1$ such that
$
d_G(x,y)\in [\alpha^{-1} \j \r , \alpha \j \r].
$
Recall the constants $\cin $ and $\cout$ from Lemma \ref{lemma:typical-FPP}, and define the event
\begin{equation}\label{eq:def-G-hat}
{\mathcal{G}}_1 (x,y;\j\r):= 
\left \{
\begin{split}
& \forall t\in \left [ \varepsilon^{1/2} \j\r , \  4\frac{\cout}{\cin^2}\j\r\right ]\cap \mathbb{Z}, \text{ and }  \forall w\in \cyl^{(\varepsilon \j\r)}_{x,y}  \text{ we have}\\
a) & \, B_G(w, \min \{1, \lambda\}\cin t)\subseteq 
A_t^{w,\min} \subseteq  A_t^{w,\max} \subseteq  \ B_G(w, \max \{1, \lambda\}\cout t) \\
b) & \, \forall \, P\in \mathscr{P}_w \text{ with }\sqrt{\varepsilon} \cout \j\r\leq |P|\leq 4\frac{\cout^2}{\cin^2}\j\r, \text{ we have }\frac{|P|}{\cout}\leq T(P)
\end{split}
\right \}.
\end{equation}
Note that the dependence on $\j\r$ is written in order to avoid confusion, but the notation is redundant as this quantity is fixed once we fix the pair $x$ and $y$.

For any pair $x,y\in V(\T)$ of $\T$-neighboring vertices, paths satisfying condition $b)$ described in $\mathcal{G}_1(x,y;\r)$ for some initial vertex $w\in \cyl^{(\varepsilon \r)}_{x,y} $ will be called \emph{typical} at scale 1.
\begin{Definition}[Good cylinders, scale $1$]
For $\cin $ and $\cout$ as in Lemma \ref{lemma:typical-FPP} and for any pair of $\T$-neighboring vertices $x,y$ we define the \emph{cylinder} $\cyl^{(\varepsilon \r)}_{x,y} $ to be \emph{good at scale} $1$ if the following two conditions are satisfied:
\begin{itemize}
\item[(i)] 
The first condition is that the event $\mathcal{G}_1(x,y;\r)$ holds cf.~\eqref{eq:def-G-hat} with $\j=1$.
\item[(ii)] 
Choose a constant $\beta$ so that
\begin{equation}\label{eq:def-BETA}
\beta:= (6+\varepsilon)(1+\alpha)\alpha^2\frac{\cout}{\cin}\max\{\lambda, \lambda^{-1}\}.
\end{equation} 
Then the second requirement is that 
the event $\mathcal{G}_2(x,y)$ below holds:
\[
\mathcal{G}_2(x,y):= \left \{ \cyl^{(\varepsilon \r+ \beta\r )}_{x,y}  \ \cap \ \{\text{seeds}\}=\emptyset\right \} .
\]
\end{itemize}
\end{Definition}
In words, condition (i) requires that for all integer times $t$ in the interval $\left [ \varepsilon^{1/2} \r , 4\frac{\cout}{\cin^2}\r\right ]$, for each vertex $w$ of the cylinder the set of vertices reached by 
a FPP process of rate $\lambda$ or $1$ started at $w$ contains a ball of radius $\min\{1, \lambda\}\cin t$ and is contained inside a ball of radius $\max\{1, \lambda\}\cout t$.
Furthermore, the passage time $T(P)$ of any path $P$ is not too short. Note that a path $P$ has an expected passage time of $|P|$, and we use the factor $1/\cout$ to get the event to hold with high probability for 
all such paths. Part $b)$ of event $\mathcal{G}_1(x,y;\r)$ will be used to show that long detours away from geodesics cannot happen, in particular, 
if $\Fo$ tries to deviate from $\gamma_{x,y}$ inside a good cylinder $\cyl_{x,y}^{(\varepsilon \r)}$, 
the time to traverse such detour cannot be too small. 


The choice of the radius of the cylinder in $\mathcal{G}_2(x,y)$ is technical and will become clear later on (cf.\ the proof of Theorem \ref{thm:survival_FPP_1} in Section \ref{sect:survival_FPP_1}).

We first aim to show that for a careful choice of the parameters the probability that a given cylinder is \emph{good at scale }$1$ is high.
Recall that $x,y \in V(\T)$.
Our first result shows that the event $\mathcal{G}_2(x,y)$ occurs with high probability, uniformly for every choice of $\T$-neighboring vertices $x$ and $y$.

\nc{c:1}

\begin{Lemma}\label{lemma:aux-1}
For any $\r\geq 1$ and constant $\uc{c:1}>0$, 
there is a $\mu_0:=\mu_0(\uc{c:1}, \Delta, \varepsilon,\beta, \delta, \alpha, \r)>0$ small enough such that for all $\mu<\mu_0$ and all $\T$-neighboring vertices $x,y$, we have
\[
\P\left [ \mathcal{G}_2(x,y)\right ]\geq 1-e^{-\uc{c:1} \r}.
\]
\end{Lemma}
Before proving this result, we state a fundamental fact.
%
Recall the bounds determined in \eqref{eq:exp-gwth}, then we have that each cylinder $\cyl^{(\varepsilon \r)}_{x,y}$ is such that 
\begin{equation}\label{eq:volume_cyl}
\left |\cyl^{(\varepsilon \r)}_{x,y}\right |\leq \Delta^{(\varepsilon \r+\delta)} \alpha\r.
\end{equation}
This bound holds due to the following.
Each vertex inside the cylinder has maximum degree $\Delta$.
By the bilipschitz embedding (Lemma \ref{lemma:T-bilip-emb}), if $d_{\T}(x,y)=1$, then $\alpha^{-1}\r \leq d_G(x,y)\leq \alpha \r$.
Each pair of vertices is joined by geodesics that are not necessarily unique, but by $\delta$-hyperbolicity two geodesics joining $x$ and $y$ are at distance at most $\delta$ from each other (cf.\ Equation \eqref{eq:2geod}).
Thus, it follows that by fixing a geodesic $\gamma_{x,y}$, then $\cyl^{(\varepsilon \r)}_{x,y}\subset \cup_{w\in \gamma_{x,y}}B_G(w, \varepsilon \r + \delta)$.
\begin{proof}[Proof of Lemma \ref{lemma:aux-1}]
Recall the definition of $\beta$ from \eqref{eq:def-BETA} and that the value of $\r$ is fixed.
Then, we can choose $\mu_0 := \mu_0(\uc{c:1}, \Delta, \varepsilon,\beta, \delta, \alpha, \r)$ such that for all $\mu<\mu_0$ the cylinder $\cyl_{x,y}^{(\varepsilon \r+\beta\r)} $ does not contain seeds with high probability.
More precisely, given $\uc{c:1}>0$ we can choose $\mu_0$ such that for all $\mu<\mu_0$ we have
$
\Delta^{\varepsilon \r+\beta\r +\delta} \alpha\r \mu < e^{-\uc{c:1}\r}.
$
In this way, by using the union bound we obtain
$
\P[\mathcal{G}_2^c(x,y)] 
 \leq \left |\cyl^{(\varepsilon \r+\beta \r)}_{x,y}\right | \mu
 \stackrel{\eqref{eq:exp-gwth}, \eqref{eq:volume_cyl} }{\leq} \Delta^{\varepsilon \r+\beta\r+\delta} \alpha\r \mu.
$
\end{proof}
The next result states that paths that are typical at scale 1 are very likely.



\nc{c:C'}
\nr{r:typ-lik}

\begin{Lemma}\label{lemma:typical-likely}
   Recall the constants $\uc{c:typical-FPP-cin}$ and $\uc{c:typical-FPP}$ from Lemma \ref{lemma:typical-FPP}.
   For any $\uc{c:typical-FPP-cin}\in (0,1)$ and $\uc{c:typical-FPP}>0$, 
   there exists $\uc{c:C'}>0$, a small enough $\varepsilon=\varepsilon(\uc{c:typical-FPP-cin}, \uc{c:typical-FPP}, \Delta )>0$ 
   and a large enough $\ur{r:typ-lik}=\ur{r:typ-lik}(\uc{c:C'},\varepsilon,\alpha)$
   so that 
   for all $\r>\ur{r:typ-lik}$ and  
    all $x,y \in V(\T)$,  $\j:=d_{\T}(x,y)$, we have 
   \[
   \P \left [\mathcal{G}_1 ^c(x,y;\j\r)\right ]\leq e^{-\uc{c:C'}\sqrt{\varepsilon}\j\r}.
   \]
   Moreover, $\mathcal{G}_1(x,y;\j\r)$ is measurable with respect to the passage times in $\cyl^{(\lceil \varepsilon \j\r+4\max \{1, \lambda\}\frac{\cout^2}{\cin^2}\j\r\rceil+1 )}_{x,y}$.
\end{Lemma}
\begin{proof}
From Lemma \ref{lemma:typical-FPP} it follows that for any $w\in \cyl^{(\varepsilon \j\r)}_{x,y}$ and $t\in  \left [\lfloor \varepsilon^{1/2} \j\r\rfloor , 4\frac{\cout}{\cin^2}\j\r\right ]$
\[
\P \Bigl (  \{B_G(w, \min \{1, \lambda\}\cin t)\not \subseteq 
A_t^{w,\min}\} \cup \{ A_t^{w,\max} \not \subseteq  \ B_G(w, \max \{1, \lambda\}\cout t)\}
\Bigr )\leq e^{-\uc{c:typical-FPP}t}+e^{-\uc{c:typical-FPP-cin}t}.
\]
Furthermore, the first part of Lemma \ref{lemma:typical-FPP2} shows that for any path of length $\ell $ such that
$
\varepsilon^{1/2} \cout \j\r\leq \ell  \leq 4\frac{\cout^2}{\cin^2}\j\r,
$
we have that for any fixed $w\in \cyl^{(\varepsilon \j\r)}_{x,y}$ 
\begin{eqnarray*}
   \P  \left [
    \exists P\in \mathscr{P}_w\text{ with }|P|=\ell\text{ such that }  T(P)<\tfrac{\ell }{\cout}
   \right ] 
   &\leq \Delta^{\ell } 2 \frac{e^{-\ell/\cout}(\ell/\cout)^\ell}{\ell!}\\
   &\leq 2e^{-\ell/\cout} \left( \frac{\Delta e \ell}{\cout \ell}\right)^\ell.
\end{eqnarray*}
Now, for any $\uc{c:typical-FPP}$, one can choose $\cout$ large enough so that the above is at most $e^{-\uc{c:typical-FPP}\ell}$.
Thus, by the union bound over all possible starting points $w\in \cyl^{(\varepsilon \j\r)}_{x,y} $, as well as all possible values of $\ell$ and $t$:
\[
\P \left [\mathcal{G}_1 ^c(x,y;\j\r)\right ] 
\leq \left |\cyl^{(\varepsilon \j\r)}_{x,y}\right | \left [ 
\sum_{\ell=\sqrt{\varepsilon}\j\r}^{4 \frac{\cout^2 }{\cin^2}\j\r}e^{-\uc{c:typical-FPP}\ell} 
+
\sum_{t = \sqrt{\varepsilon} \j\r}^{4 \frac{\cout }{\cin^2}\j\r} \left (e^{-\uc{c:typical-FPP}t}+e^{-\uc{c:typical-FPP-cin}t}\right )\right ]
 \leq \left (\alpha \j\r \Delta^{\varepsilon \j\r+\delta} \right )e^{-2\uc{c:C'} \sqrt{\varepsilon} \j\r}, 
\]
for some constant $\uc{c:C'}>0$ depending on $\uc{c:typical-FPP-cin}$ and $\uc{c:typical-FPP}$.
Now taking $\varepsilon>0$ small enough, and then $\r$ large enough establishes the first part of the lemma. 
The second part follows directly from the definition, since we have that $ A_t^{w,\max} \subseteq  \ B_G(w, \max \{1, \lambda\}\cout t)$, which means that we only need to check all passage times inside the (finite) ball $B_G(w, \max \{1, \lambda\}\cout t)$.
\end{proof}

The next lemma shows that cylinders at scale $1$ are very likely to be good.

\nc{c:good_scale_1}
\begin{Lemma}\label{lemma:good_scale_1}
Recall the constants $\uc{c:typical-FPP-cin}$ and $\uc{c:typical-FPP}$ from Lemma \ref{lemma:typical-FPP}.
There is a constant $\uc{c:good_scale_1}=\uc{c:good_scale_1}(\uc{c:typical-FPP-cin},\uc{c:typical-FPP},\varepsilon)>0$ 
so that for all $\r\geq \ur{r:typ-lik}$, where $\ur{r:typ-lik}$ is from Lemma \ref{lemma:typical-likely}, there exists a $\mu_0>0$ for which whenever $\mu<\mu_0$ we have that for all $\T$-neighboring vertices $x,y$,
\[
\P\left [ \cyl^{(\varepsilon \r)}_{x,y} \text{ is good at scale 1}\right ]\geq 1-e^{-\uc{c:good_scale_1} \r}.
\]
\end{Lemma}
\begin{proof}
The definition of good cylinder at scale $1$ implies that for any fixed pair of $\T$-neighbors $x,y\in V(\T)$ we have
$
\P \left [\cyl^{(\varepsilon \r)}_{x,y} \text{ is not good}\right ]\leq \P \left [\mathcal{G}_1^c(x,y;\r)\right ]+\P\left [\mathcal{G}_2^c(x,y)\right ].
$
The lemma follows from Lemmas \ref{lemma:aux-1} and  \ref{lemma:typical-likely}.
\end{proof}

\subsection{Higher scales}\label{sect:higher_scales}
In this section we define what \emph{higher scales} are, and what a \emph{good cylinder at a higher scale} is.
These concepts will be used in Section \ref{sect:survival_FPP_1} to show survival of $\Fo$.


\begin{Definition}[Good cylinders, higher scales]
Recall the definition of a cylinder from Definition \ref{def:cylinder}.
Consider two distinct and not $\T$-neighboring vertices $x,y\in V(\T)$ and let $ \j:=d_\T(x,y)$.
Note that $\j\geq 2$.
Then the cylinder $\cyl^{(\varepsilon \j\r)}_{x,y}$ is \emph{good at scale} $\j\geq 2$ if the event $\mathcal{G}_1(x,y;\j\r)$ defined in \eqref{eq:def-G-hat} is realized.
\end{Definition}
We emphasize that Lemma \ref{lemma:typical-likely} implies that for all $x,y\in V(\T)$ with $\j=d_\T(x,y)\geq 2$
\[
\P \left [ \cyl^{(\varepsilon \j\r)}_{x,y} \text{ is good at scale }\j\right ]=\P [\mathcal{G}_1(x,y;\j\r)]\geq 1-e^{-\uc{c:C'}\sqrt{\varepsilon}\j\r}\geq 1-e^{-\uc{c:good_scale_1}\j\r}.
\]

\section{Good $\T$-paths}\label{sec:tpaths}
In what follows, we will need the definition of a \emph{good $\T$-path} (such path is to be found on the tree $\T$).
Roughly speaking, a sequence of $\T$-neighboring vertices $\{v_i\}_{i\geq 0}\in V(\T)$ (where we set $v_0:=o$) is a good $\T$-path if it is covered by good cylinders at all scales.

\subsection{Definition and properties of Good $\T$-paths}
For any infinite $\T$-path $\gpath$ we define
\begin{equation}\label{eq:def-gd-path}
\left \{\gpath \text{ is good}\right \} \ := \ \bigcap_{k=1}^\infty \bigcap_{
\footnotesize \begin{array}{ll}
& u,v\in \gpath \\ 
& d_\T(u,v)=k
\end{array}} 
\normalsize
\left \{\cyl^{(\varepsilon k\r)}_{u,v} \text{ is good} \right \}.
\end{equation}
Now suppose that we have found an infinite good $\T$-path $\gpath$ (we will show in Section \ref{sect:ex-gpath} that indeed it exists), and fix it throughout this section.
In the following, we set
\begin{equation}\label{eq:V(gpath)}
V( \gpath) := \left \{v\in V(G) \ : \ v\in  \left (\gpath  \cup  V\bigl ( E(\gpath )\bigr ) \right )\right \},
\end{equation}
where $V\left ( E(\gpath )\right )$ is the set of vertices in $V(G)$ that lie on the ``edges'' of $\gpath$; more precisely, we define the set $V\left ( E(\gpath )\right )$ as follows.
For any pair of $\T$-neighboring vertices $w,z \in V(\T)$, there is a geodesic (with respect to $d_{\T_3}$) path on the binary tree $\T_3$ that connects $w$ and $z$, denote it by $\gamma_{\T_3}(w,z)$. 
(Note that $d_{\T_3}(w,z)=\r$.)
For all $w,z\in V(\T_3)$ consider the image of the geodesic $\gamma_{\T}(w,z)$ via the bilipschitz embedding.
This gives a sequence of $d_{\T_3}$-adjacent vertices, and each pair can be connected by a $d_G$-geodesic path.
The concatenation of such $d_G$-geodesic paths is itself a path in $G$, denote it by $\pi_G(w,z) $.
Therefore for all pairs $w,z$ such that $d_{\T}(w,z)=1$ we define 
\[
V(w,z):=\{v\in G \ : \ v\in \pi_G(w,z)\},
\]
and finally we set 
\[
V\left ( E(\gpath )\right ):= \bigcup_{ 
w,z \in \gpath, \ d_{\T}(w,z)=1   }
V(w,z).
\]
For a graphical representation see Figure \ref{fig:V(path)}.

\begin{figure}[h!]
\begin{center}
\includegraphics[scale=.8]{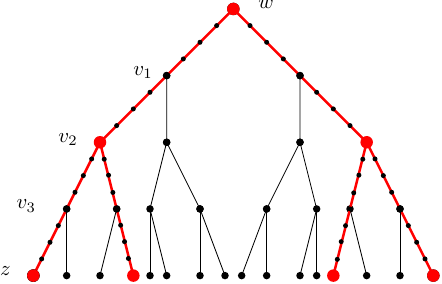}
\caption{Red vertices such as $w, v_2$ and $z$ are those on $\T$, large black vertices on a red path such as $v_1,  v_3$ are those on $\T_3$, whereas the smallest vertices are those of $G$.
Here $\r=2$.}\label{fig:V(path)}
\end{center}
\end{figure}



Once we have fixed the path $\gpath$, we define
\[
\cyl^{(\varepsilon \r/2)}_{o,\infty}:= \bigcup_{v\in V(\gpath) }B_G\left (v, \frac{\varepsilon\r}{2}\right ).
\]
Note that this quantity (and those defined throughout this section) depends on $\gpath$, but we omit this dependence in order to make the notation less heavy.
For any vertex $z\in \gpath$ define
\[
N(z):= \left \{ w\in \cyl^{(\varepsilon \r/2)}_{o,\infty} \ : \ \forall z'\in \gpath \text{ we have }  d_G(w,z)\leq  d_G(w,z')  \right \},
\]
with ties broken according to an arbitrary rule so that $N(z)$, $z\in \gpath$, forms a partition of $\cyl_{0,\infty}^{(\varepsilon \r/2)}$.
In words, for any $z\in \gpath $, $N(z)$ is the set of all vertices inside $\cyl^{(\varepsilon \r/2)}_{o,\infty}$ that are closer to $z$ than to any other vertex of the tree $\T$.

At this point we use a fundamental fact, namely that for large enough $\r$, for any vertex $x \in  \gpath$ there are vertices $u', v'\in \gpath$ such that
\begin{equation}\label{eq:Nx}
N(x) \subset \cyl^{(2\varepsilon \r/3)}_{u',v'};
\end{equation} 
in particular, $u'$ and $v'$ can be the vertices on $\gpath$ before and after $x$, which from now on we denote simply by $u$ and $v$.
Therefore for all vertices $x\in \gpath$ we can define
\begin{equation}\label{eq:CNx-CNy}
\cyl(N(x)):= \cyl^{(2\varepsilon \r)}_{u,v} .
\end{equation}
Furthermore, note that we have $d_G(N(x), \partial \cyl(N(x)))\geq \frac{4}{3}\varepsilon \r$.
However, later on we will simply use the fact that $d_G(N(x), \partial \cyl(N(x)))\geq \varepsilon \r/2$.
\begin{Remark}
Note that by construction it follows that the vertices $u,v \in   \gpath$ defined above are such that
$
d_{\T}(u,v)= 2,
$
thus, the event $\{\cyl(N(x))$ is good$\}$ is contained in the event $\{$all cylinders $\cyl^{(2\varepsilon \r)}_{w,z}$ with $w,z\in \gpath$ and $d_{\T}(w,z)=2$ are good$\}$.
\end{Remark}

%

Now we proceed with the following technical result.
Recall the definition of the internal boundary of a set, given in \eqref{eq:internal-bdary}, and the result from Proposition \ref{prop:ohshika-papad}.
\nr{r:r-delta-kappa}
\begin{Lemma}\label{lemma:close-geodesics}
   Let $\kappa=\kappa(\alpha,\delta)$ be the constant appearing in Proposition \ref{prop:ohshika-papad}, and let $\ur{r:r-delta-kappa}=\ur{r:r-delta-kappa}(\varepsilon, \delta, \kappa)$ be defined so that
   $
   \ur{r:r-delta-kappa}\geq \frac{2}{\varepsilon}(3 \delta+\kappa).
   $
   Then for all $\r \geq \ur{r:r-delta-kappa}$, for any pair of vertices $x,y\in \gpath$, any two points $\o\in \partial N(x) $ and $\ent\in \partial N(y)$, all $d_G$-geodesics $\gamma_{\o, \ent}\in \Gamma_{\o, \ent}$ are completely contained inside the set $ \cyl^{(\varepsilon\r/2)}_{\o,x}\cup \cyl^{(\varepsilon\r/2)}_{x,y} \cup \cyl^{(\varepsilon\r/2)}_{y, \ent}$.
   Furthermore, the set $ \cyl^{(\varepsilon\r/2)}_{\o,x}\cup \cyl^{(\varepsilon\r/2)}_{x,y} \cup \cyl^{(\varepsilon\r/2)}_{y, \ent}$ also contains the (unique) $d_\T$-geodesic $\gamma_\T (x,y)$.
\end{Lemma}
\begin{proof}
In order to proceed, we need to make use of the fact that triangles are $\delta$-thin.
More precisely, let $\o'$ and $\ent' $ be any two projections (that can be chosen arbitrarily if they are not unique) on the geodesic segment $\gamma_{x,y}$ of the vertices $\o$ and $\ent$ respectively.
In formulas:
\[
\begin{split}
\o' & := \text{any vertex }v\in 
\gamma_{x,y} \, \text{ such that } d_G(\o, v)=d_G(\o, \gamma_{x,y});\\
\ent' & := \text{any vertex }v\in \gamma_{x,y} \, \text{ such that } d_G(\ent, v)=d_G(\ent, \gamma_{x,y}).
\end{split}
\]
Now consider the two triangles $\{\o,\o' , \ent\} $ and $\{\o' , \ent, \ent'\} $ (cf.\ Figure \ref{fig:2triangles}).
\begin{figure}[h!]\label{fig:2triangles}
\begin{center}
\includegraphics[scale=0.6]{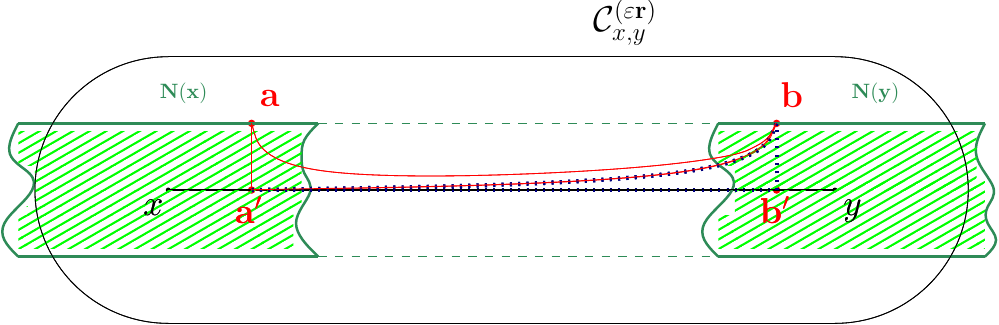}
\caption{The sets $N(x)$ and $N(y)$ are drawn in green (and they exit $\cyl_{x,y}^{(\varepsilon \r)}$).
The triangles described above are drawn with a thin red line and a blue thick dotted line respectively.}
\end{center}
\end{figure}
%
%
These are $\delta$-thin.
A direct consequence of this fact is that any (arbitrarily chosen) geodesic $\gamma_{\o,\ent}$ is contained inside the union of balls of radius $\delta$ centered at
$
\Gamma_{\o,\o'}\cup \Gamma_{\o',\ent}.
$
Using again $\delta$-thinness, we have that any geodesic $\gamma_{\o',\ent}\in \Gamma_{\o',\ent}$ has to be contained inside the union of the balls of radius $\delta$ centered at
$
\Gamma_{\o',\ent'}\cup \Gamma_{\ent', \ent}.
$
Thus, it immediately follows that $\gamma_{\o,\ent}$ is contained inside the union of balls of radius $2\delta$ centered at
$
\Gamma_{\o,\o'}\cup \Gamma_{\o',\ent'}\cup \Gamma_{\ent', \ent}.
$
A straightforward consequence of this is that any geodesic $\gamma_{\o,\ent}$ is such that
\[
\gamma_{\o,\ent}\subset \cyl^{(2\delta)}_{\o,\o'}\cup \cyl^{(2\delta)}_{\o',\ent'}\cup \cyl^{(2\delta)}_{\ent', \ent}.
\]
%
By construction, this union is contained inside $\cyl^{(3\delta)}_{\o,x}\cup \cyl^{(3\delta)}_{x,y} \cup \cyl^{(3\delta)}_{y, \ent}$ which, together with $\r\geq \ur{r:r-delta-kappa}$, gives the first part of the statement.
The final sentence of the statement follows from Proposition~\ref{prop:ohshika-papad}.
\end{proof}

\subsection{Existence of a good $\T$-path}\label{sect:ex-gpath}

The next step consists in showing that with positive probability there is an infinite good $\T$-path $\gpath$ containing the 
origin $o$.
\begin{Proposition}\label{prop:existence_gpath}
Recall the definition of good path from \eqref{eq:def-gd-path}, and assume that $\r$ is large enough.
Then,
$
\P\left [ \text{There is an infinite $\T$-path }\gpath \text{ containing }o \text{ and such that }  \gpath \text{ is good} \right ]>0.
$
\end{Proposition}
Before proceeding to the proof, we need to introduce some terminology. 

\paragraph{Cutsets.}
We refer to the term \emph{cutset} whenever we mean a subset of vertices of $\T$ that separates the root $o$ from infinity.
In particular, we will need the so-called \emph{minimal} cutsets, which are cutsets that do not have any subsets which are cutsets themselves.
More precisely, a cutset $\Pi_k$ of cardinality $k$ is such that
\[
\Pi_k\subseteq V(\T) \ : \ |\Pi_k|=k, \text{ and any infinite path containing }o\text{ passes through }\Pi_k.
\]
Furthermore, $\Pi_k$ is a \emph{minimal} cutset of cardinality $k$ if
\[
\Pi_k \text{ is a cutset, } |\Pi_k|=k, \text{ and for all subsets }  S\subsetneq \Pi_k , \ S \text{ is not a cutset}.
\]
Consider the tree $\T$ rooted at $o$.
Since $\T$ has no cycles, any vertex of the tree is a cutpoint, i.e., its removal separates the tree into two disjoint connected components.
Note that if we remove a minimal cutset from $\T$, the external boundary of the finite connected component containing $o$ is the minimal cutset.
Therefore, we can identify any minimal cutset with the internal boundary of a finite induced subtree of $\T$ containing $o$.
Since the boundary of this subtree is the set of its leaves, all minimal cutsets of cardinality $k$ (for any $k\geq 1$) correspond to the boundary of some \emph{rooted} induced subtree of $\T$ that has exactly $k$ leaves.
%
The next lemma finds an upper bound on the number of minimal cutsets of cardinality $k$.
\begin{Lemma}\label{lemma:nr-cutsets}
For all $k\geq 1$ we have
\[
\Bigl | \{\Pi_k \ : \ \Pi_k \text{ is a minimal cutset of $\T$ of cardinality }k\}\Bigr |<4^{k-1}.
\]
\end{Lemma}
\begin{proof}
Start by recalling the definition of \emph{Catalan number}: set $C_0:=1$ and for all $n\geq 1$ set
$
\operatorname{C}_n := \frac{1}{n+1} {2n \choose n}.
$
From the definition, for all $n\geq 2$, $\operatorname{C}_n$ satisfies the recursive relation
$
\frac{\operatorname{C}_n}{\operatorname{C}_{n-1}}= \frac{4n-2}{n+1}<4.
$
Hence, 
\begin{equation}\label{eq:bound_catalan}
\operatorname{C}_n\leq 4^n, \ \text{ for all }n\geq 1.
\end{equation}
%
In the case of a binary tree, a minimal cutset $\Pi_k$ is the (internal) boundary of a rooted (binary) tree which has exactly $k$ leaves.
Thus, by our previous discussion, the number of such $\Pi_k$'s is known to be the Catalan number $\operatorname{C}_{k-1} $.
Hence, for all $k\geq 1$ we have
\[
\Bigl | \{\Pi_k \ : \ \Pi_k \text{ is a minimal cutset of $\T$ of cardinality }k\}\Bigr |= \operatorname{C}_{k-1}\stackrel{\eqref{eq:bound_catalan} }{<}4^{k-1},
\]
finishing the proof.
\end{proof}

We proceed now with a proof of the proposition.
\begin{proof}[Proof of Proposition \ref{prop:existence_gpath}]
%
Consider two vertices $x,y\in V(\T)$ and let $\j:=d_\T (x,y)$, and suppose that $\cyl_{x,y}^{(\varepsilon \j\r)} $ is bad.
To simplify the notation, we assume that $x$ is an ancestor of $y$, and denote 
\[
\i:=d_\T (o,x);
\]
thus $ \i+\j=d_\T (o,y)$.
Moreover, for any vertex $v\in \T$ and integer $n\geq 0$ we denote by $\u(v,n)$ the ancestor of $v$ at generation $n$, and we set this to be $o$, if such a vertex does not exist.
More precisely, $\u(v,n)$ is the vertex in $V(\T)$ satisfying the following properties:
\begin{itemize}
\item $d_\T \bigl (o,\u(v,n)\bigr )=n$, with $ \u(v,n):=o$ whenever $d_\T \bigl (o,v\bigr )\leq n$;
\item $\u(v,n)$ 
belongs to the shortest path from $o$ to $v$. 
\end{itemize}
To determine whether a cylinder $\cyl_{x,y}^{(\varepsilon \j\r)}$ is good or not, we need to observe all passage times contained in the (larger) set
\[
\cyl_{x,y}^{(\lceil \varepsilon + 4 \max \{ 1, \lambda\}\frac{\cout^2}{\cin^2} \rceil \j\r+1)}\subseteq \cyl_{x,y}^{(\lceil \varepsilon + 4 \max \{ 1, \lambda\}\frac{\cout^2}{\cin^2} +1\rceil \j\r)}.
\]
This is a consequence of the second part of Lemma~\ref{lemma:typical-likely}.
In order to make the notation less heavy, we set
\begin{equation}\label{eq:defin-ETA}
\eta:=\left \lceil \varepsilon + 4  \max \{ 1, \lambda\}\frac{\cout^2}{\cin^2} +1\right \rceil .
\end{equation}
For every pair of vertices $x,y$ as described above, for which $\cyl_{x,y}^{(\varepsilon \j\r)} $ is bad, we declare that 
the entire (infinite) subtree rooted at $\u(x,\i-3\alpha^2\eta\j) $ is also bad, that is, all descendants of $\u(x,\i-3\alpha^2\eta\j) $ (including itself) are declared to be bad.
Throughout this proof we will say that we will \emph{remove} (or \emph{discard}) the root  $\u(x,\i-3\alpha^2\eta\j) $ and when this happens we also remove/discard the entire subtree.
Recall that for all $x\in V(\T)$ such that  $\i-3\alpha^2\eta \j\leq 0$ we set $\u(x,\i-3\alpha^2\eta\j):=o$.

Any vertex of $\T$ has at most $2^{\lceil 3\alpha^2\eta \rceil \j } $ descendants at $d_\T$-distance $\lceil 3\alpha^2\eta \rceil \j$.
This implies that for any given $v\in V(\T)\setminus \{o\} $ there are at most $2^{\lceil 3\alpha^2\eta \rceil \j} $ vertices $x$ such that $d_{\T}(o,x)=\i\geq \lceil 3\alpha^2\eta \rceil \j$ and $v=\u(x,\i-3\alpha^2\eta \j) $.
For a graphical representation see Figure \ref{fig:o-x-u}.

\begin{figure}[h!]\label{fig:o-x-u}
\begin{center}
\includegraphics[scale=0.9]{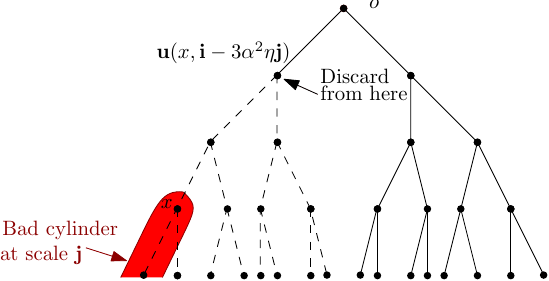}
\caption{A graphical representation of the definition of $\u(x,\i-3\alpha^2\eta \j)$.
The dashed subtree is the one that will be discarded due to the bad cylinder $\cyl_{x,y}^{(\varepsilon \j\r)}$.}
\end{center}
\end{figure}

\nc{c:good_cyl}
By Lemmas \ref{lemma:good_scale_1} and \ref{lemma:typical-likely}, there is a constant $\uc{c:good_cyl}$ such that 
for all $\j\geq 1$ and all large enough $\r$ we have
\[
\P \left [ \text{a given cylinder at scale }\j\text{ is not good}\right ]\leq e^{-\uc{c:good_cyl}\j\r}.
\]
Thus, by the union bound over all possible bad ``scale-$\j$ cylinders'' that could cause any vertex $v\in V(\T)\setminus \{o\} $ to be discarded (as the root of an infinite induced subtree as described above) we obtain
\begin{equation}\label{eq:u_removed}
\P \left [ \text{vertex }v\in V(\T)\setminus\{o\}\text{ is discarded due to a bad cylinder at scale }\j \right ]\leq 2^{3\alpha^2\eta \j}e^{-\uc{c:good_cyl}\j\r}.
\end{equation}
In particular, this quantity can be made arbitrarily small by choosing $\r$ very large, as we argue below.
The above holds for any vertex that is not $o$, since there are more vertices $x$ for which $o=\u(x,\i-3\alpha^2\eta\j)$.
The number of such vertices $x$ is at most $2^{3\alpha^2\eta \j+1}$, the size of a complete binary tree of depth $3\alpha^2\eta \j$.
Thus, by taking also the union bound over all values of $\j\geq 1$, relation \eqref{eq:u_removed} implies that
\[
\P \left [ \text{a given vertex }v\in V(\T) \text{ is removed}\right ]\leq \sum_{\j\geq 1}  2^{3\alpha^2\eta \j+1}e^{-\uc{c:good_cyl}\j\r}.
\]
\nc{c:C'''}
\nr{r:C'''}
At this point for any given constant $\uc{c:C'''}>0$ arbitrarily large, we can pick a value $\ur{r:C'''}=\ur{r:C'''}(\alpha, \eta,\uc{c:good_cyl}, \uc{c:C'''})>0$ so large that for all $\r \geq \ur{r:C'''}$
\[
\P \left [ \text{a given vertex }v\in V(\T) \text{ is removed}\right ] 
\leq \sum_{\j\geq 1}  2^{3\alpha^2\eta \j+1}e^{-\uc{c:good_cyl}\j\r}
<e^{-\uc{c:C'''}}.
\]
%
Now the next claim shows on each branch of $\T$ vertices are removed independently of one another. 
\begin{Claim}\label{claim:indep-cutset}
Consider any two vertices $u,v\in V(\T)$, such that
\begin{itemize}
\item $v$ is not an ancestor of $u$ (that is, $v\neq \u(u,n)$, for all $n\geq 0$) and
\item $u$ is not an ancestor of $v$ (that is, $u\neq \u(v,n)$, for all $n\geq 0$).
\end{itemize}
Then $\P \left [\{u \text{ is removed}\}\cap \{v \text{ is removed}\} \right ] = \P \left [u \text{ is removed}\right ]\P \left [ v \text{ is removed} \right ] $.
\end{Claim}
\begin{proof}
Suppose that there is a scale $\j$ such that vertex $u$ is removed because of the existence of a bad cylinder $\cyl_{w,z}^{(\varepsilon \j \r)} $ at scale $\j$.
Note that in this case $u=\u(w,\i-3\alpha^2\eta\j)$ where $\i =d_\T (o, w)$.
Similarly, suppose that there is a scale $\j'$ such that vertex $v$ is removed because of the existence of a bad cylinder $\cyl_{w',z'}^{(\varepsilon \j' \r)} $ at scale $\j'$.
In this case $v=\u(w',\i'-3\alpha^2\eta\j')$ where $\i' =d_\T (o, w')$.

Now take two vertices $a,a'\in V(G)$ such that $a\in \gamma_{w,z}\subset \Gamma_{w,z}$ and $a'\in \gamma_{w',z'}\subset \Gamma_{w',z'}$.
We want to show that $d_G(a,a')\geq \eta\j\r+\eta\j'\r$.
In the following, let
\[
\begin{split}
& a_\T \text{  be a vertex in } V(\T) \text{ such that } d_G(a, a_\T)=d_G(a, \gamma_{\T}(w,z)), \text{ and }\\
& a'_{\T}\text{  be a vertex in } V(\T) \text{ such that } d_G(a', a_\T')=d_G(a', \gamma_{\T}(w',z')).
\end{split}
\]
Then, because of the bilipschitz embedding we deduce that
\begin{equation}\label{eq:dG(a,a')}
 \begin{split}
 d_G(a, a') 
 &\geq - d_G(a,a_{\T}) +d_G(a_{\T},a_{\T}')- d_G(a',a_{\T}')\\
&\geq - \kappa -\alpha\r/2 +\alpha^{-1}\r d_\T(a_{\T},a_{\T}')- \kappa  -\alpha\r  /2.
\end{split}
\end{equation}
The last inequality follows from Proposition \ref{prop:ohshika-papad} together with the bilipschitz embedding.
In fact, the factor $\kappa$ is an upper bound on the distance between $a$ (resp.\ $a'$) and the image of $\gamma_{w,z}$ (resp.\ $\gamma_{w',z'}$) on $V(\T_3)$, and the largest possible $d_G$-distance between any vertex of $V(\T_3)$ and $V(\T)$ is $\alpha\r /2$.
\nr{r:kappa-alpha}
At this point we can define a value $\ur{r:kappa-alpha}=\ur{r:kappa-alpha}(\kappa, \alpha)>0$ such that
\begin{equation}\label{eq:r-kappa-alpha}
\ur{r:kappa-alpha} \geq \frac{2\kappa}{2\alpha-1}.
\end{equation}
Subsequently, by taking $\r \geq \ur{r:kappa-alpha}$  we observe that relation \eqref{eq:dG(a,a')} is bounded from below by
\[
\begin{split}
- 2\kappa -\alpha\r +\alpha^{-1}\r d_\T(w,w')
& \geq - 2\kappa -\alpha\r +\alpha^{-1}\r (d_\T(w,u)+d_\T (v, w'))\\
& \geq - 2\kappa -\alpha\r +3\alpha \eta (\j+\j')\r\\
& \stackrel{\r\geq \ur{r:kappa-alpha} }{\geq }\eta (\j+\j')\r.
\end{split}
\]
This finishes the proof of the Claim.
\end{proof}

We now complete the proof of Proposition \ref{prop:existence_gpath}.
Now we bound the probability of finding any cutset of vertices that have been removed, separating the root from infinity.
For each $k\geq 2$ fixed, given a (fixed) minimal cutset $\Pi_k$, the probability that $\Pi_k$ consists only of removed vertices is bounded from above by
\[
\begin{split}
\P \left [ \Pi_k \text{ consists only of removed vertices} \right ] & \, \stackrel{\text{Claim }\ref{claim:indep-cutset} }{=} \, \prod_{\u \in \Pi_k} \P [\u \text{ is removed}]
 \leq \left ( e^{-\uc{c:C'''}}\right )^{|\Pi_k |}
 = e^{-\uc{c:C'''}k  }.
\end{split}
\]
By the union bound, together with Lemma \ref{lemma:nr-cutsets} one has that
\[
\begin{split}
\P &\left [ \exists \text{ a minimal cutset consisting of discarded vertices}\right ]\\
& \leq \sum_{k=1}^{\infty}
e^{-\uc{c:C'''}k }\cdot \left | \{\text{ minimal cutsets of cardinality }k\}\right |
 \stackrel{\text{Lemma }\ref{lemma:nr-cutsets} }{\leq }\sum_{k=1}^{\infty}e^{-\uc{c:C'''}k }\cdot 4^{k-1}.
\end{split}
\]
\nc{c:Cv}
Now fix a large constant $\uc{c:Cv}$ so that, whenever $\uc{c:C'''}$ is large enough, we have
$
\sum_{k=1}^{\infty}e^{-\uc{c:C'''}k  }\cdot 4^{k-1}<e^{-\uc{c:Cv}}<1/1000.
$
This shows that with probability at least $1-e^{-\uc{c:Cv}}>999/1000$, there is at least one good path covered by good cylinders at all scales.
This concludes the proof of Proposition \ref{prop:existence_gpath}.
\end{proof}

\section{Survival of $\Fo$ (Proof of Theorem \ref{thm:survival_FPP_1})}\label{sect:survival_FPP_1}
Recall the definition of a good $\T$-path from \eqref{eq:def-gd-path}.
In this section we show that if we have a good $\T$-path $\gpath$ containing $o$ on $\T$ (by Proposition \ref{prop:existence_gpath} this event occurs with probability bounded away from $0$), then $\Fo$ survives indefinitely with positive probability. 
Recall the notation introduced in \eqref{eq:V(gpath)}, namely 
$
V(\gpath) = \left \{v\in V(G) \ : \ v\in \left ( \gpath \cup V\left ( E(\gpath )\right ) \right ) \right \}.
$
To achieve our goal, we make the following assumptions and, by \emph{contradiction}, we show that the presence of a good path makes it impossible for $\Fl$ to surround $\Fo$.

\paragraph{Assumptions.} 
Recall the definition of the constant $\beta$ from \eqref{eq:def-BETA}.
We assume the following conditions:
\begin{itemize}
\item[(A.1)] There is an infinite good $\T$-path $\gpath $ on $\T$ containing the origin $o$.
Note that since $\gpath $ is good we have (cf.\ Section \ref{sect:scale_1})
\[
\left \{ \bigcup\nolimits_{v\in V(\gpath) } B_G \left (v, \beta\r+\varepsilon\r \right )\right \}\cap \{\text{seeds}\}= \emptyset.
\]
\item[(A.2)] There is a vertex $y$ on $ \gpath$ such that $\Fo$ started from $o$ activates a seed $\s\in V(G)$ and the $ \Fl$ originated at $\s$ occupies a vertex in $N(y)$.
From now on, $y$ will denote the first (in time) vertex of $\gpath$ satisfying this assumption.
\end{itemize}

\begin{Definition}\label{def:FPPHE-geod}
From now on, with the terminology ``\emph{a geodesic of FPPHE}'' we mean the following.
For each vertex draw an oriented edge toward the neighbor from which it got occupied (by $\Fo$ or $\Fl$), then each vertex will have an oriented path to the origin and the collection of oriented edges forms a tree.
A geodesic of FPPHE is defined as an oriented path in this tree.
Note that a vertex $x$ will be occupied by $\Fo$ if there is no seed on the oriented path from $x$ to $o$, whereas it will be occupied by $\Fl$ otherwise.
\end{Definition}

\paragraph{Idea of the proof.}
%
Fix the good path $\gpath$.
If (A.2) occurs for some $y\in \gpath$, then there must be a vertex $x\in \gpath$, such that
\begin{equation}\label{eq:find-x}
\begin{split}
 & d_\T (o,x)<d_\T(o,y), \\
& \text{there is a geodesic of FPPHE from $N(x)$ to $N(y)$}, \quad \text{ and }\\
& \text{a geodesic of FPPHE between }N(x)\text{ and }N(y)\text{ passes through }\s.
\end{split}
\end{equation}
If there is more than one such $x$, we consider the one which is closest to $y$.
\begin{Remark}
Note that $y$ cannot coincide with the root $o$.
In fact, assumption (A.1) rules out the possibility that $\Fo$ finds a seed 
before having completely occupied $N(o)$.
More precisely, since $\varepsilon$ is much smaller than $\beta$, by the time $\Fo$ reaches a  seed, the set $N(o)$ will be completely occupied by $\Fo$, giving no chance to $\Fl$ to ever reach it.
The same reasoning shows that in general $x\neq y$.
\end{Remark}
We will show that if (A.2) occurs, then by definition the cylinder $\cyl^{(\varepsilon d_{\T}(x,y)\r)}_{x,y}$ has to be bad, which contradicts (A.1), as $x,y\in \gpath$.

Now we are ready to give a proof of Theorem \ref{thm:survival_FPP_1}.
\begin{proof}[Proof of Theorem \ref{thm:survival_FPP_1}]
First, since Lemma \ref{lemma:typical-FPP} holds for some values  $\cin\in (0,1)$ small enough, and $\cout>0$ large enough, we can safely  assume that
$\cout > \alpha \max \{\lambda^{-1}, \lambda\}=\alpha \frac{\max\{1, \lambda\}}{\min\{1, \lambda\}}. 
$
Moreover, we choose $\varepsilon'=\varepsilon'(\alpha, \cout, \lambda)>0$ so small, that for all $\varepsilon<\varepsilon' $ we have 
\begin{equation}\label{eq:assumpt-eps}
\varepsilon<\min \left \{\frac{1}{\alpha}, \left (\frac{1}{2\alpha\cout \max \left \{ 1, \lambda\right \} }\right )^2\right \}.
\end{equation}
Recall the definition of $\cyl\bigl (N(x)\bigr )$ from \eqref{eq:CNx-CNy}.
An immediate consequence of Lemma \ref{lemma:close-geodesics} is that
\[
\begin{split}
& \forall \, \o\in \partial N(x), \, \ent\in \partial N(y),\text{ all }  \gamma_{\o, \ent}\in \Gamma_{\o, \ent} \text{ are contained inside }
\cyl^{(\varepsilon \r)}_{x,y}\cup \cyl (N(x))\cup \cyl (N(y) ).
\end{split}
\]

Now suppose that assumptions (A.1) and (A.2) both hold.
From now on, we will denote by $y$ the vertex in $ \gpath$ that satisfies Assumption (A.2) and by $x$ the one satisfying \eqref{eq:find-x}, and let $\j=d_{\T}(x,y)$.
Observe that \eqref{eq:assumpt-eps} together with the fact that $d_G(x,y)\in [\alpha^{-1}\j\r, \alpha\j\r]$ imply
\[
\frac{d_G(x,y)}{\cout\max \left \{ 1, \lambda\right \}}\geq \frac{\alpha^{-1}}{\cout \max \left \{ 1, \lambda\right \}}\j\r \geq 2\varepsilon^{1/2} \j\r >\varepsilon^{1/2} \j\r.
\]
Using a similar reasoning we obtain
\[
\frac{d_G(x,y)}{\cin \min\{1, \lambda\}}\leq  \frac{\alpha\j\r }{\cin \min\{1, \lambda\}}\leq \frac{\alpha \max \{1, \lambda\} \j\r }{\cin \min\{1, \lambda\}} \leq \frac{\cout }{\cin}\j\r \leq 4\frac{\cout}{\cin^2}\j\r.
\]
We will split the remaining part of the proof of Theorem \ref{thm:survival_FPP_1} into two cases, according to whether $\j$ is small or large.
The value that separates the two cases is an arbitrary value which simplifies the computations, but has no interpretation in the description of the model.

\paragraph{First case: 
$1 \leq \j\leq \left \lceil  \min \{\lambda, \lambda^{-1}\}\frac{\cin}{\cout}\frac{\beta}{\alpha} - 3 - \varepsilon \right \rceil $.}
In this case we will show that before $\Fo$ can move from $N(x)$ to activate a seed $\s$, it already goes from $N(x)$ to $N(y)$.
From Assumption (A.1) it follows that the cylinder $\cyl_{x,y}^{(\varepsilon \j \r)} $ is good, and 
thus any path (in particular a geodesic of $\Fo$) from any fixed vertex $\o \in N(x)$ to a seed $\s$ takes time bounded from below by
\begin{equation}\label{eq:min-1st-case}
\frac{\beta}{\max\{1, \lambda\}\cout}\r.
\end{equation}
Now we consider an upper bound for any $\Fo$ geodesic that passes close to the graph geodesics.
Note that by definition
$\min \{\lambda, \lambda^{-1}\}=\frac{\min\{1,\lambda\}}{\max \{1,\lambda\}}.$
Once again we exploit Proposition \ref{prop:ohshika-papad} and Lemma \ref{lemma:close-geodesics}, which guarantee that the $d_\T$-geodesic $\gamma_{\T}(x,y)$ is contained in $\cyl_{x,y}^{(\varepsilon \r)}$.
The maximum distance between any vertex of $N(x)$ and any vertex of $N(y)$ is at most 
\[
\alpha \j \r +2 \left \lceil \alpha \frac{\r}{2}\right \rceil +  \varepsilon \r \leq \alpha \j \r + \alpha \r+2+  \varepsilon \r\leq \alpha\j\r +2\alpha \r +\varepsilon \r,
\]
and using the range limitation of $\j$ we obtain
\[
\alpha \j \r +2\alpha \r + \varepsilon \r \stackrel{\text{(1st case)}}{\leq} \alpha\left ( \min \left \{ \lambda, \lambda^{-1}\right \}\frac{\cin}{\cout} \frac{\beta}{\alpha} -2-\varepsilon\right )\r +2\alpha \r + \varepsilon \r
\leq  \alpha\frac{\min\{1,\lambda\}}{\max \{1,\lambda\}}\frac{\cin}{\cout} \frac{\beta}{\alpha} \r.
\]
Otherwise stated, we have
\[
\sup_{\o \in N(x), \ent \in N(y)} d_G(\o ,\ent )\leq  \frac{\min\{1,\lambda\}}{\max \{1,\lambda\}}\frac{\cin}{\cout} \beta \r.
\]
Now we consider a path from $\o$ to $\ent$ inside $\cyl_{x,y}^{(\varepsilon \r)}\cup N(x)\cup N(y)$ as a concatenation of shorter  paths of length between $\sqrt{\varepsilon}\cout \r $ and $4 \frac{\cout ^2}{\cin^2}\r$ which we refer to as \emph{sub-paths}.
Note that these sub-paths are \emph{good} in the sense of part (b) in the definition of $\mathcal{G}_1(x,y;\r)$, cf.\ \eqref{eq:def-G-hat}.
Hence, their passage time is at most their length multiplied by $(\cin \min\{1, \lambda\})^{-1}$.
Thus their concatenation will give that disregarding the interaction with $\Fl$, $\Fo$ goes from $\o$ to $\ent$ in time at most
\[
\frac{1}{\cin \min\{1,\lambda\}} \frac{\min\{1,\lambda\}}{\max \{1,\lambda\}}\frac{\cin}{\cout} \beta \r
= \frac{1}{\max\{1, \lambda\}\cout} \beta \r.
\]
Thus, by comparing this result with what we found in \eqref{eq:min-1st-case}, we see that this implies that $\Fo$ manages to occupy $ N(y)$ before it manages to activate a seed.
Thus in this case the proof is concluded.

\paragraph{Second case: $\j > \left \lceil \min \{\lambda, \lambda^{-1}\}\frac{\cin}{\cout}\frac{\beta }{\alpha}- 3 - \varepsilon \right \rceil $.}
It is easy to verify from 
\eqref{eq:def-BETA} that this necessarily implies $\j\geq 3$.
This case corresponds to the situation where the path leaves from $N(x)$, stays completely outside of at least $\j-2$ sets $N(z)$ for some $z\in\gpath$ such that $d_G(o,x)<d_G(o,z)<d_G(o,y)$, and subsequently enters $\cyl^{(\varepsilon \r/2)}_{o,\infty}$ at $N(y)$.

We start by computing an upper bound on the time of all quickest geodesic paths (with respect to FPPHE, recall Definition \ref{def:FPPHE-geod}).
A geodesic path starting at some vertex $w\in N(x)$ and ending at some vertex $w' \in N(y)$ is completely contained inside the set $\cyl^{(\varepsilon \r/2)}_{x,y}\cup \cyl (N(x))\cup \cyl (N(y) )$ by Lemma \ref{lemma:close-geodesics}, and thus it has length bounded from above by
\[
d_G(w,x)+d_G(x,y)+d_G(y,w')\leq 
\left (\frac{\varepsilon\r}{2}+\left \lceil \frac{\alpha \r}{2}\right \rceil \right )+\alpha \j\r +\left ( \frac{\varepsilon\r}{2}+\left \lceil \frac{\alpha \r}{2}\right \rceil \right ).
\] 
\nr{r:one-more}
For $\varepsilon>0$ satisfying \eqref{eq:assumpt-eps}, we can take $\ur{r:one-more}=\ur{r:one-more}(\varepsilon)>0$ so large that for all $\r\geq \ur{r:one-more}$ we have
$\varepsilon\r+2\leq \r.$
Thus for all $\j\geq 1$ (and therefore for all $\j \geq \lceil 2\alpha (1+\alpha)\rceil $) we obtain 
\[
\varepsilon\r+\alpha \j\r +2\left \lceil \frac{\alpha \r}{2}\right \rceil \leq \varepsilon\r+\alpha\j\r +(\alpha\r+2)\leq 4\frac{\alpha}{\cin}\j\r.  
\]
Stated in a more clear way, this is saying that
\begin{equation}\label{eq:all-scale-1-good}
\sup_{\o \in N(x), \ent\in N(y)}d_G(\o , \ent)\leq 4\frac{ \alpha}{\cin}\j\r.
\end{equation}
Recall the definition of the event $\mathcal{G}_1(x,y;\j\r)$, then we proceed as in the first case, namely we consider sub-paths inside $\cyl_{x,y}^{\varepsilon\r/2}\cup N(x)\cup N(y)$.
Then, it follows that disregarding the interactions with $\Fl$, $\Fo$ goes from $N(x)$ to $N(y)$ in time at most
\begin{equation}\label{eq:up-bound-geod}
\frac{1 }{\min \{1, \lambda\}\cin }\left ( \varepsilon\r+2+\alpha \r+\alpha \j\r \right )\leq 4\frac{ \alpha}{\min \{1, \lambda\}\cin^2}\j\r.
\end{equation}
By construction, this is an upper bound on the time needed by $\Fo$ started in $N(x)$ to completely occupy the set $\cyl^{(\varepsilon\r/2)}_{x,y}\cup  N(x) \cup N(y)$.

Now we proceed with a lower bound on the time needed to any geodesic path of FPPHE that avoids $\j-2$ regions $N(\cdot)$ to go from $N(x)$ to $N(y)$.
By Lemma \ref{lemma:close-geodesics}, for any $\o\in N(x)$ and $\ent\in N(y)$ we have that $\o$ and $\ent$ are connected by a geodesic that is completely contained inside $\cyl^{(\varepsilon \r/2)}_{x,y}\cup \cyl^{(\varepsilon \r/2)}_{\o, x} \cup \cyl^{(\varepsilon \r/2)}_{y, \ent} $.
Now, relation \eqref{eq:bilipschitz-tree} implies that $\alpha^{-1}\j \r$ is the smallest possible $d_G$-distance between $x$ and $y$.
Thus, by construction, the $d_G$-distance between $N(x)$ and $N(y)$ is bounded from below by 
\[
\alpha^{-1}\j \r -2\left \lceil \alpha\frac{\r}{2}\right \rceil-\varepsilon \r\geq \alpha^{-1}\j \r -2\alpha\r-\varepsilon \r-2.
\]
\nr{r:new-one-1}
Define $\ur{r:new-one-1}:=\ur{r:new-one-1}(\alpha, \varepsilon):=2/(\alpha\varepsilon) $ and let $\r \geq \ur{r:new-one-1}$.
By using the lower bound on $\j $, we deduce 
\[
\begin{split}
\alpha^{-1}\j \r & -2\alpha\r-\varepsilon \r -2 \geq \alpha^{-1} \left ( \min \{\lambda, \lambda^{-1}\}\frac{\cin}{\cout}\frac{\beta }{\alpha}- 3 - \varepsilon\right ) \r -2\alpha\r-\varepsilon \r-2\\
& \stackrel{\eqref{eq:def-BETA} }{\geq }\alpha^{-1} \left [ \min \{\lambda , \lambda^{-1}\}\frac{\cin}{\cout}\left ( (6+\varepsilon)(1+\alpha)\alpha \frac{\cout}{\cin }\max\{\lambda,\lambda^{-1}\}\right ) - 3 - \varepsilon\right ] \r -2\alpha\r-\varepsilon \r-2\\
& = \alpha^{-1} \left ( (6+\varepsilon)(1+\alpha)\alpha - 3 - \varepsilon\right ) \r -2\alpha\r-\varepsilon \r-2\\
& \geq \left [(6+6\alpha +\varepsilon+\varepsilon \alpha)-3 \alpha^{-1}-\varepsilon\alpha^{-1}-2\alpha-\varepsilon\right ]\r -2\\
& \stackrel{\varepsilon<1 }{\geq }\left [ 6+4 \alpha + \varepsilon\alpha-4 \alpha^{-1}\right ]\r-2 \ \stackrel{\alpha \geq 1}{\geq } \ (6 + \varepsilon\alpha)\r -2 \ \stackrel{ \r \geq \ur{r:new-one-1}}{ \geq } \  6\r.
\end{split}
\]
\nr{r:maybe-last} 
\nr{r:another-1}
Now, by simply choosing $\ur{r:maybe-last}=\ur{r:maybe-last}(\delta, \varepsilon)$ so large that for all $\r \geq \ur{r:maybe-last}$ we have
$6 \r \geq \max \{50 , 50 \delta\}$ and $\varepsilon \r\geq  9\delta.$
It follows that $d_G(\o,\ent) \geq \max \{50 , 50 \delta\}$, and therefore we are in the conditions to apply Proposition \ref{prop:detour_cylinders} with $L=\varepsilon \r$.
As a consequence, the length of the detour that the geodesic of FPPHE makes is bounded from below by $(\alpha^{-1}\j \r -2\alpha \r-\varepsilon \r-2)\delta 2^{\varepsilon\r/(4\delta)}$ which, for $\r$ large enough, is much larger than 
$\frac{4\cout^2 \j \r}{\cin^2}$. Then, since $\cyl_{x,y}^{(\varepsilon \j \r)}$ is good, any path from $N(x)$ of length $\frac{4\cout^2 \j \r}{\cin^2}$ incurs a passage time of at least 
$$
   \frac{1}{\max\{1,\lambda\}}\left(\frac{4\cout^2 \j \r}{\cin^2}\right) \frac{1}{\cout} > \frac{4\alpha \j \r}{\min\{1,\lambda\} \cin^2}.
$$
So, by comparing with \eqref{eq:up-bound-geod}, it becomes clear that the passage time along this path takes longer than the time $\Fo$ takes to go from anywhere in $N(x)$ to anywhere in $N(y)$, establishing the desired contradiction and completing the proof of Theorem \ref{thm:survival_FPP_1}.
\end{proof}

\section{Survival of $\Fl$ (Proof of Theorem \ref{thm:survival_FPP_lambda})}\label{sect:survival_FPP_lamda}
%
Here we show that with positive probability $\Fl$ will be able to ``escape'' towards infinity, before $\Fo$ can reach it and surround it.
We will show that the event ``$\Fl$ survives indefinitely'' occurs with probability one.

\subsection{Step 1: Construction of a quasi-geodesic infinite ray}
Let $t>0$ be a fixed integer, and start $\Fo$ from the origin $o$.
Since we are dealing with the two processes $\Fo$ and $\Fl$, we need to define the set of vertices of the graph that have been reached by either process by time $t$.
We set
\[
\Aol_t:=\{v\in V(G) \ : \ v\text{ has been occupied by }\Fo \text{ or }\Fl \text{ by time }t\},
\]
where if $v$ hosts a seed then it is included in $\Aol_t$ only if the seed has already been activated by time $t$.
Recall that $\partial \Aol_t$ denotes the internal boundary of $\Aol_t$.
At first sight, it is tempting to believe that, in some cases (for example when $G$ is vertex-transitive), 
one could have that for every $t\geq 0$ the  vertex $x \in \partial \Aol_t$ that is the furthest from $o$ lies on an infinite geodesic ray (w.r.t.\ the graph metric) which intersects $\Aol_t$ only at $x$.
But this statement is not necessarily true even when $G$ is vertex-transitive and $\Aol_t$ is close to a ball;
the interested reader is referred to \cite{BrieusselGournay} and references therein.
%
We now prove a technical result, based on \cite[Lemma 2]{Bogopolski}, that allows us to overcome this problem.
\begin{Lemma}\label{lemma:bog}
   Let $G$ be a hyperbolic graph.
   Fix two integers $\sigma\geq 8\delta+1$ and $s\geq \sigma$, and take any vertex $x\in V(G)$ of distance at least $\sigma$ from $o$.
   Let $x_1\in \gamma_{x,o}$ be the vertex at distance $\sigma$ from $x$.
   Then, any vertex $y\in B_G(x,s)$ such that
   \begin{equation}\label{eq:reqy}
      d_G(o,y) < d_G(o,x) +d_G(x,y)-2\sigma
   \end{equation}
   belongs to $B_G(x_1,s-\sigma+4\delta)$.
\end{Lemma}

Before proceeding to the proof of the above lemma, we will need another technical result.
\begin{Lemma}\label{lemma:defhyp}
   Let $G$ be a hyperbolic graph. 
   Take any geodesic triangle $\{a,b,c\}$ in $G$.
   Then, for all vertices $B,C\in V(G)$ with $B\in \gamma_{a,b}$ and $C\in \gamma_{a,c}$ such that
   \begin{equation}\label{eq:bog}
   d_G(a,B)=d_G(a,C)\leq \frac{1}{2} \left ( d_G(a,b)+d_G(a,c)-d_G(b,c)\right ), 
   \end{equation}
   we have that 
   $d_G(B,C)\leq 4 \delta.$
\end{Lemma}
\begin{proof}
   We have to split the proof of this fact into two cases, which are drawn in Figure \ref{fig:bog}:
   \begin{itemize}
   \item[(i)] The vertices $B$ and $C$ satisfy \eqref{eq:bog} and 
   $
   d_G( B, \gamma_{b,c})> \delta  \text{ or }  d_G( C, \gamma_{b,c})> \delta .
   $
   \item[(ii)] The vertices $B$ and $C$ satisfy \eqref{eq:bog} and 
   $
   d_G( B, \gamma_{b,c})\leq \delta  \text{ and }  d_G( C, \gamma_{b,c})\leq \delta .
   $
   \end{itemize}
   
   \begin{figure}[h!]\label{fig:bog}
   \begin{center}
   \includegraphics[scale=.9]{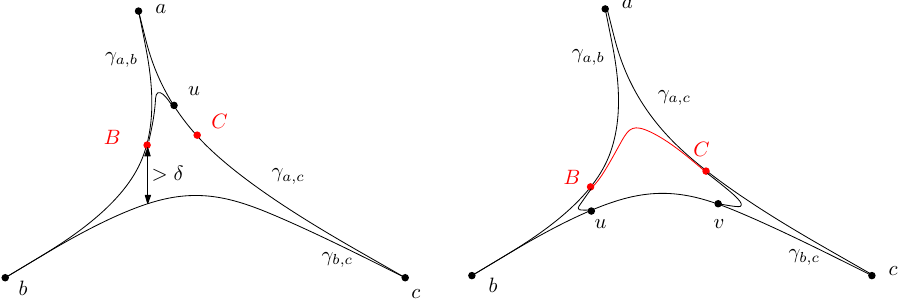}
   \caption{A representation of case (i) (on the left) and case (ii) (on the right) of Lemma \ref{lemma:defhyp}.}
   \end{center}
   \end{figure}
   
   \emph{Case (i)}. Assume that $d_G( B, \gamma_{b,c})> \delta $ (the proof if $d_G( C, \gamma_{b,c})> \delta$ is symmetric).
   Since $G$ is $\delta$-hyperbolic, then there must be a vertex $u\in \gamma_{a,c}$ such that $d_G(u, B)\leq \delta$.
   The triangle inequality implies that 
   \begin{equation}\label{eq:bog-aux1}
    d_G( a,C)\stackrel{ \eqref{eq:bog}}{=}d_G( a,B)\leq d_G(a,u)+d_G(u, B)\leq d_G(a,u)+\delta.
   \end{equation}
   If $d_G(a,u)\leq d_G(a,C)$, then
   \[
   d_G(B,C)\leq d_G(B,u)+d_G(u,C)\leq \delta +(d_G(a,C)-d_G(a,u))\stackrel{\eqref{eq:bog-aux1} }{\leq }2\delta.
   \]
   If $d_G(a,C)\leq d_G(a,u) $, then
   \[
   \begin{split}
   d_G(B,C) & \leq d_G(B,u)+d_G(u,C) \leq \delta +d_G(u,C)\\
   & = \delta + \left (d_G(a,u)-d_G(a,C)\right ) \stackrel{\eqref{eq:bog} }{= }\delta + d_G(a,u)-d_G(a,B) \\
   & \leq \delta +\left ( d_G(a,B)+d_G(B,u)\right )-d_G(a,B) \leq 2\delta.
   \end{split}
   \]
   Thus the lemma holds in case (i).
   
   \noindent
   \emph{Case (ii)}. By assumption, there are vertices $u, v\in \gamma_{b,c}$ such that
   $d_G(B,u)\leq \delta$ and $d_G(C,v)\leq \delta$.
   First we consider the case $d_G(b,u)\leq d_G(b,v)$, for which  $d_G(b,c) =d_G(b,u)+d_G(u,v)+d_G(v,c)$.
   Relation \eqref{eq:bog} clearly implies that
   \[
   d_G(a,B)+d_G(a,C)=2d_G(a,B)\leq  d_G(a,b)+d_G(a,c)-d_G(b,c), 
   \]
   that is
   \begin{equation}\label{eq:bog-aux2}
   d_G(b,c)\leq d_G(a,b)+d_G(a,c)- (d_G(a,B)+d_G(a,C)).
   \end{equation}
   Now, since  $u,v\in \gamma_{b,c}$ and $B\in \gamma_{a,b} $ and $C \in \gamma_{a,c}$ we have
   \[
   \begin{split}
   d_G(b,c) & =d_G(b,u)+d_G(u,v)+d_G(v,c)
   \stackrel{\eqref{eq:bog-aux2} }{\leq }d_G(a,b)+d_G(a,c)- (d_G(a,B)+d_G(a,C))\\
   & \leq d_G(a,b)+d_G(a,c)-\left ( d_G(a,b)-d_G(b,B)\right )- \left (d_G(a,c)-d_G(c,C) \right )\\
   & = d_G(b,B) + d_G(c,C) .
   \end{split}
   \]
   Again using the definition of $u$ and $v$ we find
   \[
   d_G(b,B)\leq d_G(b,u)+d_G(u,B) \ \text{ and } \ d_G(c,C)\leq d_G(c,v)+d_G(v,C).
   \]
   Therefore we obtain:
   \[
   d_G(b,c) =d_G(b,u)+d_G(u,v)+d_G(v,c) \leq d_G(b,u)+d_G(u,B) + d_G(c,v)+d_G(v,C),
   \]
   which in turn implies 
   $
   d_G(u,v)\leq d_G(u,B) +d_G(v,C)\leq 2\delta.
   $
   At this point we obtain
   \[
   d_G(B,C)\leq d_G(u,B) +d_G(v,C)+d_G(u,v)\leq 4\delta,
   \]
   showing that the lemma holds in case (ii) when $d_G(b,u)\leq d_G(b,v)$.
   
   If $d_G(b,v)< d_G(b,u)$ then we have two possible situations. 
   Either $d_G(B,v)\leq \delta$, in which case $d_G(B,C)\leq d_G(B,v)+d_G(v,C)\leq 2\delta$, or $d_G(B,v)> \delta$.
   In the latter case we observe that if $d_G(v,u)\leq 2\delta$ we are fine since $d_G(B,C)\leq 4\delta$ follows as above.
   So, the only problematic case left to analyze is
   \begin{equation}\label{eq:delta2delta}
   d_G(B,v)> \delta \text{ and } d_G(v,u) > 2\delta.
   \end{equation}
   In this situation we see that
   $
   d_G(b,u)=d_G(b,v)+d_G(v,u)\stackrel{\eqref{eq:delta2delta} }{>}d_G(b,v)+2\delta.
   $
   On the other hand, by the triangle inequality
   \[
   d_G(b,u)\leq d_G(b,B)+d_G(B,u)\leq d_G(b,B)+\delta,
   \]
   and these two inequalities together imply
   \begin{equation}\label{eq:dbb-delta}
   d_G(b,v)<d_G(b,B)-\delta.
   \end{equation}
   Now, considering the triangle $\{a,b,v\} $  we have
   \begin{equation}\label{eq:dav-low}
   d_G(a,v)\geq d_G(a,b)-d_G(b,v)\stackrel{\eqref{eq:dbb-delta} }{>}d_G(a,b)-d_G(b,B)+\delta=d_G(a,B)+\delta.
   \end{equation}
   Using the triangle inequality we see that
   \begin{equation}\label{eq:dav-up}
   d_G(a,v)\leq d_G(a,C)+d_G(C,v)\leq d_G(a,C)+\delta = d_G(a,B)+\delta.
   \end{equation}
   Comparing \eqref{eq:dav-low} and \eqref{eq:dav-up} we obtain a contradiction, showing that the two conditions in \eqref{eq:delta2delta} are not compatible in the current setting, meaning that at least one between $d_G(B,v)\leq  \delta$ and $d_G(v,u) \leq 2\delta$ has to hold.
   This completes the proof of the lemma in Case (ii).
\end{proof}

\begin{proof}[Proof of Lemma~\ref{lemma:bog}]
   Given Lemma~\ref{lemma:defhyp}, the proof is a simple generalization of that of \cite[Lemma 2]{Bogopolski}, but we include it for the sake of completeness.
   Let $x$ be any fixed vertex and $s$ as in the statement.
   Take an element $y\in B_G(x, s)$ and suppose that $y$ satisfies~\eqref{eq:reqy}, namely that
   \begin{equation}\label{eq:no-bog}
   d_G(o,y) < d_G(o,x) +d_G(x,y)-2\sigma.
   \end{equation}
   Recall $x_1\in \gamma_{o,x}$ in the statement and let $z\in \gamma_{x,y}$ be such that
   $
   d_G(x,z)=d_G(x,x_1)=\sigma.
   $
   By \eqref{eq:no-bog} it follows that
   \[
   \frac{1}{2}\left ( d_G(o,x) +d_G(x,y)-d_G(o,y) \right )>\sigma =d_G(x,x_1)=d_G(x,z).
   \]
   Using Lemma~\ref{lemma:defhyp} with $a=x, b=o, c=y, B=x_1 $ and $C=z $,  we obtain that $d_G(x_1,z)\leq 4\delta$.
   At this point we have
   \[
   \begin{split}
   d_G(x_1,y) \leq d_G(x_1,z)+d_G(z,y)=4\delta+(d_G(x,y)-d_G(x,z))
   = 4\delta +s-\sigma.
   \end{split}
   \]
   Thus $y\in B_G(x_1,s-\sigma+4\delta)$.
%
%
\end{proof}
\begin{Remark}\label{rem:after-bog}
Note that from the proof of Lemma \ref{lemma:bog} and the fact that $\sigma\geq 8\delta+1$ we can deduce that there is at least one element $y\in \partial B_G(x,s)$ such that $d_G(o,y) \geq d_G(o,x) +d_G(x,y)-(16 \delta+2)=d_G(o,x) +s-(16 \delta+2)$.
In fact, all elements that satisfy \eqref{eq:reqy} are contained inside $B_G(x_1,s-\sigma+4\delta)$, but $\partial B_G(x,s)\not \subset B_G(x_1,s-\sigma+4\delta) $.
%
As we mentioned in the introduction, our results do hold for trees in a trivial way, hence we assume directly that $\delta\geq 1$.
Thus, the above implies that there is at least one element $y\in \partial B_G(x,s)$ such that 
\[
d_G(o,y) \geq d_G(o,x) +s-{18} \delta.
\]
\end{Remark}

Now we will use Lemma \ref{lemma:bog} and Remark \ref{rem:after-bog} inductively in order to find an infinite ray that will progressively go far away from the origin.
We will henceforth assume that $G$ is a vertex-transitive graph. 
%
This is needed to obtain an upper bound that is independent of $x$ on the number of vertices $y$ satisfying the requirement of Lemma~\ref{lemma:bog}. Such bound is at most $|B_G(x_1,s-4\delta)|$.

Fix  a finite value $R_1>0$ that is large enough with respect to $\delta$, $\lambda$, $\cin$ and $\cout$; in particular, $R_1$ will be large enough with respect to all parameters except for $\mu$. 
Let $\tau_0\geq 0$ be an arbitrary time in the evolution of FPPHE.
\begin{Lemma}\label{lemma:dist-agg}
   For any $t>0$, choose a vertex $x_t \in \partial \Aol_t$ so that
   \[
      d_G(o,x_t)=\max_{z\in \partial \Aol_t}d_G(o,z).
   \]
   For any given shape of the aggregate $\Aol_{\tau_0}$ we can find an infinite ray $\overline{\gamma} $ started at $o$ for which there is a sequence of vertices 
   $w_{\tau_0}^{(0)}, w_{\tau_0}^{(1)}, w_{\tau_0}^{(2)},\ldots \in \overline{\gamma}$ such that $w_{\tau_0}^{(0)}=x_{\tau_0}$ and, for all $i\geq 1$, we have
   \[
      \sum_{k=1}^{2i} R_1^k \leq d_G(\Aol_{\tau_0}, w_{\tau_0}^{(i)}) \leq {18}\delta i + \sum_{k=1}^{2i} R_1^k.
   \]
   Moreover, for all $i\geq 1$ it also holds that
   \[
      d_G\left ( w_{\tau_0}^{(i-1)},  w_{\tau_0}^{(i)}\right ) = R_1^{2i}+R_1^{2i-1} +{18} \delta
   \]
   and 
   \[
      d_G\left ( o, w_{\tau_0}^{(i)}\right) \geq d_G\left(o,w_{\tau_0}^{(i-1)}\right ) + R_1^{2i}+R_1^{2i-1}.
   \]
\end{Lemma}
\begin{Remark}
Note that Lemma \ref{lemma:dist-agg} would hold for any deterministic connected set that contains the origin. 
However, for sake of clarity, we wrote it only referred to the aggregate because we need to apply it to the set $\Aol_t$ defined as above.
\end{Remark}
\begin{proof}
   We will find the sequence of vertices inductively, and we start by setting a value $S_1:=R_1^2+R_1+{18}\delta$.
   From Lemma \ref{lemma:bog} (together with Remark \ref{rem:after-bog}) we find that there is an element $w\in \partial B_G(x_{\tau_0},S_1)$ such that 
   \[
   \begin{split}
   d_G(o,x_{\tau_0})+S_1 \geq d_G(o,w) & \geq d_G(o,x_{\tau_0}) +d_G(x_{\tau_0},w)-{18} \delta\\
   & =d_G(o,x_{\tau_0}) +S_1-{18} \delta,
   \end{split}
   \]
   where the last equality follows from the fact that $w$ is on the boundary of the ball centered at $x_{\tau_0}$.
   By the definition of $S_1$ we have
   \begin{equation}\label{eq:d(o,w1)}
   d_G(o,x_{\tau_0})+R_1^2+R_1+{18}\delta \geq d_G(o,w)  \geq d_G(o,x_{\tau_0}) +R_1^2+R_1,
   \end{equation}
   and by our definition  of $x_{\tau_0}$ we have that 
   \[
   S_1=d_G(w, x_{\tau_0})\geq d_G(w, \Aol_{\tau_0}) \geq d_G(o,w)-d_G(o,x_{\tau_0})\stackrel{\eqref{eq:d(o,w1)} }{\geq }R_1^2+R_1.
   \]
   Such vertex $w$ is the first element of the sought sequence, thus we call it $w_{\tau_0}^{(1)}$.
   The first segment of the ray $\overline{\gamma}$ consists of the concatenation of an arbitrary geodesic from $0$ to $x_{\tau_0}$ and an arbitrary geodesic $\gamma_{x_{\tau_0}, w_{\tau_0}^{(1)}}$.
   
   Now we proceed inductively.
   Suppose that we have found all the first $k\geq 1$ elements of the sought sequence, that is, we have defined $\left \{ w_{\tau_0}^{(i)}\right \}_{i=1}^k$, and now we want to define the element $w_{\tau_0}^{(k+1)}$.
   Set 
   $$
      S_{k+1}:= R_1^{2(k+1)}+R_1^{2(k+1)-1}+{18}\delta,
   $$ 
   and observe that by Lemma \ref{lemma:bog} and Remark \ref{rem:after-bog} we have that there exists an element $w\in \partial B_G(w_{\tau_0}^{(k)},S_{k+1})$ such that 
   \[
   \begin{split}
   d_G(o,w_{\tau_0}^{(k)})+S_{k+1} \geq d_G(o,w) 
   \geq
   d_G(o,w_{\tau_0}^{(k)}) +S_{k+1}-{18} \delta.
   \end{split}
   \]
   We will set $w_{\tau_0}^{(k+1)}=w$. Thus, from the definition of $w_{\tau_0}^{(k)}$ and setting $w_{\tau_0}^{(0)}=x_{\tau_0}$, it follows that
   \[
      d_G(w, \Aol_{\tau_0})
      \geq d_G(o,w)-d_G(o,x_{\tau_0}) 
      = \sum_{i=1}^{k+1} \left(d_G(o,w_{\tau_0}^{(i)})-d_G(o,w_{\tau_0}^{(i-1)})\right).
   \]
   Hence, we obtain the bound
   \[
      d_G(w, \Aol_{\tau_0})
      \geq \sum_{i=1}^{k+1}\left(S_{i}-{18}\delta\right)
      =\sum_{i=1}^{2(k+1)} R_1^i.
   \]
   For the upper bound, we obtain
   \[
      d_G(w, \Aol_{\tau_0})
      \leq d_G(w, x_{\tau_0})
      \leq \sum_{i=1}^{k+1} d_G(w_{\tau_0}^{(i)},w_{\tau_0}^{(i-1)})
      ={18}\delta (k+1) + \sum_{i=1}^{2(k+1)} R_1^i.
   \]
   Such vertex $w$ is the next element of the sought sequence, thus we denote it by $w_{\tau_0}^{(k+1)}$, and the next segment of the ray $\overline{\gamma}$ consists of an arbitrarily chosen geodesic $\gamma_{w_{\tau_0}^{(k)}, w_{\tau_0}^{(k+1)}}$.
   
   Simply using Remark \ref{rem:after-bog} leads to
   $
   d_G \left ( w_{\tau_0}^{(k)}, w_{\tau_0}^{(k+1)}\right ) =S_{k+1},
   $
   which concludes the proof.
\end{proof}
The main idea of the proof of Theorem \ref{thm:survival_FPP_lambda} is the following.
We take a sequence of large balls $B_{\tau_0}^{(1)},B_{\tau_0}^{(2)},\ldots$ (of increasing radii) 
centered at the vertices $\{w_{\tau_0}^{(k)}\}_{k\geq 1}$, and condition on the event that the first one $B_{\tau_0}^{(1)}$ is completely full of seeds (the ball has finite volume, hence this event has positive probability).
This will force $\Fo$ to make a long detour around it to occupy any vertex that is further away in $\overline{\gamma}$, which will require a very long time.
Meanwhile, $\Fl$ started from the seeds will spread along $\overline{\gamma}$ and start occupying the subsequent balls.
The rest of the proof will consist in proving that for any such attempt, with positive probability, $\Fl$ will succeed in occupying the infinite sequence of balls. 
We proceed now with the next step, which consists in defining the balls that should be occupied by $\Fl$ and avoided by $\Fo$.
For a graphic representation refer to Figure \ref{fig:proof-lambda-2}.

\begin{figure}[h!]
\begin{center}
\hspace{\stretch{1}}\includegraphics[scale=0.4]{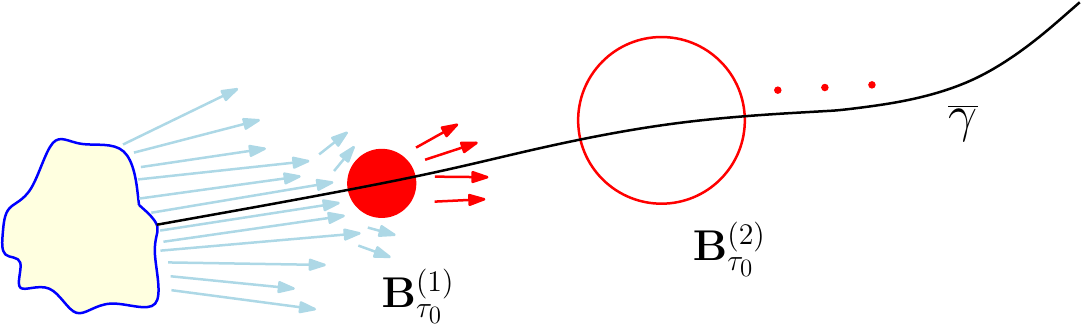}
\hspace{\stretch{1}}\includegraphics[scale=0.4]{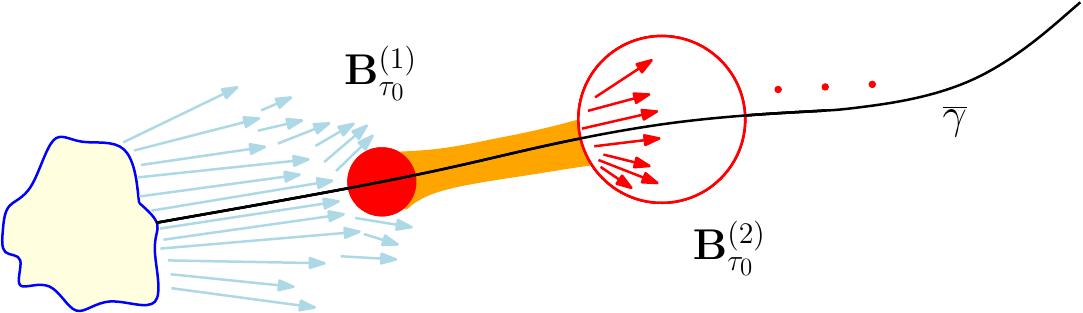}\hspace{\stretch{1}}
\caption{\emph{Left:} The light blue arrows represent paths of FPP started from $\Aol_{\tau_0}$ which eventually activate the seeds in $\B_{\tau_0}^{(1)}$.
While the geodesic of $\Fo$ started in $\Aol_{\tau_0}$ (light blue arrows) try to go around the ball full of seeds, paths of $\Fl$ (red arrows) have already occupied and exited $\B_{\tau_0}^{(1)}$.
\emph{Right:} While the paths of FPP started in $\Aol_{\tau_0}$ (blue arrows) are still trying to surround the first ball, paths of $\Fl$ have already filled up a long region (colored in orange) and started to reach the following ball.
This will be useful to show that (with positive probability) $\Fo$ is too slow to ever surround $\Fl$.}
\label{fig:proof-lambda-2}
\end{center}
\end{figure}


\subsection{Step 2: Construction of the balls}
\subsubsection{First ball}\label{sect:ball-1}
Recall the value of $R_1$ and the sequence of vertices $\{w_{\tau_0}^{(k)}\}_{k\geq 1} \in \overline{\gamma}$ defined above.
From Lemma \ref{lemma:dist-agg} it follows that
\begin{equation}\label{eq:d(w,A)}
R_1+R_1^2+{18} \delta\geq 
d_G(\Aol_{\tau_0}, w_{\tau_0}^{(1)})
\geq R_1+R_1^2.
\end{equation}
Now we  place a ball of radius $R_1$ centered at $w_{\tau_0}^{(1)}$, and denote it by $\B_{\tau_0}^{(1)}$: this will be the ``first ball'' which we need in order to start the whole procedure.
Refer to Figure \ref{fig:proof-lambda-1}.
Note that with positive probability (bounded from below by $\mu^{\Delta^{R_1}}>0$), the ball $\B_{\tau_0}^{(1)}$ is completely filled with seeds, so we assume this event holds in the subsequent steps. 
To formalize this, we define the event
$$
   \mathcal{E}_{\tau_0}^{(1)}=\left\{B_{\tau_0}^{(1)}\setminus \{\mathrm{seeds}\} = \emptyset\right\}.
$$

\begin{Remark}
The upper bound on the distance in relation \eqref{eq:d(w,A)} will be useful later on in order to ensure that the ball $\B_{\tau_0}^{(1)}$ is activated quickly enough.
\end{Remark}

\begin{figure}[h!]
\begin{center}
\includegraphics[scale=0.4]{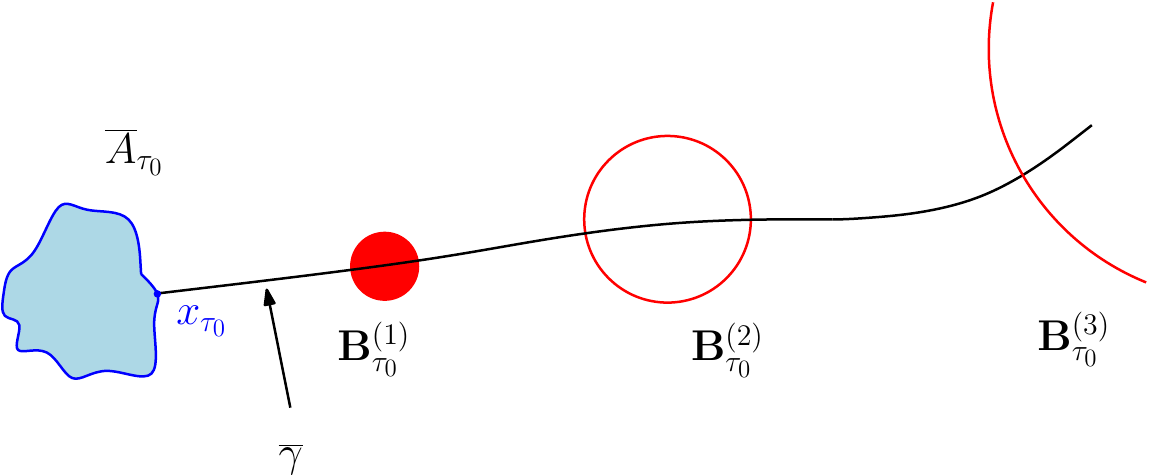}
\caption{The blue cluster is $\Aol_{\tau_0}$, and the solid red ball $B_{\tau_0}^{(1)}$ is the one full of seeds.}
\label{fig:proof-lambda-1}
\end{center}
\end{figure}

Now we proceed with the construction of the other balls using Lemma \ref{lemma:dist-agg} inductively.

\subsubsection{Second ball}\label{sect:ball-2andk}
Start by taking $w_{\tau_0}^{(2)}$ found in Lemma \ref{lemma:dist-agg}.
As we know, we have
$w_{\tau_0}^{(2)}\in \overline{\gamma}$, and $d_G(w_{\tau_0}^{(1)}, w_{\tau_0}^{(2)})= R_1^4+R_1^3+{18} \delta$.
From now on we will denote
\[
R_2:=R_1^2,
\quad\text{and set}\quad
\B_{\tau_0}^{(2)}:=B_G(w_{\tau_0}^{(2)}, R_2).
\]
We refer to Figure \ref{fig:proof-lambda-1}, and remark that from our definitions it follows that 
$$
   d_G(\B_{\tau_0}^{(1)}, \B_{\tau_0}^{(2)})=R_1^4+R_1^3-R_1^2-R_1+{18}\delta \leq 2 R_2 R_1^2.
$$

At this point, using $\delta$-thinness, we obtain the next result, which is illustrated in 
Figure \ref{fig:geo-balls}.
\begin{figure}[h!]
\begin{center}
\includegraphics[scale=0.4]{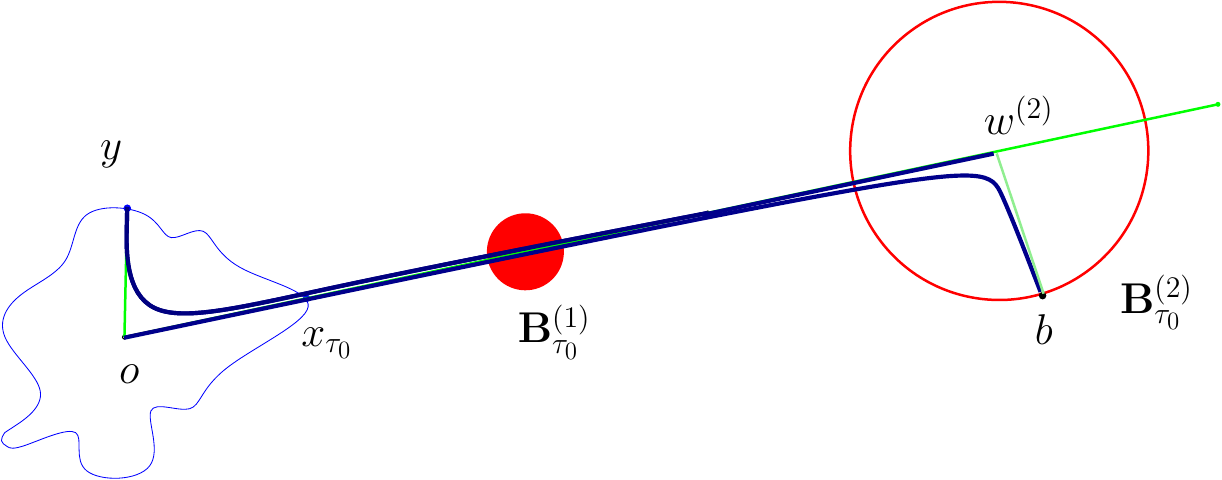}
\caption{A representation of the situation described in Lemma \ref{lemma:geodesics-in-ball}, where the white shape on the left is $\Aol_{\tau_0}$, and $x_{\tau_0}, y$ are as described above.
The light green segments represent geodesics, as well as the dark blue ones. We did not draw the geodesic(s) between $y$ and $b$, in order to avoid confusion.}
\label{fig:geo-balls}
\end{center}
\end{figure}
\begin{Lemma}\label{lemma:geodesics-in-ball}
   For all $y\in \partial \Aol_{\tau_0}$ and each vertex $b\in \B_{\tau_0}^{(2)}$, all geodesics $\gamma_{y,b}$ connecting $y$ with $b$ are such that
   $
   \gamma_{y,b}\cap B(w_{\tau_0}^{(1)}, 23\, \delta)\neq \emptyset.
   $
\end{Lemma}
\begin{proof}
Fix any $y\in \partial \Aol_{\tau_0}$ and any vertex $b\in \B_{\tau_0}^{(2)}$.
For this proof it suffices to choose $b\in \partial \B_{\tau_0}^{(2)}$, as every geodesic that enters $\B_{\tau_0}^{(2)}$ has to contain at least a vertex of $\partial \B_{\tau_0}^{(2)}$.
%
The core of this proof consists in showing that
\begin{equation}\label{eq:w,gamma}
d_G(w_{\tau_0}^{(1)}, \gamma_{o, w_{\tau_0}^{(2)}})\leq 21\delta.
\end{equation}
In fact, suppose that \eqref{eq:w,gamma} holds. 
By our choice of $R_1$ it follows that for the fixed $y\in \partial \Aol_{\tau_0}$ and $b\in \partial \B_{\tau_0}^{(2)}$
\[
d_G(\gamma_{o,y},w_{\tau_0}^{(1)})>25\delta, \quad d_G(w_{\tau_0}^{(1)}, \gamma_{w_{\tau_0}^{(2)},b})>25\delta, \quad d_G(w_{\tau_0}^{(2)}, y) \geq  d_G(w_{\tau_0}^{(2)}, \Aol_{\tau_0})>25\delta.
\]
Then, by $\delta$-thinness we obtain that
\[
d_G(w_{\tau_0}^{(1)}, \gamma_{y, b})\leq \delta +d_G(w_{\tau_0}^{(1)}, \gamma_{o, b}) \leq 2\delta + d_G(w_{\tau_0}^{(1)}, \gamma_{o, w_{\tau_0}^{(2)}})\stackrel{\eqref{eq:w,gamma} }{\leq }23\, \delta,
\]
as claimed.
Thus we proceed to show the validity of \eqref{eq:w,gamma}.
Consider the triangle $\{o,w_{\tau_0}^{(1)}, w_{\tau_0}^{(2)} \}$ (for a graphical representation see Figure \ref{fig:geo-balls}), and recall that by Lemma \ref{lemma:dist-agg} we know that
\begin{equation}\label{eq:conseq-lemma-bog}
d_G(o, w_{\tau_0}^{(2)})\geq d_G(o, w_{\tau_0}^{(1)}) + d_G(w_{\tau_0}^{(1)}, w_{\tau_0}^{(2)})-{18} \delta.
\end{equation}
Let $v\in \gamma_{o,w_{\tau_0}^{(1)}}$ denote a vertex such that
$
19\delta < d_G(w_{\tau_0}^{(1)}, v)\leq 20\delta.
$
If we show that 
$
d_G(v, \gamma_{o, w_{\tau_0}^{(2)}})\leq \delta,
$
then we would immediately deduce that
\[
d_G(w_{\tau_0}^{(1)}, \gamma_{o, w_{\tau_0}^{(2)}})\leq d_G(w_{\tau_0}^{(1)}, v) + d_G(v, \gamma_{o, w_{\tau_0}^{(2)}})\leq 21\delta,
\]
that is \eqref{eq:w,gamma}.
We will show this by contradiction; so from now on assume that $d_G(v, \gamma_{o, w_{\tau_0}^{(2)}})> \delta$.
This and $\delta$-thinness of the triangle $\{o,w_{\tau_0}^{(1)},w_{\tau_0}^{(2)}\}$ implies that there exists a vertex $v'\in \gamma_{w_{\tau_0}^{(1)}, w_{\tau_0}^{(2)}}$ such that
\begin{equation}\label{eq:vv'}
d_G(v,v')\leq \delta.
\end{equation}
Then we have that
\[
d_G(o , w_{\tau_0}^{(1)})=d_G(o,v)+d_G(v, w_{\tau_0}^{(1)})
\quad\text{and}\quad
%
d_G(w_{\tau_0}^{(1)}, w_{\tau_0}^{(2)})=d_G(w_{\tau_0}^{(1)}, v') + d_G(v', w_{\tau_0}^{(2)}).
\]
In particular, by using these relations in \eqref{eq:conseq-lemma-bog} we deduce that
\[
d_G(o, w_{\tau_0}^{(2)})\geq d_G(o,v)+d_G(v, w_{\tau_0}^{(1)}) + d_G(w_{\tau_0}^{(1)}, v') + d_G(v', w_{\tau_0}^{(2)})-{18} \delta.
\]
On the other hand, by the triangle inequality we always have
\[
d_G(o, w_{\tau_0}^{(2)}) \leq d_G(o,v)+ d_G(v,v')+d_G(v', w_{\tau_0}^{(2)})\stackrel{\eqref{eq:vv'} }{\leq} d_G(o,v)+ \delta+d_G(v', w_{\tau_0}^{(2)}).
\]
The above two inequalities imply that 
$$
   d_G(v, w_{\tau_0}^{(1)}) + d_G(w_{\tau_0}^{(1)}, v')-{18} \delta
   \leq 
   \delta,
$$
but this is a contradiction since $d_G(v, w_{\tau_0}^{(1)})> 19\delta$.
%
Therefore \eqref{eq:w,gamma} holds and the lemma is proven.
\end{proof}

We will now proceed to show that the probability that $\Fo$ starting from anywhere in $\Aol_{\tau_0}$ gets close enough to $\B_{\tau_0}^{(2)}$ is very low.
Consider the portion of the geodesic $\gamma_{w_{\tau_0}^{(1)},w_{\tau_0}^{(2)}}$ that does not intersect $\B_{\tau_0}^{(1)}$, that is
$
\gamma_{w_{\tau_0}^{(1)},w_{\tau_0}^{(2)}}\setminus  \B_{\tau_0}^{(1)}.
$
Now define
\begin{equation}\label{eq:ins-pzato1}
\text{P}_{\tau_0}^{(1)}:= \B_{\tau_0}^{(2)}\cup \left (\gamma_{w_{\tau_0}^{(1)},w_{\tau_0}^{(2)}}\setminus  \B_{\tau_0}^{(1)}\right ),
\end{equation}
and let $G^{(1)}$ be the sub-graph of $G$ induced by the ``removal'' of $\B_{\tau_0}^{(1)}$.
In other words, all paths in the graph $G^{(1)}$ must avoid the ball $\B_{\tau_0}^{(1)}$, inducing possibly exponentially long detours.

Subsequently we set 
\begin{equation}\label{eq:T_1}
\mathcal{T} _1:= R_1^6; 
\end{equation}
the reason for this choice will become clear later.
Recall the construction of FPPHE from the passage times $\{t_e\}_{e\in E(G)}$ described in Section~\ref{sec:constrfpphe}.
From now on, for every two vertices $u,v\in V(G) $, we fix a geodesic $\gamma_{u,v}$ in $G$ such that if $u,v$ are vertices belonging to the same geodesic segment of $\overline{\gamma}$, then $\gamma_{u,v}$ is chosen as the same geodesic 
appearing in $\overline{\gamma}$. Then, we define
\[
\Tol(u\to v):= \sum_{e\in\gamma_{u,v}} t_{e},
\]
and note that $\Tol(u\to v)$ corresponds to the passage time of a rate-$1$ FPP over the path $\gamma_{u,v}$. (The subscript in $\Tol$ is to emphasize that passage times are of rate $1$.)
If $\bar w_{\tau_0}^{(1)}$ denotes the vertex of $\gamma_{x_{\tau_0},w_{\tau_0}^{(1)}}$ at distance $R_1$ from $w_{\tau_0}^{(1)}$ (that is, $\bar w_{\tau_0}^{(1)}=\gamma_{x_{\tau_0},w_{\tau_0}^{(1)}}\cap \partial B_{\tau_0}^{(1)}$),
then we have that 
\begin{equation}\label{eq:conseq-T-bar}
\begin{split}
\mathcal{T}_1 & \geq R_1\left(\max\{1,\lambda^{-1}\} d_G(x_{\tau_0},\bar w_{\tau_0}^{(1)}) + \lambda^{-1}d_G(\bar w_{\tau_0}^{(1)}, w_{\tau_0}^{(2)}) + \lambda^{-1}(R_1+R_2)\right)\\
    &= R_1\left(\max\{1,\lambda^{-1}\} (R_1^2+{18}\delta) + \lambda^{-1}(R_1^4+R_1^3+R_1^2+ 2R_1+{18}\delta)\right).
\end{split}
\end{equation}
If we disregard the factor of $R_1$ outside the parenthesis, the above bounds have the following meaning. The term $\max\{1,\lambda^{-1}\} d_G(x_{\tau_0},\bar w_{\tau_0}^{(1)})$ is the expected time 
of the slowest between $\Fo$ and $\Fl$ to go from $x_{\tau_0}$ along $\gamma_{x_{\tau_0},w_{\tau_0}^{(1)}}$ until activating the seed located in $\bar w_{\tau_0}^{(1)}$. 
Then, $\lambda^{-1}d_G(\bar w_{\tau_0}^{(1)}, w_{\tau_0}^{(2)})$ is the expected time that $\Fl$ takes to go from $\bar w_{\tau_0}^{(1)}$ along $\overline{\gamma}$ until hitting $w_{\tau_0}^{(2)}$, 
and $\lambda^{-1}R_1$ and $\lambda^{-1}R_2$ are the times that (with high probability) $\Fl$ takes to 
go from $w_{\tau_0}^{(1)}$ to a vertex in the boundary of $\B_{\tau_0}^{(1)}$ and from $w_{\tau_0}^{(2)}$ to a vertex in the boundary of $\B_{\tau_0}^{(2)}$, respectively.
The factor $R_1$ comes into play just to assure that, with high probability, we have 
\begin{align*}
   \mathcal{T}_1 &\geq \max\{1,\lambda^{-1}\}\Tol\left (x_{\tau_0}\to \bar w_{\tau_0}^{(1)} \right )+\lambda^{-1}\Tol\left (\bar w_{\tau_0}^{(1)}\to w_{\tau_0}^{(2)} \right )\\
   &\quad + \lambda^{-1}\max_{x\in \B_{\tau_0}^{(1)}} \Tol(w_{\tau_0}^{(1)} \to  x) + \lambda^{-1}\max_{x\in \B_{\tau_0}^{(2)}}\Tol(w_{\tau_0}^{(2)}\to x).
\end{align*}
We are ready to define an \emph{enlargement} of $\text{P}_{\tau_0}^{(1)}$ as follows:
\[
\text{EP}_{\tau_0}^{(1)}:= \bigcup_{x\in \text{P}_{\tau_0}^{(1)}} B_{G^{(1)}}\left (x, \cout \mathcal{T}_1\right ).
\]
The reason why this is needed, is because of measurability.
In fact, as we will show, in order to understand crucial information about the process we only need to look at the passage times inside the set $\text{EP}_{\tau_0}^{(1)}$.
For the same reason, we need to define the \emph{boundary} of $\text{EP}_{\tau_0}^{(1)}$ in the graph $G^{(1)}$ and not in $G$.
More precisely, we set
\[
\partial_1 \text{EP}_{\tau_0}^{(1)}:=\left \{ x\in \text{EP}_{\tau_0}^{(1)} \ : \ \exists y \in G^{(1)}\setminus \text{EP}_{\tau_0}^{(1)} \text{ such that }\{x,y\}\in E(G^{(1)}) \right \}.
\]
It is fundamental to emphasize that Lemma \ref{lemma:geodesics-in-ball}, together with our choice of $\mathcal{T}_1$ (cf.\ \eqref{eq:T_1}) leads to
\[
d_{G^{(1)}}(\Aol_{\tau_0}, \text{EP}_{\tau_0}^{(1)}) \ \stackrel{\text{Lem.\ }\ref{lemma:geodesics-in-ball},\text{Prop.\ }\ref{prop:gromov-result} }{\geq } \ \delta 2^{R_1-22\delta-2} \ \geq  \ 10 \mathcal{T}_1.
\]
In particular, it follows that $\text{EP}_{\tau_0}^{(1)}$ and $\Aol_{\tau_0}$ are disjoint, so to go from $\Aol_{\tau_0}$ to $\text{P}_{\tau_0}^{(1)}$, $\Fo$ needs to pass through $\partial_1 \text{EP}_{\tau_0}^{(1)}$, 
and to go from $\partial_1 \text{EP}_{\tau_0}^{(1)}$ to $\text{P}_{\tau_0}^{(1)}$ it takes much longer than $\mathcal{T}_1$. 
For a graphical representation refer to Figure  \ref{fig:ins-pz2}.


\begin{figure}[h!]
\begin{center}
\includegraphics[scale=0.4]{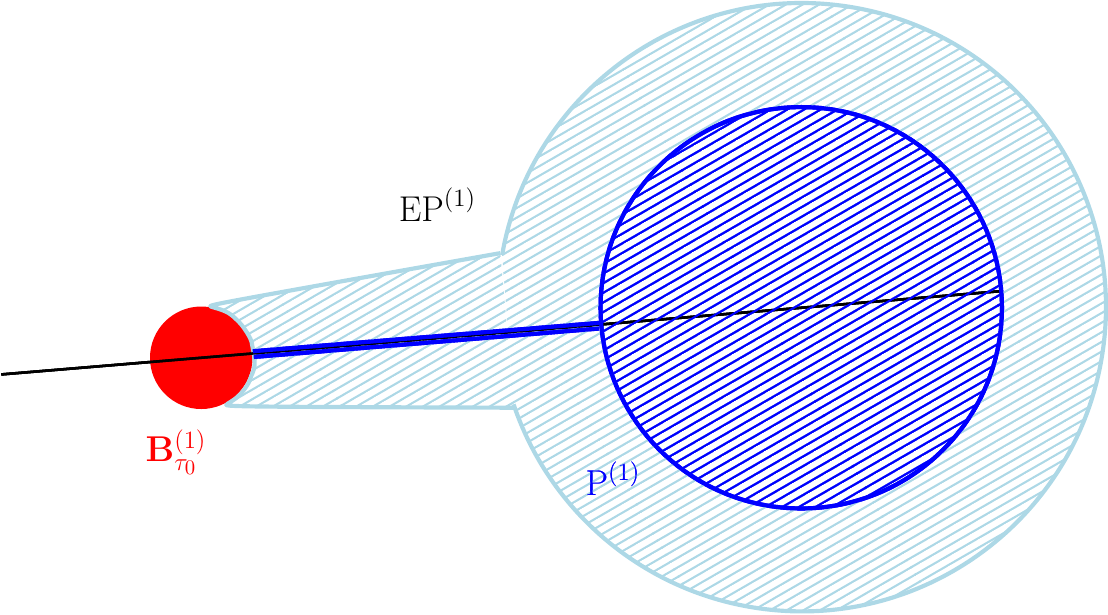}
\caption{The set $\text{P}_{\tau_0}^{(1)}$ is represented in dark blue (dashed ball $\B_{\tau_0}^{(2)} $together with the thick segment), and its enlargement $\text{EP}_{\tau_0}^{(1)}$ is colored light blue (dashed).
The red ball is $\B_{\tau_0}^{(1)}$.}
\label{fig:ins-pz2}
\end{center}
\end{figure}

%
We now let $T^{(1)}(\cdot)$ be the same function $T(\cdot)$ of the passage times of rate $1$ as in~\eqref{eq:def_T(P)} but using only edges from $G^{(1)}$.
We formalize the discussion above by defining the following \emph{good events}:
\begin{equation}\label{eq:event-F12}
\mathcal{F}_1^{(2)}:=\left \{
\begin{split}
   \mathcal{T}_1 &\geq \max\{1,\lambda^{-1}\}\Tol\left (x_{\tau_0}\to \bar w_{\tau_0}^{(1)} \right )+\lambda^{-1}\Tol\left (\bar w_{\tau_0}^{(1)}\to w_{\tau_0}^{(2)} \right )\\
   &\quad + \lambda^{-1}\max_{x\in B_{\tau_0}^{(1)}} \Tol(w_{\tau_0}^{(1)} \to  x) + \lambda^{-1}\max_{x\in B_{\tau_0}^{(2)}}\Tol(w_{\tau_0}^{(2)}\to x)
\end{split}
\right \},
\end{equation}
and
\begin{equation}\label{eq:event-F22}
\mathcal{F}_2^{(2)} := \left\{T^{(1)}\left(\text{P}_{\tau_0}^{(1)} \to \partial_1 \text{EP}_{\tau_0}^{(1)}\right) > \mathcal{T}_1\right\}.
\end{equation}
At this point we are able to define a crucial event:
\begin{equation}\label{eq:1st-attempt-2ball}
\mathcal{E}_{\tau_0}^{(2)}:= \mathcal{F}_1^{(2)}\cap \mathcal{F}_2^{(2)}.
\end{equation}
In words, if the event $\mathcal{E}_{\tau_0}^{(2)}$ occurs, then it implies that $\Fl$ completely occupies $\text{P}_{\tau_0}^{(1)}$ before $\Fo$ can move through $G^{(1)}$ from $\Aol_{\tau_0}$ to $\text{P}_{\tau_0}^{(1)}$. 


\begin{Lemma}\label{lemma:pzato1}
We can set $R_1$ large enough with respect to $G$ so that 
$
\P \bigl[ \mathcal{E}_{\tau_0}^{(2)} \bigr ]\geq \frac{3}{4}.
$
\end{Lemma}
\begin{proof}
For the purpose of this proof, denote by $\pi_1:=\pi_1(\overline{\gamma})$ and $\widehat{\pi}_1:=\widehat{\pi}_1(\overline{\gamma})$ the first and the last vertices along $\overline{\gamma}$ such that 
\[
\pi_1\in \overline{\gamma}\cap \B_{\tau_0}^{(1)}, \quad \text{and} \quad \widehat{\pi}_1\in \overline{\gamma}\cap \B_{\tau_0}^{(1)}.
\]
Note that $\pi_1 = \bar w_{\tau_0}^{(1)}$.
Moreover, set
%
$
\Upsilon_{\tau_0}^{(1)}:=C_{1}+C_{2}+C_3,
$
where we have defined the following quantities
\[
\begin{aligned}
 &C_{1}:=\max\{1,\lambda^{-1}\}\Tol(x_{\tau_0 } \to \pi_1),\quad\quad
 & &C_{2}:= \lambda^{-1} \Tol(\pi_1 \to w_{\tau_0}^{(2)}),\\
 &C_{3}:= \lambda^{-1}\max_{x\in \B_{\tau_0}^{(1)}} \Tol(w_{\tau_0}^{(1)} \to  x),
 & &C_{4}:= \lambda^{-1}\max_{x\in \B_{\tau_0}^{(2)}}\Tol(w_{\tau_0}^{(2)}\to x).
\end{aligned}
\]

The value of $\Upsilon_{\tau_0}^{(1)}$ is an over-estimate on the time needed for FPP started in $\Aol$ to occupy the set 
$
\gamma_{x_{\tau_0 },\pi_1}  \cup  \B_{\tau_0}^{(1)}  \cup  \text{P}_{\tau_0}^{(1)}.
$
Note that the value of $\mathcal{T}_1$ is defined in such a way that
\[
\begin{split}
   \P\left (\mathcal{F}_1^{(2)}\right )
   &= \P \left ( \Upsilon_{\tau_0}^{(1)}\leq \mathcal{T}_1 \right )\\
   &\geq 1 - \P\left(\Tol(x_{\tau_0} \to \pi_1) \geq R_1\,d_G(x_{\tau_0}, \pi_1)\right) 
      - \P\left(\Tol(\pi_1 \to w_{\tau_0}^{(2)}) \geq R_1\,d_G(\pi_1, w_{\tau_0}^{(2)})\right) \\
   &\quad - \Delta^{R_1}\max_{x\in \B_{\tau_0}^{(1)}}\P\left(\Tol(w_{\tau_0}^{(1)}\to x) \geq R_1^2)\right)
    - \Delta^{R_2}\max_{x\in \B_{\tau_0}^{(2)}}\P\left(\Tol(w_{\tau_0}^{(2)}\to x) \geq R_1R_2)\right).
\end{split}
\]
Using the second part of Lemma~\ref{lemma:typical-FPP2}, as well as the fact that $R_1$ is large enough, we obtain 
\[
   \P\left (\mathcal{F}_1^{(2)}\right )
   \geq 1 - \exp\left(-R_1^2/2\right).
\]
Furthermore, since $\Fo$ has to pass through $\partial_1 \text{EP}_{\tau_0}^{(1)}$ before reaching $\text{P}_{\tau_0}^{(1)}$, we can exploit the symmetry of the process to obtain 
\[
\begin{split}
\P & \left ( \exists v\in \partial \text{EP}_{\tau_0}^{(1)} \ : T^{(1)}(v \to \text{P}_{\tau_0}^{(1)})\leq \mathcal{T}_1\right )
= \P \left ( \exists v\in \partial \text{P}_{\tau_0}^{(1)} \ : T^{(1)}(v \to \text{EP}_{\tau_0}^{(1)}) \leq \mathcal{T}_1\right )\\
& \leq \sum_{v\in \partial \text{P}_{\tau_0}^{(1)}}\P\left ( T^{(1)}(v\to \partial_1 \text{EP}_{\tau_0}^{(1)})\leq \mathcal{T}_1\right ) \\
& \leq \left |\partial \text{P}_{\tau_0}^{(1)}\right |e^{-\uc{c:typical-FPP}\mathcal{T}_1} \\
& \leq ( R_1^4+R_1^3+{18}\delta+ \Delta^{R_2})e^{-\uc{c:typical-FPP}\mathcal{T}_1} .
\end{split}
\]
The lemma then follows since $
\P \left ( \mathcal{E}_{\tau_0}^{(2)}\text{ fails}\right ) \leq \P \left ( \mathcal{F}_1^{(2)} \text{ fails}\right ) + \P\left ( \mathcal{F}_2^{(2)} \text{ fails}\right )$, $\mathcal{T}_1=R_1^6$
and $R_1$ is large enough.
\end{proof}
So far we have shown that it is likely that $\B_{\tau_0}^{(2)}$ gets completely occupied by $\Fl$ before $\Fo$ comes even close to it.
In the next subsection we show that we can repeat this reasoning in an inductive way.

\subsubsection{Inductive procedure}

From now on, for all $k\geq 3$ set 
\[
R_k:=R_1^{2(k-1)}.
\]
Consider all the vertices $\left \{w_{\tau_0}^{(k)}\right \}_{k\geq 1}$ from Lemma \ref{lemma:dist-agg}.
Subsequently we set
\[
\B_{\tau_0}^{(k)}:=B_G(w_{\tau_0}^{(k)}, R_k).
\]
Note that 
$$
   d_G(\B_{\tau_0}^{(k-1)}, \B_{\tau_0}^{(k)})
   =R_1^{2k}+R_1^{2k-1}+{18}\delta - R_1^{2k-2}-R_1^{2k-4}
   \leq 2 R^{2k} = 2 R_{k+1}.
$$

Using $\delta$-thinness we obtain the next result, which is a generalization of Lemma \ref{lemma:geodesics-in-ball}.
\begin{Lemma}\label{lemma:geodesics-in-ball2}
For all $k\geq 2$, for all $y\in \partial \Aol_{\tau_0}$ and each vertex $b\in \B_{\tau_0}^{(k)}$, all geodesics $\gamma_{y,b}$ connecting $y$ with $b$ are such that
$
\gamma_{y,b}\cap B_G(w_{\tau_0}^{(k-1)}, 23\, \delta)\neq \emptyset.
$
\end{Lemma}
\begin{proof}
The proof is analogous to that of Lemma \ref{lemma:geodesics-in-ball} and thus we omit the details.
\end{proof}
We will now proceed to show that the probability of FPP started in $\Aol_{\tau_0}$ to get anywhere close to $\B_{\tau_0}^{(k)}$ is decaying exponentially fast in $k$.

We proceed inductively following the main ideas developed in the previous section.
Take the portion of the geodesic segment $\gamma_{w_{\tau_0}^{(k-1)},w_{\tau_0}^{(k)}} \subset \overline{\gamma}$ that does not intersect $\B_{\tau_0}^{(k-1)}$, that is
$
\gamma_{w_{\tau_0}^{(k-1)},w_{\tau_0}^{(k)}}\setminus  \B_{\tau_0}^{(k-1)}.
$
Now define
\begin{equation}\label{eq:ins-pzatok}
\text{P}_{\tau_0}^{(k-1)}:= \B_{\tau_0}^{(k)}\cup \left (\gamma_{w_{\tau_0}^{(k-1)},w_{\tau_0}^{(k)}}\setminus  \B_{\tau_0}^{(k-1)}\right ),
\end{equation}
and let $G^{(k-1)}$ be the sub-graph of $G$ induced by the ``removal'' of $\B_{\tau_0}^{(k-1)}$.
In other words, all paths in the graph $G^{(k-1)}$ must avoid the ball $\B_{\tau_0}^{(k-1)}$, inducing possibly exponentially long detours.
We also define $T^{(k-1)}(\cdot)$ as the function that defines the passage times in the graph $G^{(k-1)}$.
Subsequently we set
\begin{equation}\label{eq:T_k-1}
\mathcal{T}_{k-1}:= R_1^{2k+2}. 
\end{equation}
We are now ready to define an \emph{enlargement} of $\text{P}_{\tau_0}^{(k-1)}$ as follows:
\[
\text{EP}_{\tau_0}^{(k-1)}:= \bigcup\nolimits_{x\in \text{P}_{\tau_0}^{(k-1)}} B_{G^{(k-1)}}\Bigl (x, \cout \sum\nolimits_{j=1}^{k-1} \mathcal{T}_j\Bigr ).
\]
As before, we will define the boundary
\[
\partial_{k-1}\text{EP}_{\tau_0}^{(k-1)} := \left \{ x\in \text{EP}_{\tau_0}^{(k-1)} \ : \ \exists y \in G^{(k-1)}\setminus \text{EP}_{\tau_0}^{(k-1)} \text{ such that } \{x,y\}\in E(G^{(k-1)}) \right \}.
\]
Lemma \ref{lemma:geodesics-in-ball2} together with our choice of $R_1 $ guarantee that, for an appropriate constant $c$, we have 
$
d_{G^{(k-1)}}(\Aol_{\tau_0}, \text{EP}_{\tau_0}^{(k-1)})\geq c 2^{\delta R_{k-1}} \geq 10 \sum_{i=1}^{k-1}T_{i},
$
which in particular ensures that $\text{EP}_{\tau_0}^{(k-1)}$ does not intersect $\Aol_{\tau_0}$.

As we did in the case $k=2$ we need to define some further good events and then we will show a generalization of Lemma \ref{lemma:pzato1}.
For all $k\geq 3$ we set
$\bar w_{\tau_0}^{(k-1)}$ to be the last vertex of $\overline{\gamma}$ inside $\B_{\tau_0}^{(k-1)}$, and 
\begin{equation}\label{eq:event-F1k}
\mathcal{F}_1^{(k)}:=\left \{
   \mathcal{T}_{k-1} \geq \lambda^{-1}\Tol\left (\bar w_{\tau_0}^{(k-1)}\to \bar w_{\tau_0}^{(k)} \right )+ \lambda^{-1}\max_{x\in \B_{\tau_0}^{(k)}}\Tol(w_{\tau_0}^{(k)}\to x).
\right \},
\end{equation}
and
\begin{equation}\label{eq:event-F2k}
\mathcal{F}_2^{(k)} :=
\left\{T^{(k-1)}\left(\text{P}_{\tau_0}^{(k-1)} \to \partial_1 \text{EP}_{\tau_0}^{(k-1)}\right) > \sum_{i=1}^{k-1}\mathcal{T}_i\right\}.
\end{equation}
Thus we set
\begin{equation}\label{eq:1st-attempt-kball}
\mathcal{E}_{\tau_0}^{(k)}:= \mathcal{F}_1^{(k)}\cap \mathcal{F}_2^{(k)}.
\end{equation}
Note that the sum in \eqref{eq:event-F2k} is needed since one cannot guarantee that $\Fo$ has to enter some $\text{EP}_{\tau_0}^{(j)}$, $j<k-1$, before entering $\text{EP}_{\tau_0}^{(k-1)}$. 
\begin{Lemma}\label{lemma:pzatok}
There is a constant $c>0$ such that, for all large enough $R_1$ with respect to $G$, we have 
$
\P \left [ \mathcal{E}_{\tau_0}^{(k)} \right ]\geq 1-\exp\left(-c R_1^{2k+2}\right).
$
\end{Lemma}
\begin{proof}
The proof is very similar to that of Lemma \ref{lemma:pzato1}; we will describe the main steps here.
Again for the purpose of this proof, denote by $\widehat{\pi}_{k-1}:=\widehat{\pi}_{k-1}(\overline{\gamma})$ the last vertices along $\overline{\gamma}$ such that 
$
   \widehat{\pi}_{k-1}\in \overline{\gamma}\cap \B_{\tau_0}^{(k-1)}.
$
Note that $\widehat{\pi}_{k-1} = \bar w_{\tau_0}^{(k-1)}$.
Moreover, set
\[
\Upsilon_{\tau_0}^{(k-1)}:=\lambda^{-1} \Tol(\widehat{\pi}_{k-1} \to w_{\tau_0}^{(k)})  +  \lambda^{-1}\sup_{x\in \partial \B_{\tau_0}^{(k)}}\Tol (w_{\tau_0}^{(k)} \to x).
\]
As before, the value of $\Upsilon_{\tau_0}^{(k-1)}$ is an over-estimate on the time needed for $\Fl$ to go from $ \B_{\tau_0}^{(k-1)}$ to fully occupy the set $ \text{P}_{\tau_0}^{(k-1)}$.
As for the case $k=2$, we observe that $T_{k-1}$ is such that
\[
\P\left (\mathcal{F}_1^{(k)}\right )\geq \P \left ( \Upsilon_{\tau_0}^{(k-1)}\leq \mathcal{T}_{k-1} \right ) \geq 1-e^{-R_1^{2k+2}/4}.
\]
Reasoning as in Lemma \ref{lemma:pzato1} we obtain
\[
\P \left( \mathcal{F}_2^{(k)} \right) \geq 
1 - |\text{P}_{\tau_0}^{(k-1)}|\exp\left(-\uc{c:typical-FPP}\sum_{i=1}^{k-1}\mathcal{T}_i\right).
\]
Since $\sum_{i=1}^{k-1} \mathcal{T}_i=\sum_{i=1}^{k-1} R_1^{2i+4}\geq R_1^{2k+2}$ and $|\text{P}_{\tau_0}^{(k-1)}| \leq 2R_{k+1}+\Delta^{R_k}\leq 2 R_1^{2k}+\Delta^{R_1^{2(k-1)}}$, we have 
\[
\P \left( \mathcal{F}_2^{(k)} \right) \geq 
1 - \exp\left(-\uc{c:typical-FPP}R_1^{2k+2}/2\right).
\]
%
The lemma then follows since $\P \left ( \mathcal{E}_{\tau_0}^{(k)}\text{ fails}\right ) \leq \P \left ( \mathcal{F}_1^{(k)} \text{ fails}\right ) + \P\left ( \mathcal{F}_2^{(k)} \text{ fails}\right )$.
\end{proof}
The end of this section is devoted to showing that the event that all balls $\left \{\B_{\tau_0}^{(k)}\right \}_{k=2}^\infty$ are filled up by $\Fl$ before $\Fo$ can come any close to them is bounded away from zero.

\begin{Theorem}\label{thm:proc}
Let $\mathcal{E}_{\tau_0}^{(k)}$ be defined as \eqref{eq:1st-attempt-kball} for all $k\geq 2$.
Then, taking $R_1$ large enough with respect to $G$, we obtain
$
\P \left [ \bigcap_{k=2}^\infty \mathcal{E}_{\tau_0}^{(k)}\right ]\geq \frac{2}{3}.
$
\end{Theorem}
\begin{proof}
This follows from Lemmas \ref{lemma:pzato1} and \ref{lemma:pzatok}, since 
$
\P \left [ \bigcap_{k=2}^\infty \mathcal{E}_{\tau_0}^{(k)}\right ]\geq 1-\sum_{k\geq 2} \P \left (\mathcal{E}_{\tau_0}^{(k)} \text{ fails}\right )$.
\end{proof}

\subsection{Completing the proof of Theorem~\ref{thm:survival_FPP_lambda}}
\begin{proof}[Proof of Theorem~\ref{thm:survival_FPP_lambda}]
The above argument (summarized in Theorem~\ref{thm:proc}) implies that the procedure succeeds with probability bounded from below by $2/3$, which is not enough for our purposes, as we want to show that it succeeds almost surely.

Suppose that there is a first value $K_0$ for which the event $\mathcal{E}_{\tau_0}^{(K_0)}$ fails to occur.
If this happens, then one of the following two events has occurred:
\begin{itemize}
\item The passage times in $\gamma_{w_{\tau_0}^{(K_0-1)},w_{\tau_0}^{(K_0)}}$ or in $\B_{\tau_0}^{(K_0)}$ are too large (delaying the progress of $\Fl$);
\item The passage times from $\partial_{K_0-1}\text{EP}_{\tau_0}^{(K_0-1)}$ to $\text{P}_{\tau_0}^{(K_0-1)}$ are too small (speeding up the detour of $\Fo$).
\end{itemize}
Now we observe that both these (bad) events are measurable with respect to the passage times inside the set $\cup_{k=2}^{K_0}\text{EP}_{\tau_0}^{(k-1)}$.
Thus, the event
\[
\{K_0 \text{ is the smallest value of }k\text{ for which }\mathcal{E}_{\tau_0}^{(k)}\text{ fails}\}
\]
is measurable with respect to the passage times inside $\text{EP}_{\tau_0}^{(K_0-1)}$.

Subsequently, we inductively define a sequence of stopping times $\tau_1, \tau_2, \ldots$ such that $\tau_0<\tau_1<\ldots$.
At each attempt $j\geq 0$ we consider the procedure described above applied to the aggregate $\Aol_{\tau_j}$.
More precisely, for a given attempt $j$, 
if the $j$-th attempt is successful (for some $j\geq 0$) then $\Fl$ will produce an infinite cluster, and we are done.
Otherwise, let $K_j$ be the first value for which $\mathcal{E}_{\tau_j}^{(K_j)} $ fails.
Inductively we define
\[
\tau_{j+1}=\tau_{j+1}(K_{j}):= \inf\left \{
t> \tau_j \ : \ 
\begin{split}
\bigcup_{s=0}^j \left(\bigcup_{k=2}^{K_j} \text{EP}_{\tau_s}^{(k-1)} \cup \B_{\tau_s}^{(1)}\right) \text{ are fully occupied}\\
\text{ by either $\Fo$ or $\Fl$} \text{ by time }t
\end{split}
\right \}.
\]
Notice that for every $j\geq 0$ the value of $\tau_j$ is defined so that all edges whose passage times have been observed during the first $j-1$ attempts have both endpoints occupied by either $\Fo$ or $\Fl$.
Therefore they will have no further influence on the future development of the process.
Since every attempt will succeed with probability bounded from below by $2/3$, which is independent of $A_{\tau_s}$, eventually the procedure will succeed almost surely.
\end{proof}

\section{Open questions and concluding remarks}\label{sect:questions}

Our Theorem \ref{thm:survival_FPP_1}, and consequently Corollary \ref{corollary:coexistence}, holds when the underlying graph is both hyperbolic and non-amenable.
However, it is easy to find graphs that are non-amenable but not hyperbolic for which our results hold. One example is a \emph{free product of groups}, defined as follows.
Consider $m\geq 2$ finitely generated groups $\Gamma_1, \Gamma_2, \ldots , \Gamma_m$, with identity elements $e_1, e_2, \ldots , e_m$ respectively.
Then the free product $\Gamma:=\Gamma_1 \ast \Gamma_2 \ast \ldots \ast \Gamma_m$ is the set of all words of the form
$
x_1 x_2 \cdots x_n
$,
where $x_1, x_2, \ldots, x_n \in \bigcup_{i=1}^m \Gamma_i\setminus \{e_i\}$.

A more intuitive way to visualize this product is as follows: consider a copy of $\Gamma_1$ and to each vertex $v\in \Gamma_1$ attach a copy of $\Gamma_2, \ldots, \Gamma_m$ by ``gluing'' (i.e., identifying) $e_2, \ldots , e_m$ and $v$ into a single vertex.
Then, inductively, for every vertex on each copy of $\Gamma_i$ attach a copy of $\Gamma_j$, for all $j\neq i$.
This construction gives rise to  a ``\emph{cactus-like}'' structure, which whenever all factors $\{\Gamma_i\}_{i=1}^m$ are finite, turns into a \emph{tree-like} structure.
In this latter case, $\Gamma$ is indeed hyperbolic (for example, a $d$-regular tree is a free product of $d$ copies of the trivial group of two elements). 
Free products have been intensively studied (cf.\ e.g., \cite{CandelleroGilch-Asymptotics,CandelleroGilchMueller-BRWs} and references therein) in relation with different behaviors of random walks and branching random walks on graphs, and it turns out that many interesting results in these works appear when $\Gamma$ is not hyperbolic.

Exploiting the cactus-like structure of $\Gamma$ (both in the hyperbolic and the non-hyperbolic case), one can see that if $\Fo$ survives with positive probability on at least one of the infinite factors, then the $\Fo$ and $\Fl$ will coexist forever.
An example of this phenomenon occurring when $\Gamma$ is not hyperbolic is when one of the free factors is $\Z^d$, with $d\geq 2$, and the initial density of seeds $\mu$ is small enough.
The fact that $\Fo$ can survive with positive probability on $\Z^d$ is shown in \cite{Stauffer-Sidoravicius-MDLA}, thus one considers the evolution of $\Fo$ on the initial copy of $\Z^d$ and that of $\Fl$ on any copy whose origin (which, by definition coincides with the identity element of the group) is occupied by a seed.
The cactus-like structure of $\Gamma$ guarantees that the behavior of the process inside each factor does not interfere with what happens in other factors.
This example shows that hyperbolicity is not a necessary condition for 
Theorem \ref{thm:survival_FPP_1} and Corollary \ref{corollary:coexistence}. 

\begin{Question}
   Is there an analogue of Theorem \ref{thm:survival_FPP_1} and Corollary \ref{corollary:coexistence} for general non-amenable graphs?
\end{Question}

Our Theorem \ref{thm:survival_FPP_lambda} shows that, on hyperbolic graphs, there is no regime of strong survival. 
It is natural to believe that the growth of the graph plays an important role in the survival of $\Fl$. 
\begin{Question}                       
   Is there an analogue of Theorem \ref{thm:survival_FPP_lambda} (that is, survival of $\Fl$ for all $\mu,\lambda$) for general graphs of exponential growth?
\end{Question}

The above is not true on $\mathbb{Z}^d$, as shown in \cite{Stauffer-Sidoravicius-MDLA}. The proof in \cite{Stauffer-Sidoravicius-MDLA} uses the fact that FPP has a shape on $\mathbb{Z}^d$.
It seems reasonable to believe that the same should hold for any graph $G$ whenever FPP in $G$ concentrates around a deterministic shape. This could be the case, for example, in graphs of polynomial growth (excluding trivial cases 
such as $\mathbb{Z}$, where the isoperimetric dimension is $1$ --- see~\cite{CandelleroTeixeira}).
\begin{Question}
   Is there a regime of strong survival in graphs of polynomial growth with isoperimetric dimension bigger than $1$? 
\end{Question}

Finally, a fascinating open question from \cite{Stauffer-Sidoravicius-MDLA} is whether there exists coexistence in $\Z^d$, where even the case $d=2$ is open. 
\begin{Question}
   Is there coexistence on $\mathbb{Z}^d$, for $d\geq 2$?
\end{Question}

\section*{Acknowledgements}
This work started when E. Candellero was affiliated to the University of Warwick.
E. Candellero acknowledges support from the project ``Programma per Giovani Ricercatori Rita Levi Montalcini'' awarded by the Italian Ministry of Education and support by ``INdAM -- GNAMPA Project 2019''.
A. Stauffer acknowledges support from an EPSRC Early Career Fellowship.

\section*{Appendix}
\appendix
\section{Proof of Proposition \ref{prop:detour_cylinders}}\label{sect:proof_prop_detour}
The aim of this section is to provide a proof of Proposition \ref{prop:detour_cylinders}.
We are given two vertices $x,y$ at distance $d\geq  50 \delta$ (cf.\ Proposition~\ref{prop:detour_cylinders}), and the constraint that the path under consideration must go from $x$ to $y$ avoiding the cylinder $\cyl^{(L)}_{u,v} $, where $u$ and $v$ are as in the statement. 

Start by observing that the vertices $u$ and $v$ belong to a geodesic segment from $x$ to $y$, hence the segment of $\gamma_{x,y}$ joining $u$ and $v$ is a geodesic segment from $u$ to $v$, which we denote $\gamma_{u,v}$ (geodesic segments are piece-wise geodesics).

The classical result by Gromov (cf.\ Proposition \ref{prop:gromov-result}, or \cite[Section 7]{Gromov}) states that if on $G$ all triangles are $\delta$-thin, then the length of a path between two vertices $x$ and $y$ that avoids a ball of radius $r$ centered at a point of the geodesic joining $x$ and $y$ has length at least $\delta 2^{r/\delta}$.
We exploit this fact, using that $\cyl^{(L)}_{u,v} $ is defined as a union of balls centered at a $d_G$-geodesic.
Each path that avoids the cylinder must avoid many balls (a number linear in the length $d$ of the cylinder), hence by Gromov's result it will have length of order at least $\delta 2^{r/\delta}\cdot d$.

Now we proceed with a formal proof.
Let $P $ denote any path that goes from $x$ to $y $ that avoids the cylinder $\cyl^{(L)}_{u,v}$, and let $n$ denote its $d_G$-length.
In particular, it will be convenient to express $P$ as a sequence of vertices such as 
\[
P=(P_0, P_1, \ldots, P_n), \quad \text{ with }P_0=x \text{ and }P_n=y.
\]
We will find a lower bound on $n$.
A first consideration is that since $P$ avoids $\cyl^{(L)}_{u,v}$, it must avoid the first ball of $\cyl^{(L)}_{u,v}$, meaning
$
P\cap B_G(u, L)=\emptyset.
$
For every vertex $P_\ell$ of the path $P$, let $\Gamma_{P_\ell, y}$ denote the set of geodesics starting at $P_\ell$ and ending at $y$.

\paragraph{Step 1.}
Consider the first (i.e.\ the smallest) index $\ell$ for which $P_\ell$ has a geodesic to $y$ that does \emph{not} intersect the ball $B_G(u, L)$, and set $\w_1:=P_\ell$.
In formulas:
\[
\w_1:=\min_{1\leq \ell\leq n}\{P_\ell \ : \ \exists \gamma \in \Gamma_{P_\ell, y} \text{ s.t.\ }\gamma\cap B_G(u, L)=\emptyset\}.
\]
Note that since $P$ avoids the cylinder of radius $L$, we have that $d_G(\w_1, \gamma_{u,v})\geq L$.
Having found $\w_1$, take the previous vertex in $P$ (which in our previous notation corresponded to $P_{\ell-1}$), and denote it by $\ow_1$.
For a graphic representation, see Figure \ref{fig:prop-2}.


\begin{figure}[h!]
\begin{center}
\includegraphics[scale=0.9]{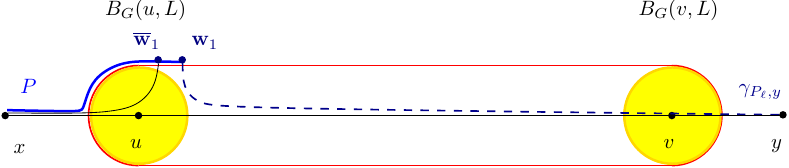}
\caption{The construction described above: the yellow balls are $ B_G(u,L)$ and $ B_G(v,L)$.
The path $P$ is drawn in blue, and $\w_1$ is the first vertex that has a  geodesic to $y$ that does not touch $ B_G(u,L)$ (the dashed line).
The thin black line is a geodesic connecting $x$ to $\ow_1$.}
\label{fig:prop-2}
\end{center}
\end{figure}

The next result will be important to show that the length of the path $P$ between $x$ and $\w_1$ is very large.
\begin{Claim}\label{claim:x-w1}
For all geodesics $\gamma_{x, \w_1}$ between $x$ and $\w_1$  we have $ \gamma_{x, \w_1}\cap B_G(u, \delta)\neq \emptyset$.
\end{Claim}
\begin{proof}
Let $\gamma_{\w_1,y}$ denote a geodesic such that $\gamma_{\w_1,y}\cap B_G(u, L)=\emptyset$. 
Then, clearly
\begin{equation}\label{eq:aux-w1}
d_G(u,\gamma_{\w_1,y})\geq L\geq 9\delta>\delta.
\end{equation}
Moreover, since the triangle $\{x, \w_1, y\}$ is $\delta$-thin, every vertex of $\gamma_{x,y}$ has to be contained in the set $\cyl^{(\delta)}_{x,\w_1}\cup \cyl^{(\delta)}_{\w_1, y}$.
Since $u \in \gamma_{x,y}$, then \eqref{eq:aux-w1} implies that $u\in \cyl^{(\delta)}_{x,\w_1}$, which is the claim.
\end{proof}
Now choose any geodesic $\gamma_{x, \w_1}$, by Claim \ref{claim:x-w1} we know that there is at least a vertex on $\gamma_{x, \w_1}$ at distance at most $\delta$ from $u$.
Consider a ball of radius $L-\delta$ centered at any such vertex, and call it $B_1$.
Since $B_1\subset B_G(u,L)$ we deduce that $P$ when going from $x$ to $\w_1$ avoids $B_1$ which is centered at a vertex of the geodesic $\gamma_{x, \w_1}$.
By Proposition \ref{prop:gromov-result} we have that 
\begin{equation}\label{eq:d_2_delta_1}
d_P(x,\w_1)\geq \delta 2^{(L-\delta)/\delta}
,
\end{equation}
where we have set
\begin{equation}\label{eq:d_P}
d_P(a,b):=\text{number of edges that $P$ crosses on its way from vertex $a$ to vertex } b.
\end{equation}

\paragraph{Step 2.}
Fix a vertex $\omega_1\in \gamma_{u,v}$ such that $d_G(u,\omega_1)= L+1+ 2\delta $ (i.e., at distance $ d_G(x,u)+L+1+ 2\delta $ from $x$).
The next claim will be used to show that every geodesic $\gamma_{\w_1, y}$ passes at distance at most $2\delta$ from $\omega_1$.
For a graphical representation see Figure \ref{fig:prop-3}.

\begin{figure}[h!]
\begin{center}
\includegraphics[scale=0.9]{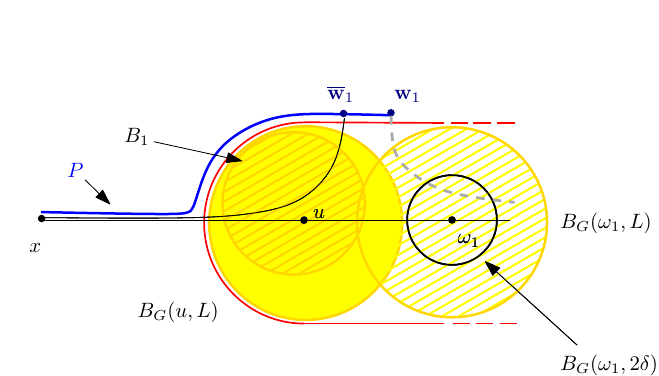}
\caption{A closeup of Figure \ref{fig:prop-2}, with a representation of $B_1$, of $B_G(\omega_1,2\delta)$ and $B_G(\omega_1,L)$.}
\label{fig:prop-3}
\end{center}
\end{figure}

\begin{Claim}\label{claim:omega_1}
For all vertices $b\in B_G(u, L)$ and all geodesics $\gamma_{b,y}$ we have $\gamma_{b,y}\cap B_G(\omega_1, \delta)\neq \emptyset $.
\end{Claim}
\begin{proof}
The proof of this fact follows again from $\delta$-thinness. 
More precisely, if we consider the triangle $\{b,u,y\}$, then for any geodesic $\gamma_{b,y}$ we have 
$
\gamma_{b,y}\subset \cyl^{(\delta)}_{b,u}\cup \cyl^{(\delta)}_{u,y}.
$
Since $b\in B_G(u, L)$, then
$
d_G(\gamma_{b,u}, \omega_1)\geq d_G(\omega_1, B_G(u,L))=2\delta+1>\delta.
$
Thus we must have that $d_G(\omega_1, \gamma_{b,y})\leq \delta$, which is the claim.
\end{proof}
By definition of $\ow_1$ we have that all geodesics $\gamma_{\ow_1, y}\cap B_G(u, L)\neq \emptyset$.
Therefore, for every vertex in this intersection we can apply Claim \ref{claim:omega_1}, obtaining that for all geodesics $\gamma_{\ow_1, y}$
\[
\gamma_{\ow_1, y}\cap B_G(\omega_1, \delta)\neq \emptyset.
\]
Now we observe that, as $d_G(\w_1, \ow_1)=1$ and as we assume that $\delta>0$, for every pair of geodesics $\gamma_{\w_1,y}$ and $\gamma_{\ow_1,y}$, we have
$
\sup_{z \in \gamma_{\w_1,y}}d_G(z, \gamma_{\ow_1,y})\leq \max\{1, \delta\}=\delta.
$
Thus, for all geodesics $\gamma_{\w_1,y}$ we have
\begin{equation}\label{eq:w1-omega1}
\gamma_{\w_1,y}\cap B_G(\omega_1, 2\delta)\neq \emptyset.
\end{equation}

\paragraph{Step 3.}
Now consider the ball $B_G(\omega_1, L) $, this will have a similar role to the previous $B_G(u, L) $.
In order to continue, we define a vertex $\w_2$ (which will have a similar role to $\w_1$ in the previous step).
On the path $P$ find the first vertex $\w_2$ (after $\w_1$ in $P$) that has a geodesic to $y$ that does not cross $B_G(\omega_1, L)$.
In formulas:
\[
\w_2:=\min_{1\leq \ell\leq n}\{P_\ell> \w_1 \ : \ \exists \gamma \in \Gamma_{P_\ell, y} \text{ such that }\gamma\cap B_G(\omega_1, L)=\emptyset\},
\]
where $P_\ell> \w_1$ means that $P_\ell$ comes after $\w_1$ in the path $P$.
Moreover, let $\ow_2$ denote its predecessor on the path $P$.
Note that from the respective definitions it follows that the vertices $\ow_1, \w_1, \w_2$ are all distinct. 

At this point, a proof completely analogous to that of Claim \ref{claim:x-w1} (replacing $u$ with $\omega_1$, $x$ with $\w_1$ and $\w_1$ with $\w_2$) shows that for all geodesics $\gamma_{\w_1,\w_2}\in \Gamma_{\w_1,\w_2}$ we have
$
\gamma_{\w_1,\w_2}\cap B_G(\omega_1 , 4 \delta)\neq \emptyset.
$
The factor $4\delta$ comes from the following facts. 
The distance between $\omega_1$ and any geodesic $\gamma_{\w_1,y}$ is at most $2\delta$, as stated in \eqref{eq:w1-omega1}.
Now let $\omega_1'$ denote any vertex that belongs to $\gamma_{\w_1,y} \cap B_G(\omega_1 , 2 \delta)$.
The proof of Claim \ref{claim:x-w1} together with the subsequent reasoning shows that $d_G(\omega_1', \gamma_{\w_1,\w_2})\leq 2\delta$.
Thus,
\[
d_G(\omega_1,\gamma_{\w_1,\w_2})\leq d_G(\omega_1,\omega_1')+d_G(\omega_1',\gamma_{\w_1,\w_2})\leq 4\delta.
\]
Subsequently, just as we did in Step 1 (cf.\ \eqref{eq:d_2_delta_1}) we can deduce that there is a ball $B_2$ of radius $L-4\delta$ and centered at some point of any geodesic $\gamma_{\w_1, \w_2}$  such that the portion of $P$ joining $\w_1$ to $\w_2$ avoids $B_2$.
Proposition \ref{prop:gromov-result} implies
\begin{equation}\label{eq:d_2_delta_2}
d_P(\w_1, \w_2)\geq \delta 2^{(L-4\delta)/\delta},
\end{equation}
which by construction yields to
\[
d_P(x, \w_2)\stackrel{ \eqref{eq:d_2_delta_1}, \eqref{eq:d_2_delta_2} }{=} d_P(x, \w_1)+d_P(\w_1, \w_2)\geq 2\delta 2^{(L-4\delta)/\delta}.
\]
As we did at the beginning of Step 2, fix a vertex $\omega_2\in \gamma_{u,v}$ such that $d_G(\omega_1,\omega_2)= L+1+ 2\delta $.
Then, with a similar proof to that of Claim \ref{claim:omega_1} we can show that for all vertices $b\in B_G(\omega_1, L)$ and all geodesics $\gamma_{b,y}$ we have
$
\gamma_{b,y}\cap B_G(\omega_2, \delta)\neq \emptyset.
$
As in Step 2, this yields that for all geodesics $\gamma_{\w_2,y}$
$
\gamma_{\w_2,y}\cap B_G(\omega_2, 2\delta) \neq \emptyset.
$

\paragraph{Step 4.}
At this point we can start an inductive procedure to find a lower bound on the length of $P$, in particular it suffices to repeat Step 3 until we get close to the end of the cylinder.
More precisely, the $k$-th time that we start over with Step 3 we consider the ball $B_G(\omega_{k-2}, L)$ replaced with $B_G(\omega_{k-1}, L)$, define a vertex $\w_k$ on the path $P$ such that
\[
\w_k:=\min_{1\leq \ell\leq n}\{P_\ell> \w_{k-1} \ : \ \exists \gamma \in \Gamma_{P_\ell, y} \text{ such that }\gamma\cap B_G(\omega_{k-1}, L)=\emptyset\},
\]
and let $\ow_k$ denote its predecessor on the path $P$.

At this point, a proof completely analogous to that of Claim \ref{claim:x-w1} (inductively replacing $u$ with $\omega_{k-2}$, $x$ with $\w_{k-1}$ and $\w_1$ with $\w_k$) shows that for all geodesics $\gamma_{\w_{k-1},\w_k}$ we have
$
\gamma_{\w_{k-1},\w_k}\cap B_G(\omega_{k-1} , 4 \delta)\neq \emptyset.
$
We note that the factor $4\delta$ comes from the same reasoning as in Step 3.
From this we deduce that
$
d_P(\w_{k-1}, \w_k)\geq \delta 2^{(L-4\delta)/\delta}.
$
In conclusion we obtain
\begin{equation}\label{eq:d_2_delta_sum}
d_P(x, \w_k)\geq d_P(x, \w_1)+\ldots + d_P(\w_{k-1}, \w_k)\geq k\delta 2^{(L-4\delta)/\delta}\stackrel{L\geq 9\delta }{\geq }k\delta 2^{L/(2\delta)}.
\end{equation}
Then we define $\omega_k\in \gamma_{u,v}$ such that $d_G(\omega_{k-1},\omega_k)= L+1+ 2\delta $.
Then we can show for all geodesics $\gamma_{\w_k,y}$ we have 
$
\gamma_{\w_k,y}\cap B_G(\omega_k, 2\delta) \neq \emptyset,
$
and proceed inductively.

We have good control on the positions of the vertices $\{\w_j\}_j$, and therefore we can apply these iterations safely until we get to balls towards the end of the cylinder (i.e., close to $v$).
To be on the safe side, we can perform this reasoning for almost the whole length of the cylinder, just ignoring the last few balls.
We continue our iterations until the last ball reaches distance $2L$ from the end of the cylinder, which means distance $2L+d_G(v,y)$ from vertex $y$.

The total number $K$ of such iterations is the number of balls that $P$ avoids and are used to define the sequence of vertices $\{\w_j\}_{j=1}^K$.
Thus, a lower bound on $K$ is given by
\begin{equation}\label{eq:nr_avoided_balls}
K\geq 
\left \lfloor \frac{d_G(u,v)-2L}{L+1+2\delta} \right \rfloor  \geq \frac{d_G(u,v)}{3 L}.
\end{equation}
Finally, by putting together \eqref{eq:d_2_delta_sum} and \eqref{eq:nr_avoided_balls}, we obtain the proposition.

\end{document}